\documentclass[11pt, leqno]{article}

\usepackage{amsmath,amsthm,amsfonts,amssymb,tabularx}
\usepackage[latin1]{inputenc}
\usepackage{graphicx}
\usepackage{color}
\usepackage{nameref}
\usepackage{lineno}
\usepackage{caption}
\usepackage{subcaption}

\usepackage[inner=3cm,outer=4.5cm,bottom=4cm,top=4cm]{geometry}

\usepackage{hyperref}
\usepackage{pdfpages}

\hypersetup{urlcolor=blue, colorlinks=true}

\newtheorem{theorem}{Theorem}[section]
\newtheorem{proposition}[theorem]{Proposition}

\newtheorem{defn}[theorem]{Definition}
\newtheorem{lemma}[theorem]{Lemma}

\newtheorem{remark}{Remark}
\newcommand\R{\mathbb{R}}

\newcommand\intr{\int_{\mathbb{R}^N}}

\newcommand\SqLop{ \left( (-\Delta)^{-1}\mathcal{L}^{1-s}_\epsilon \right)^{\frac{1}{2}} }

\numberwithin{equation}{section}

\newcommand{\RN}{\mathbb{R}^N}

\begin{document}

%\setpagewiselinenumbers

%\linenumbers

\title{\bf Existence of weak solutions for a general porous  \\ [8pt]
medium equation with nonlocal pressure \\ { \ } }

\author{\bf Diana Stan, F\'elix del Teso  and Juan Luis V\'azquez}

\maketitle

\begin{abstract}
We study  the general  nonlinear diffusion equation $u_t=\nabla\cdot (u^{m-1}\nabla (-\Delta)^{-s}u)$ that describes a flow through a porous medium which is driven by a nonlocal pressure. We consider  constant parameters  $m>1$ and $0<s<1$,  we assume that  the solutions are non-negative  and the problem is posed in the whole space. In this paper we prove existence of weak solutions  for all integrable initial data $u_0 \ge 0$ and for all exponents $m>1$ by  developing a new approximation method that allows to treat the range $m\ge 3$, that could not be covered by previous works.  We also extend the class of initial data to include any  non-negative measure $\mu$ with finite mass.  In passing from bounded initial data to measure  data  we make strong use of an $L^1$-$L^\infty$ smoothing effect and other functional estimates.  Finite speed of propagation is established for all $m\ge 2$, and this property implies the existence of free boundaries. The authors had already proved that finite propagation does not hold for $m<2$.
\end{abstract}

\vspace{1.3cm}

\vspace{0.3cm}


\small
\vspace{1.5cm}

\noindent \textbf{Addresses:}\\
Diana Stan, {\tt diana.stan@unican.es}, Departamento de Matem\'{a}ticas, Estad\'istica y Computaci\'on, Universidad de Cantabria, Av. de los Castros, 39005 Santander, Cantabria, Spain. \\
 F\'{e}lix del Teso, {\tt fdelteso@bcamath.org}, Basque Center for Applied Mathematics, Alameda de Mazarredo 14, 48009 Basque-Country, Spain. \\
  Juan Luis V{\'a}zquez, {\tt juanluis.vazquez@uam.es}, Departamento de Matem\'{a}ticas, Universidad Aut\'{o}noma de Madrid, Campus de Cantoblanco, 28049 Madrid, Spain.

\

\newpage
\small
\tableofcontents

\normalsize

\section{Introduction}

In this paper we study the following  evolution equation of diffusive type with nonlocal effects
\begin{equation}\label{model1}
   \left\{ \begin{array}{ll}
  \partial_t u = \nabla \cdot (u^{m-1} \nabla (-\Delta)^{-s}u)    &\text{for } x \in \RN, \, t>0,\\[2mm]
  u(0,x)  =u_0(x) &\text{for } x \in \RN,
    \end{array}  \right.
\end{equation}
for $u=u(x,t)$, exponents $m > 1$, $0<s<1$, and space dimension $N\geq 1$. We will only consider nonnegative data and solutions $u_0, u\ge 0$ on physical grounds. The problem will be posed in the whole space, with $x\in\RN$ and $t>0$. Here $ (-\Delta)^{-s}$ denotes the inverse of the fractional Laplacian operator as defined in \cite{Stein70}.

Our aim is to construct weak solutions for all   nonnegative   initial data $u_0 \in L^1(\RN)$ and for all the stated range of parameters. Model \eqref{model1}   reduces to   the Porous Medium Equation   $\partial_t u = \nabla \cdot (u^{m-1} \nabla u)$ when $s=0$,  \cite{VazBook}, but here we allow for a new dependence via  the inverse fractional Laplacian operator, $\partial_t u = \nabla \cdot (u^{m-1} \nabla p)$ with $p=(-\Delta)^{-s}u$, which accounts for nonlocal effects in the diffusive process.   For convenience we   will call this intermediate variable $p$ the pressure, though it is not in agreement with the usual PME convention unless $m=2$.

  Model \eqref{model1}   was studied for $m=2$ by Caffarelli and V\'azquez starting with \cite{CaffVaz,CaffVaz2}, followed by \cite{CaffSorVaz,CaffVazReg,CarrilloHuangVazquez}.   In these papers existence of weak solutions, finite speed of propagation, local H\"older regularity, and asymptotic behaviour were established for the particular model. This model and ours are  particular cases   of the general equations proposed in  \cite{GiacLeb1, GiacLeb2} in statistical physics, that take the  form $u_t=\nabla\cdot(\sigma(u)\nabla {\mathcal{L}}(u))$. There is also a physical motivation in the theory of dislocations proposed by Head, that has been investigated by Biler, Karch and Monneau \cite{BilerKarchMonneau} for $m=2$ in one space dimension. However, the extension of the dislocation model to several dimensions leads to a more complicated system that falls outside of the present investigation.  Finally,  we point out that the gradient flow structure for \eqref{model1} with $m=2$ has been recently developed in \cite{Lisini} using Wasserstein metrics in the style of \cite{AGS2008}.
Uniqueness of   suitable solutions   is still an open problem for all these models in several space dimensions, but it holds for $N=1$ according to \cite{BilerKarchMonneau}.   See more on this issue in Section \ref{sec:comments}.

Existence of   a class of weak solutions for $m \in(1,3)$, obtained as limits of approximations,  was proved by the   present authors   in \cite{StanTesoVazCRAS,StanTesoVazJDE} under some extra decay conditions on the initial data. In that paper we employed a rather standard regularization of the singular operator by considering   a suitable smooth kernel  $K_\epsilon $ such that   $K_\epsilon \star u \to |x|^{-(N-2s)} \star u = (-\Delta)^{-s} u$. Energy estimates allowed us to obtain compactness, but only in the   stated range   of $m$. New methods seemed to be needed   to tackle    the more degenerate case $m\ge 3$;    it is the purpose of the present paper to address and solve that problem.   A further discussion on this issue can be found in Section \ref{sec:comments}. The main step we take here in order to prove existence of weak solutions of \eqref{model1} is a novel approximation method. It consists in interpreting model \eqref{model1} in  the form
$$
u_t= \nabla \cdot(u^{m-1} \nabla (-\Delta)^{-1} \mathcal{L}  u).
$$
Then, we approximate the operator  $\mathcal{L}= (-\Delta)^{1-s}$  by
\[
\mathcal{L}^{1-s}_\epsilon [u](x)=C_{N,1-s}\int_{\RN}\frac{u(x)-u(y)}{\left(|x-y|^2+\epsilon^2\right)^{\frac{N+2-2s}{2}}}dy.
\]
  This approach to model \eqref{model1} allows us to prove some needed  $L^p$-estimates, that are an essential tool in order to derive convergence of the solutions of the approximating problems.

We start by assuming initial data $u_0 \in L^1(\RN)\cap L^\infty(\RN)$, $u_0 \ge 0$, and we prove existence of a class of weak solutions   constructed   via an approximating method   that uses the preceding observation and proceeds via several approximation steps.   The paper combines a great variety of compactness techniques and the detailed proofs show how the available energy estimates can be used step by step as we pass to the limit in the approximating models. The main difficulties  of the construction are: the nonlocal and nonlinear character of the equation, absence of comparison principle, absence of explicit self-similar solutions (except very particular cases, c.f \cite{StanTesoVazTrans}).

A second contribution of the paper is the generality of the initial data. We  may take $u_0 =\mu \in \mathcal{M}^+(\R^N)\,$, the space of nonnegative Radon measures on $\RN$ with finite mass. This  covers in particular the case of merely integrable data $u_0\in L^1(\RN)$. We cover that issue in Section  \ref{SectionExistL1} where we obtain existence of  weak solutions for the whole range $1<m<\infty$, generalizing the results of \cite{CaffVaz} and \cite{StanTesoVazJDE}, where the cases $m=2$ and $m\in (1,3)$ were covered respectively. This rounds up the existence theory.

Another positive property of this approach is that it can be successfully generalized to more general equations of the form
\[
u_{t}(x,t) = \nabla \cdot (G'(u) \nabla (-\Delta)^{-s}u),
\]
where $G: [0,+\infty)\to [0,+\infty)$ is a regular function with at most linear growth at the origin.

A remarkable property of  many diffusive PDE's of degenerate type  is  finite speed of propagation,   which means that the support of the solutions may spread but only with finite speed.   When we combine degenerate nonlinearities (powers with $m>1$) and nonlocal effects it is not clear whether finite propagation will hold or not. The property  was first  observed by Caffarelli and V\'azquez in \cite{CaffVaz} for the model with $m=2$, see also \cite{BilerKarchMonneau} for $N=1$. In \cite{StanTesoVazJDE} we discovered that the nonlinearity has a strong influence on the speed of propagation property of solutions independently of $s\in (0,1)$. Indeed, we proved two different   types of behaviour   depending on the exponent $m$: finite speed of propagation for $m\in (2,3)$ and infinite speed of propagation for $m\in (1,2)$. A numerical simulation using \cite{delTesoCalcolo} pointed us to this change in the  positivity property of the solution. We establish here the property of finite propagation for all $m\ge 2$. See Figure \ref{figura1}.
  Paper \cite{StanTesoVazSpringer} by the present authors contains a survey of results on this equation and its motivations, including   the main results of the present paper. Moreover,  as a further contribution the asymptotic behaviour of solutions with integrable data is established in $N=1$. The problem is still open in several dimensions.

Let us comment on some   closely related literature. Indeed, another possible extension   of the model  studied by Caffarelli and V\'azquez in \cite{CaffVaz} for $m=2$ has been considered in \cite{BilerImbertKarch, BilerKarchMonneau, Imb16}. They assume that $p=(-\Delta)^{-s}u^{m-1}$   and the resulting equation is
\[
\partial_t u = \nabla \cdot (u\nabla (-\Delta)^{-s}u^{m-1}).
\]
   In that case there exists a weak solution with finite speed of propagation for the range $m>1$. Moreover, they find explicit Barenblatt self-similar profiles\footnote{We note for comparison reasons that in their notation $\alpha= 2(1-s)$.}.    It is also proved that finite propagation holds for all $m>1$, which implies a strong qualitative difference with our model  \eqref{model1} where finite propagation happens only for $m>2$.    We can also consider  models  including nonlinearities on both terms like
\[
\partial_t u = \nabla \cdot (u^{m}\nabla (-\Delta)^{-s}u^{n}).
\]
They are interesting for comparison purposes. Work on this last model is naturally more incomplete, we refer to \cite{StanTesoVazTrans, DoZh17}.

We  finally recall that there is another model of nonlocal porous medium equation:
\begin{equation}\label{FPME}
v_t + (-\Delta)^{s}( v^{m} )=0
\end{equation}
with $m >0$ and ${s} \in (0,1)$ for which the theory has been quite developed in \cite{PQRV1,PQRV2,BV2014,VazquezBarenblattFPME,BV2014,BFV18}, see also the survey paper \cite{Vazquez2014}.
 Infinite propagation holds for this model even if $m>1$.  A very interesting result is the connection between model \eqref{model1} and model \eqref{FPME}: we have found  in \cite{StanTesoVazTrans}  an exact transformation formula between self-similar solutions of the two models,   \eqref{FPME} and \eqref{model1},   but it only applies to the  range $m<2$ of our present model.   We finally refer to \cite{VazCIME} or a general presentation of the state of the art in nonlinear diffusion including linear and nonlinear models with local and nonlocal operators.

%%%%%%%%%%%%%%%%%%%%%%%%%%%%%%%%%%%%%%%%%%%%%%%%%%%%%%%%%%%%%%%%%%%%

\section{Precise statement of the main results}

We recall that all data and solutions are nonnegative and we will stress this fact when convenient. In this section will only present the results for integrable and bounded initial data since establishing the existence and main properties in this case contains the main difficulties. For  clarity of exposition, we delay to Section \ref{SectionExistL1} the case of measure data since it is an independent contribution of the paper.

\begin{defn}\label{def1}
Let $u_0\in L^1_{\textup{loc}}(\R^N)$ and nonnegative. We say that $u\geq0 $ is a weak solution of Problem \eqref{model1} if: \\ (i) $u \in L^1_{\textup{loc}}(\RN \times (0,T))$ , (ii) $\nabla (-\Delta)^{-s}u \in L^1_{\textup{loc}}(\RN \times (0,T))$, (iii) $u^{m-1} \nabla (-\Delta)^{-s} u\in L^1_{\textup{loc}}(\RN \times (0,T))$ and
\begin{equation*}
\int_0^T\int_{\RN} u \phi_t\,dxdt-\int_0^T\int_{\RN}  u^{m-1} \nabla (-\Delta)^{-s} u \cdot\nabla \phi \,dxdt+ \int_{\RN} u_0(x) \phi(x,0)dx=0
\end{equation*}
  for all test functions  $\phi \in C^1_c(\RN \times [0,T))$.
\end{defn}

We state our main results on the existence and qualitative properties of solutions.

\begin{theorem}\label{Thm1PMFP}
Let  $1<m <\infty$, $N\ge 1$, and let $u_0\in L^1(\RN)\cap L^\infty (\RN)$ and nonnegative.
Then there exists a weak solution $u\geq0$ of Problem \eqref{model1}  such that $u\in L^1(\RN \times (0,T))$, $u\in L^\infty (\RN \times (0,T))$, and $ (-\Delta)^{\frac{1-s}{2}} u^{r} \in L^2(\RN \times (0,T))$ for all $r>m/2 $. Moreover, $u$ has the following properties:
\begin{enumerate}
\item \textbf{(Conservation of mass)} For all $0 < t < T$ we have
$\displaystyle{
\int_{\RN}u(x,t)dx=\int_{\RN}u_0(x)dx.
}$
\item \textbf{($L^{\infty}$ estimate) } For all $0 < t < T$ we have  $||u(\cdot,t)||_\infty\leq ||u_0||_\infty$.

\item \textbf{($L^p$ energy estimate)} For all $1<p<\infty$ and $0 < t < T$ we have
\begin{equation}\label{Lpdecay} \intr u^p(x,t)dx + \frac{4p(p-1)}{(m+p-1)^2}\int_0^t \int_{\RN}\left|(-\Delta)^{\frac{1-s}{2}} u^{\frac{m+p-1}{2}}\right|^2 dxdt \le \intr u_0^p(x)dx .
\end{equation}

\item \textbf{(Second energy estimate)} For all $0 < t < T$  we have
\begin{equation}\label{SecondEnergy}
  \frac{1}{2} \intr \left|(-\Delta)^{-\frac{s}{2}} u(t)\right|^2 dx  +
\int_0^t \intr  u^{m-1} \left| \nabla  (-\Delta)^{-s} u(t)\right|^2 dx  dt \le \frac{1}{2} \intr \left|(-\Delta)^{-\frac{s}{2}} u_0\right|^2 dx.
\end{equation}

\end{enumerate}
\end{theorem}

\medskip

\begin{remark}{\rm  (a) The a priori estimates 1, 2, 3 and 4 for Problem \eqref{model1} can be derived in a formal way as in  \cite[Section 3]{StanTesoVazJDE}. A rigorous proof for 1, 2 and 4 when $m\in(1,3)$ can be found in that paper. The approximation used there does not allow to cover the whole range $m\in (1, +\infty)$ because of the lack of an $L^p$ type energy estimate like \eqref{Lpdecay}. However, 1 and 2 follow as in \cite{StanTesoVazJDE} and therefore they will not be discussed   in detail   here.}

\noindent {\rm  (b) We would like to note that estimates \eqref{Lpdecay} and \eqref{SecondEnergy} do not present any special form or extra difficulty when $m=2$, $m=3$ or $m>3$, as it happened with the First Energy Estimate \eqref{FirstEnergy} used in \cite{StanTesoVazJDE} and \cite{CaffVaz}. See Section \ref{sec:comments} for a more detailed discussion about this fact.}
\end{remark}

\begin{theorem}[\textbf{Smoothing effect}] \label{ThmSmoothing}
Let $u\geq0$ be a weak solution of Problem \eqref{model1} with  nonnegative   initial data $u_0 \in L^1(\RN) \cap L^{\infty}(\RN)$  as constructed in Theorem \ref{Thm1PMFP}.  Then,
\begin{equation}\label{smoothform}
\displaystyle \| u(\cdot,t)\|_{L^{\infty}(\RN)} \le C_{N,s,m,p} \, t^{-\gamma_p}\|u_0\|_{L^p(\RN)}^{\delta_p} \quad \textup{for all} \quad t>0,
\end{equation}
where
$ \gamma_p=\frac{N}{(m-1)N+2p(1-s)}$, $ \delta_p=\frac{2p(1-s)}{(m-1)N+2p(1-s)}$.
\end{theorem}
\begin{proof}
We  combine \eqref{Lpdecay} with the  Nash-Gagliardo-Niremberg Inequality \eqref{NGNIneq} applied to the function $\displaystyle {f= u^{(m+p-1)/2}}$ to get a starting point for a Moser iteration. Then we continue as in  \cite[Theorem 8.2]{PQRV2} where the authors consider the model $u_t + (-\Delta)^{\sigma/2}u^m=0$ for $\sigma= 2-2s$. From here, the proof is straightforward.
\end{proof}

\begin{remark} {\rm In the limit $m \to 1^
+$, Theorems \ref{Thm1PMFP}, \ref{ThmSmoothing} (and also Theorem \ref{ThmExistL1}) recover some of the results of the linear Fractional Heat Equation (cf.  \cite{BSV16})}.
\end{remark}

\begin{theorem}
Let $m \ge 2$, $N\ge 1$, $s\in (0,1)$. Let $u$ be a weak solution of Problem \eqref{model1} as constructed in Theorem \ref{ThmExistL1} with compactly supported initial data $u_0 \in L^1(\RN)$.  Then $u(\cdot,t)$ is compactly supported for all $t>0$, i.e. the solution has finite speed of propagation.
\end{theorem}
\begin{proof}
Once we construct a weak solution of Problem \eqref{model1}, we apply the results from \cite{StanTesoVazJDE}. The proof is based on a careful construction of barrier functions, called \emph{true supersolutions} in \cite{CaffVaz}.
\end{proof}

%%%%%%%%%%%%%%%%%%   new section %%%%%%%%%%%%%%%%%%%%%%%%

\section{Functional setting}

\subsection{The fractional Laplacian and the inverse operator}\label{Subsec:FracLap}
We remind some definitions and basic notions for the functional setting of the problem. We will work with the following functional spaces (see \cite{Hitch2012}).
Let $\mathcal{F}$ denote the Fourier transform. For given $s\in(0,1)$ we consider the space
$$
H^s(\RN):=\left\{u:L^2(\RN): \int_{\RN}(1+|\xi|^{2s})|\mathcal{F}u(\xi)|^2d\xi < +\infty \right\},
$$
with the norm
$$
\|u\|^2_{H^s(\RN)}:=\|u\|^2_{L^2(\RN)} + \int_{\RN} |\xi|^{2s}|\mathcal{F}u(\xi)|^2d\xi.
$$
For functions $u \in H^s(\RN)$, the fractional Laplacian operator is defined by
$$(-\Delta)^{s} u (x)= C_{N,s} \, \text{P.V.}\int_{\RN} \frac{u(x)-u(y)}{|x-y|^{N+2s}}dy=   \mathcal{F}^{-1}(|\xi|^{2s}(\mathcal{F}u)),
$$
for $x\in \RN$, where $C_{N,s}=\pi^{-(2s+N/2)}\Gamma(N/2+s)/\Gamma(-s).$
 Then,
$$
\|u\|^2_{H^s(\RN)}=\|u\|^2_{L^2(\RN)} +  C \| (-\Delta)^{s/2}u\|_{L^2(\RN)}^2.
$$
For functions $u$ that are defined on a subset $\Omega \subset \RN$ with $u=0$ on the boundary $\partial \Omega$, we will use the \emph{restricted} version of the fractional Laplacian computed by extending the function $u$ to the whole $\RN$ with $u=0$ in $\RN \setminus \Omega.$  The same idea is used to define the $H^s(\RN)$ norm for functions defined in $\Omega$.

If $N>2s$, the inverse operator $(-\Delta)^{-s}$ coincides with the Riesz potential of order $2s$. It can be represented by convolution with the Riesz kernel $K_s$:
\[
(-\Delta)^{-s} u=K_s*u, \ \ \ K_s(x)=\frac{1}{c(N,s)}|x|^{-(N-2s)},
\]
where $c(N,s)=\pi^{N/2-2s}\Gamma(s)/\Gamma((N-2s)/2).$ Notice that $K_s \in L^1_{\textup{loc}}(\RN)$. When $N=1$ and $s\in[1/2,1)$ we have to consider the composed operator $\nabla (-\Delta)^{-s}$. This operator use to be called \em nonlocal gradient \em and is denoted by $\nabla^{1-2s}$ (c.f. \cite{BilerImbertKarch,StanTesoVazJDE}). See Section \ref{SectN1} for a more detailed discussion of this range.

\subsection{Approximation of the fractional Laplacian \texorpdfstring{ $(-\Delta)^s$}{} } \label{Subsec:AproxFracLap} Let $\epsilon>0$ and $u:\RN \to \R$. We define the operator
\begin{equation}\label{AproxFracLap}\mathcal{L}^s_\epsilon [u] (x):= C_{N,s}\int_{\RN} \frac{u(x)-u(y)}{\left(|x-y|^{2}+\epsilon^2\right)^{\frac{N+2s}{2}}}dy,
\end{equation}
for $x\in \RN.$
We will use the notation
$$
J_\epsilon^s(z):=\frac{C_{N,s}}{\left(|z|^{2}+\epsilon^2\right)^{\frac{N+2s}{2}}} \quad \text{for } z \in \RN.
$$
It is clear that $\| J_\epsilon^s\|_{L^1(\RN)}<\infty$ since $J_\epsilon^s$ is integrable at infinity and nonsingular at the origin. Thus \eqref{AproxFracLap} is equivalent to
\begin{equation}\label{AproxFracLap:Equiv}
\mathcal{L}^s_\epsilon [u] (x) =u(x)\|J_\epsilon^s\|_{L^1(\RN)} - (u(t,\cdot)\star J_\epsilon^s) (x).
\end{equation}
This kind of zero-order operators has been considered in the literature, see e.\,g. \cite{AMRT, IgRo09,Ros08}. For any $\epsilon>0$,  $\mathcal{L}_\epsilon^s$ is an integral operator with non-singular kernel and  $\mathcal{L}^s_\epsilon [u]  \to (-\Delta)^{s}u$  pointwise in $\RN$ as $\epsilon \to 0$ for suitable  functions $u$. This approximation can also be seen as a consequence of the fact that the fractional Laplacian can be computed by passing to the limit in the representation of the solution of an harmonic extension problem (using the explicit Poisson formula), as proved by Caffarelli and Silvestre in \cite{CaffSilvExt}.

We can define the bilinear form
$$\mathcal{E}_\epsilon(u,v)= \frac{C_{N,s}}{2}\int_{\RN}\int_{\RN} \frac{\left(u(x)-u(y) \right)\left(v(x)-v(y) \right)}{(|x-y|^2+\epsilon^2)^{\frac{N+2s}{2}}}dx dy  \quad  \text{for }u,v\in D(\mathcal{L}_{\epsilon}),$$
and the quadratic form
$$ \overline{\mathcal{E}}_\epsilon(u):=\mathcal{E}_\epsilon(u,u)=\frac{C_{N,s}}{2}\int_{\RN}\int_{\RN} \frac{\left[u(x)-u(y) \right] ^2}{(|x-y|^2+\epsilon^2)^{\frac{N+2s}{2}}}dx dy . $$
The bilinear form ${\mathcal{E}}_\epsilon$ is well defined for functions in    $L^2(\R^N)$ since the $J_\epsilon^s$ is bounded and integrable. We refer to \cite{dTEnJa16b} for a precise discussion of the natural spaces in a more general framework.

% the space $\dot{H}_{\epsilon}^s (\RN)$, which is the closure of $C_c^\infty(\RN)$ with respect to the Gagliardo seminorm given by $\overline{\mathcal{E}}_\epsilon$. We define
%\begin{equation}\label{spaceHs}
%H_{\epsilon}^s (\RN)=\left\{ u\in L^2(\RN):  \overline{\mathcal{E}}_\epsilon(u) < \infty \right\}.
%\end{equation}
%The space $H_{\epsilon}^s (\RN)$ is endowed with the standard norm
%$$\|u\|^2_{H_{\epsilon}^s}=  \|u\|^2_{L^2(\RN)} +\overline{\mathcal{E}}_\epsilon(u). $$
%Clearly,
%\begin{equation}\label{SpaceHInclusions}
%H^s (\RN) \subset H_{\epsilon_1}^s(\RN) \subset  H_{\epsilon_2}^s(\RN) \quad \textup{ for } \quad 0<\epsilon _1 <\epsilon_2.
%\end{equation}

\begin{lemma}\label{lem:l1linftyL}Let $0<s<1$. Then, for every $\epsilon>0$,  we have that $$\mathcal{L}^s_\epsilon: L^1(\RN) \cap L^{\infty}(\RN) \to L^1(\RN) \cap L^{\infty}(\RN).$$  Moreover,
$$
\| \mathcal{L}^s_\epsilon[u]\|_{L^1(\RN)} \le 2\|u \|_{L^1(\RN)}  \|J_\epsilon^s \|_{L^1(\RN)},
$$
$$\| \mathcal{L}^s_\epsilon[u]\|_{L^{\infty}(\RN)} \le 2\|u \|_{L^{\infty}(\RN)}  \|J_\epsilon^s \|_{L^1(\RN)}. $$
\end{lemma}

\begin{proof}
Let $u \in L^1(\RN) \cap L^{\infty}(\RN)$, then using \eqref{AproxFracLap:Equiv} and the Young Inequality for convolutions
 the stated estimates follow.
%
% \color{red} Lo siguiente NO ES NECESARIO SI USAMOS \eqref{AproxFracLap:Equiv}
%\begin{equation*}
%\begin{split}
%\| \mathcal{L}^s_\epsilon[u]\|_{L^1(\RN)}&=\int_{\RN} \left|\int_{\RN}(u(x)-u(x+y) )J_\epsilon^s(y)dy\right| dx\leq \int_{\RN} J_\epsilon^s(y) \int_{\RN}\left|u(x)-u(x+y) \right|dx dy\\
%&\leq 2\|u \|_{L^1(\RN)}\int_{\RN} J_\epsilon^s(y)dy,
%\end{split}
%\end{equation*}
%and
%\begin{equation*}
%\begin{split}
%| \mathcal{L}^s_\epsilon[u](x)|&=\left|\int_{\RN}(u(x)-u(x+y) )J_\epsilon^s(y)dy\right|\leq \int_{\RN}  \left|u(x)-u(x+y) \right|J_\epsilon^s(y) dy\\
%&\leq 2\|u \|_{L^\infty(\RN)}\int_{\RN} J_\epsilon^s(y)dy .
%\end{split}
%\end{equation*}
\end{proof}

\noindent\textbf{The restricted operator.}     For smooth functions $f:B_R \to \R$   we extend $f=0$ on $\RN \setminus B_R$. In this way $\mathcal{L}^s_\epsilon $ is well defined for $f\in L^2(B_R)$ by \eqref{AproxFracLap}.

We will also use the following result regarding the composed operator $\nabla (-\Delta)^{-1} \mathcal{L}^{1-s}_\epsilon$ that we will treat in Section \ref{sec:aproxinverse} as a natural approximation of $\nabla(-\Delta)^{-s}$.

\begin{lemma} \label{Lemma:nonlocal} Let $0<s<1$. Then, for every $\epsilon,R>0$  we have that $$\nabla (-\Delta)^{-1} \mathcal{L}^{1-s}_\epsilon: L^1(B_R) \cap L^{\infty}(B_R) \to L^1(B_R) \cap L^{\infty}(B_R).$$  Moreover,
\[
\|\nabla (-\Delta)^{-1}\mathcal{L}_\epsilon^{1-s}[f]\|_{L^1(B_R)} \leq C ( \left\| f\right\|_{L^\infty(B_R)} +  \left\| f\right\|_{L^1(B_R)})
\]
\[
\|\nabla (-\Delta)^{-1}\mathcal{L}_\epsilon^{1-s}[f]\|_{L^\infty(B_R)} \leq C( \left\| f\right\|_{L^\infty(B_R)} +  \left\| f\right\|_{L^1(B_R)}).
\]
\end{lemma}
\begin{proof}
We will write $\sim$ and $\lesssim$ to represent identities and inequalities up to constants depending on $R,N$ and $\epsilon$.

For $N\ge 2$ and $p=\{1,\infty\}$, we use Lemma \ref{lem:l1linftyL} with $f$ extended by 0 outside $B_R$ and the explicit form of the Newtonian potential to get
\begin{align*}
 \| \nabla (-\Delta)^{-1} \mathcal{L}_\epsilon^{1-s}[f]\|_{L^p( B_R)} &\lesssim   \left\| \int_{\RN}\frac{1}{|x-y|^{N-1}} |\mathcal{L}_\epsilon^{ 1-s}[f(y)] |dy \right\|_{L^p(B_R)}\\
 &\le \|\mathcal{L}_\epsilon^{1-s}[f] (y)\|_{L^p(B_R)}\,  \int_{B_R} \frac{1}{|x|^{N-1}}dx \lesssim  \|f\|_ {L^p(B_R)}.
\end{align*}
When $N=1$, we note that $\nabla (-\Delta)^{-1} g(x)=-\int_{-\infty}^x g(y) dy$, and thus \[  \nabla (-\Delta)^{-1} \mathcal{L}_\epsilon^{1-s}[f](x)=-\int_{-\infty}^x  \mathcal{L}_\epsilon^{1-s}[f](y) dy.\]
Then,
\begin{align*}
\| \nabla (-\Delta)^{-1} \mathcal{L}_\epsilon^{1-s}[f](x)\|_{L^p(B_R)}& \lesssim
\left\| J^{1-s}_\epsilon \right\|_{L^1(\RN)} \int_{-\infty}^\infty  |f(y)|  dy \lesssim \|f\|_{L^1(B_R)}.
\end{align*}
\end{proof}

\noindent\textbf{Square root.} The operator $\mathcal{L}^s_\epsilon$ has a square-root in the Fourier transform sense  \cite[Lemma 3.7]{EndalJakobsenTeso}, that we denote by $(\mathcal{L}^s_\epsilon)^{\frac{1}{2}}$. We have that
$$ <u,\mathcal{L}^s_\epsilon[u]>_{L^2(\RN)} =\|(\mathcal{L}^s_\epsilon)^{\frac{1}{2}}[u]\|^2_{L^2(\RN)}. $$
This implies that
\begin{align*}
<\mathcal{L}^s_\epsilon[u],u>_{L^2(\RN)}&=C_{N,s} \int_{\RN}\int_{\RN} u(x) \frac{u(x)-u(y)}{(|x-y|^2+\epsilon^2)^{\frac{N+2s}{2}}}dx dy\\&=
\frac{C_{N,s}}{2}\int_{\RN}\int_{\RN} \frac{\left[u(x)-u(y) \right] ^2}{(|x-y|^2+\epsilon^2)^{\frac{N+2s}{2}}}dx dy \\
&=\frac{C_{N,s}}{2}\int_{\RN}\int_{\RN} \left[\frac{u(x)-u(y) }{(|x-y|^2+\epsilon^2)^{\frac{N+2s}{4}}}\right] ^2 dx dy,
\end{align*}
where the second identity is obtained by symmetry. We get the following characterization of $(\mathcal{L}^s_\epsilon)^{\frac{1}{2}}$:
\begin{equation}\label{operatoronehalf}
 \int_{\RN}\left((\mathcal{L}^s_\epsilon)^{\frac{1}{2}} [u](x) \right)^2 dx  =\frac{C_{N,s}}{2}\int_{\RN}\int_{\RN} \left[\frac{u(x)-u(y) }{(|x-y|^2+\epsilon^2)^{\frac{N+2s}{4}}}\right] ^2 dx dy.
 \end{equation}

\begin{theorem}[\textbf{Generalized Stroock-Varopoulos Inequality for $\mathcal{L}^s_\epsilon $}]Let $u \in   L^2(\RN) $. Let $\psi: \R \to \R$ such that $\psi \in C^1(\R)$ and $\psi' \ge 0$. Then
\begin{equation}\label{GenStroockVarAprox}
\int_{\RN}\psi(u)\mathcal{L}^s_\epsilon [u] dx \ge \int_{\RN}\left|(\mathcal{L}^s_\epsilon)^{\frac{1}{2}}[\Psi (u)] \right|^{2} dx,
\end{equation}
where $\psi'=(\Psi')^2$.
\end{theorem}
\begin{proof}
 We have that:
\begin{align*}\int_{\RN}\psi(u)\mathcal{L}^s_\epsilon [u] dx &=C_{N,s}\int_{\RN}\int_{\RN}\psi(u(x)) \frac{u(x)-u(y)}{(|x-y|^2+\epsilon^2)^{\frac{N+2s}{2}}}\, dxdy\\
&=\frac{C_{N,s}}{2}\int_{\RN}\int_{\RN} \left[\psi(u(x))-\psi(u(y)) \right] \frac{u(x)-u(y)}{(|x-y|^2+\epsilon^2)^{\frac{N+2s}{2}}}dx dy.
\end{align*}
Now, we use that if $\psi$ is such that $\psi'\ge 0$ and $\psi'=(\Psi')^2$, then
$$\left( \psi(a)-\psi(b)\right) (a-b) \ge \left( \Psi(a)-\Psi(b)\right)^2, \quad \forall a,b \in \RN. $$
For convenience, we give the proof of this pointwise inequality based on the Fundamental Theorem of Calculus and the Cauchy-Schwarz Inequality:
\begin{align*}
\left( \Psi(a)-\Psi(b)\right)^2 &= \left( \int_{b}^a \Psi'(z)dz \right)^2 \le  (a-b)\int_{b}^a \left( \Psi'(z)\right)^2 dz \\
&= (a-b)\int_{b}^a \psi'(z) dz = (a-b) (\psi(a)-\psi(b)).
\end{align*}
We deduce, using \eqref{operatoronehalf}, that
\begin{align*}\int_{\RN}\psi(u)\mathcal{L}_\epsilon (u) dx & \ge \frac{C_{N,s}}{2} \int_{\RN}\int_{\RN} \frac{\left[\Psi(u(x))-\Psi(u(y)) \right]^2 }{(|x-y|^2+\epsilon^2)^{\frac{N+2s}{2}}}dx dy= \int_{\RN}\left|(\mathcal{L}^s_\epsilon)^{\frac{1}{2}}\Psi (u(x)) \right|^{2} dx.
\end{align*}
\end{proof}
\begin{remark}{\rm (i) We refer to \cite{dTEnJa16b} for a related result with more general nonlinearities and nonlocal operators. \\
\noindent (ii) Note that we recover the classical Stroock-Varopoulos Inequality for $\mathcal{L}_\epsilon$ by taking $\psi(u)=|u|^{q-2}u$:
$$
\int_{\RN}|u|^{q-2}u \,\mathcal{L}^s_\epsilon (u) dx \ge \frac{4(q-1)}{q^2}\int_{\RN}\left|(\mathcal{L}^s_\epsilon)^{1/2} (u^{q/2})\right|^{2} dx.
$$
\noindent We refer to Stroock \cite{StroockLargeDev}, Liskevich and Semenov \cite{LiskevichSemPAMS} where this kind of inequality is proved for  general sub-markovian operators. }
\end{remark}

\subsection{Approximation of the inverse fractional Laplacian \texorpdfstring{$(-\Delta)^{-s}$}{}, \texorpdfstring{$s\in (0,1)$}{}}\label{sec:aproxinverse}

  By using  \eqref{AproxFracLap} we introduce   an approximation for the inverse fractional Laplacian $(-\Delta)^{-s}$ and the nonlocal gradient $\nabla^{1-2s}$  that will play an important role in the sequel to solve the difficulties created by estimates like \eqref{FirstEnergy} in the range $m\ge 3$.   More precisely we propose to approximate $(-\Delta)^{-s}$ by $(-\Delta)^{-1} \mathcal{L}^{1-s}_\epsilon$ and $\nabla^{1-2s}$ by $\nabla (-\Delta)^{-1} \mathcal{L}^{1-s}_\epsilon$.
\begin{lemma}\label{LemmaAproxInverseFracLap}
a) Let $N\ge  1$, $s\in (0,1)$ and $s < \frac{N}{2}$.  Then for every $f \in L^1(\RN)$ such that $(-\Delta)^{-s}f\in L^2(\RN)$ we have that
\begin{equation*}
\mathcal{I}_{\epsilon}:=\int_{\RN} \left( (-\Delta)^{-1} \mathcal{L}^{1-s}_\epsilon [f] - (-\Delta)^{-s}f\right) \phi \,dx  \to 0 \quad \text{as}\quad \epsilon \to 0 , \quad \forall \phi \in C_c^{\infty}(\RN).
\end{equation*}

b)  Let $N\ge  1$, $s\in (0,1)$.  Then for every  $f \in L^1(\RN)$ such that $\nabla^{1-2s}f\in L^2(\RN)$  we have that
\begin{equation*}
\mathcal{I}_{\epsilon}:=\int_{\RN} \left(\nabla (-\Delta)^{-1} \mathcal{L}^{1-s}_\epsilon [f] - \nabla^{1-2s}f\right) \phi \,dx  \to 0 \quad \text{as}\quad \epsilon \to 0, \quad \forall \phi \in C_c^{\infty}(\RN).
\end{equation*}
\end{lemma}
\begin{proof}a)
Given any operator $T$, let $S_T(\xi)$ be the Fourier symbol associated to the operator $T$ whenever it is well defined.  Now, we  employ  Plancherel's Theorem to obtain:
\begin{equation*} \begin{split}
\mathcal{I}_{\epsilon} &= \int_{\RN}  \left( \mathcal{S}_{(-\Delta)^{-1}}(\xi) \mathcal{S}_{\mathcal{L}^{1-s}_\epsilon}(\xi)  -  \mathcal{S}_{(-\Delta)^{-s}}(\xi) \right) \widehat{f} \widehat{ \phi} d\xi  =: \int_{\RN}  F_\epsilon(\xi) d\xi.
\end{split}
\end{equation*}
We want to pass to the limit as $\epsilon \to0$ in $\mathcal{I}_{\epsilon}$. For that purpose we need to find an $L^1$ dominating function for $F_\epsilon$. We recall that for $s\in(0,1)$ we have that
\begin{equation}\label{FourSym}
\mathcal{S}_{\mathcal{L}^{1-s}_\epsilon}(\xi)=\int_{|z|>0} \frac{1-\cos(z\cdot \xi)}{(|z|^2+\epsilon^2)^{\frac{N+2(1-s)}{2}}}dz \quad \textup{and} \quad \mathcal{S}_{(-\Delta)^{1-s}}(\xi)=\int_{|z|>0} \frac{1-\cos(z\cdot \xi)}{|z|^{N+2(1-s)}}dz \sim |\xi|^{2(1-s)}.
\end{equation}
Moreover $\mathcal{S}_{(-\Delta)^{-s}}(\xi)=\mathcal{S}_{(-\Delta)^{-1}}(\xi)\mathcal{S}_{(-\Delta)^{1-s}}(\xi)$. Note that $0\leq\mathcal{S}_{\mathcal{L}^{1-s}_\epsilon}(\xi)\leq \mathcal{S}_{(-\Delta)^{1-s}}(\xi)$ for every $\xi \in \RN$. Then
\begin{equation*} \begin{split}   |F_\epsilon(\xi) | &\le
 \left|  \mathcal{S}_{(-\Delta)^{-1}}(\xi) \mathcal{S}_{\mathcal{L}^{1-s}_\epsilon}(\xi)  \right|  \left| \widehat{f}  \right|  \left|\widehat{ \phi} \right|  +  \left|  \mathcal{S}_{(-\Delta)^{-s}}(\xi)  \right|   \left| \widehat{f}  \right|  \left|\widehat{ \phi} \right| \\
 & \le \left|  \mathcal{S}_{(-\Delta)^{-1}}(\xi)
 \mathcal{S}_{(-\Delta)^{1-s}}(\xi)  \right|  \left| \widehat{f}  \right|  \left|\widehat{ \phi} \right|  +  \left|  \mathcal{S}_{(-\Delta)^{-s}}(\xi)  \right|   \left| \widehat{f}  \right|  \left|\widehat{ \phi} \right| \\
 &= 2 \left|  \mathcal{S}_{(-\Delta)^{-s}}(\xi) \widehat{f}  \widehat{ \phi} \right|  \, \leq \, C |\xi|^{-2s}  \left|\widehat{f}\right| \left| \widehat{\phi}\right|.
\end{split}
\end{equation*}
We conclude that $|F_\epsilon(\xi) |\leq G(\xi,t):=C\left| |\xi|^{-2s} \widehat{f}  \widehat{\phi}\right| \in L^1(\RN)$ since $\widehat{f}\in L^\infty(\RN)$ and  $\widehat{\phi}\in \mathcal{S}(\RN)$, the Schwartz space of rapidly decaying functions. Moreover, we can see from \eqref{FourSym} that $F_\epsilon(\xi) \to 0 $ pointwise as $\epsilon \to 0$. Then we use the Dominated Convergence Theorem to conclude that
$|\mathcal{I}_{\epsilon} | \to 0$ as $\epsilon \to 0.$

b) The proof follows as above noting that $\mathcal{S}_{\nabla}= i \xi$ and
$|F_\epsilon(\xi) |\leq C\left| |\xi|^{1-2s} \widehat{f}  \widehat{\phi}\right| \in L^1(\RN)$ .

\end{proof}

\section{Existence of weak solutions via approximating problems}\label{SectExist}

In order to prove existence of weak solutions of Problem \eqref{model1} we  proceed by considering an approximating problem. We  regularize the degeneracy of the nonlinearity, the singularity of the fractional operator, we also add a vanishing viscosity term to get more regularity and we restrict the problem to a bounded domain.    We write the equation in the form
$$
u_t= \nabla \cdot(u^{m-1} \nabla (-\Delta)^{-1} (-\Delta)^{1-s}u).
$$
The idea is to consider the approximation of the $(-\Delta)^{1-s}$ given by \eqref{AproxFracLap}, that is
$$
\mathcal{L}^{1-s}_\epsilon (u)(x)=C_{N,1-s}\int_{\RN}\frac{u(x)-u(y)}{\left(|x-y|^2+\epsilon^2\right)^{\frac{N+2-2s}{2}}}dy,
$$
defined for functions $u$ in the natural space $  L^2(\R^N) $. We consider the approximating problem
\[
\left\{
\begin{array}{ll}
(U_1)_t= \delta \Delta U_1 +\nabla \cdot((U_1+\mu)^{m-1}\nabla (-\Delta)^{-1} \mathcal{L}^{1-s}_\epsilon [U_1])&\text{for } (x,t)\in B_R \times (0,T),\\
U_1(x,0)=\widehat{u}_0(x) &\text{for } x \in B_R,\\
U_1(x,t)=0 &\text{for } x\in \partial B_R, \ t\in (0,T),
\end{array}
\right.
\tag{$P_{\epsilon\delta\mu R}$}\label{ProblemEpsMuDeltaR}
\]
with parameters $\epsilon,\delta,\mu, R>0$. We use the notation $B_R:=B_R(0)$. The initial data $\widehat{u}_0$ is a smooth approximation of $u_0$.
 We recall that the operator $ \mathcal{L}^{1-s}_\epsilon [U_1]$ is defined by formula \eqref{AproxFracLap} extending the function $U_1$ by $0$ on $\RN \setminus B_R$ as in Section \ref{AproxFracLap}. Moreover, $U_1 \in L^2(0,T:H_0^1(B_R))$ as we will prove in formula \eqref{energy6}, therefore it has the right decay at the boundary $\partial B_R$ that  allows its extension by $0$.

 The existence of a weak solution of Problem \eqref{model1} is done by passing to the limit step-by-step in the approximating problems as follows. We denote by $U_1$ the solution of the approximating Problem \eqref{ProblemEpsMuDeltaR} with parameters $\epsilon,\delta,\mu,R$. Afterwards, we obtain $U_2=\lim_{\epsilon\to 0}U_1$ and $U_2$  solves an approximating Problem \eqref{ProblemMuDeltaR} with parameters $\delta,\mu,R$. Next, we take $U_3=\lim_{R\to \infty}U_2$ that will be a solution of Problem \eqref{ProblemMuDelta}, $U_4:=\lim_{\mu \to 0}U_3$ solving Problem \eqref{ProblemDelta}. Finally we obtain $u=\lim_{\delta \to 0}U_4$ which solves Problem \eqref{model1}.
Notice that the $\delta\to 0$ is the last limit considered in the approximation process. This is because the $\delta \Delta$-term gives $H^1_0(B_R)$ regularity for $U_1$ and $U_2$, respectively $H^1(\RN)$ for $U_3$ and $U_4$. Thus $U_1$ and $U_2$ will be solutions to Dirichlet problems with homogenous boundary conditions. The $H^1_0(B_R)$ regularity allows their extension by $0$ to $\RN \setminus B_R$ and thus the nonlocal operators involved in the equations are properly defined as in Sections \ref{Subsec:FracLap} and \ref{Subsec:AproxFracLap}.

%
%An important tool in the proof of existence of weak solutions is the concept of mild solution of Problem \eqref{ProblemEpsMuDeltaR} , i.e. fixed points of the following map given by the Duhamel's formula
%$$
%\mathcal{T}(v)(x,t)= e^{\delta t \Delta}u_0(x) + \int_0^t \nabla e^{\delta (t-\tau) \Delta} \cdot  G(v)(x,\tau)  d\tau, \quad G(v)= (v+\mu)^{m-1}\nabla (-\Delta)^{-1} \mathcal{L}_\epsilon^{1-s}[v],
%$$
%where $e^{ t \Delta}$ is the Heat Semigroup. The map $v \mapsto \mathcal{T}(v)$,  $$\mathcal{T}: C((0,T):L^1(B_R) \cap L^\infty(B_R)) \to  C((0,T):L^1(B_R) \cap L^\infty(B_R)) $$ is well defined and moreover, $\mathcal{T}$ is a contraction. By the Banach contraction principle we obtain that there exists a fixed point $\mathcal{T}(U_1)=U_1$. It remains to prove that $U_1$ is a weak solution of Problem \eqref{ProblemEpsMuDeltaR}. The method of mild solutions via Duhamel formula was successfully employed to prove existence of approximated solutions for another nonlocal porous medium model by Biler, Imbert and Karch in \cite{BilerImbertKarch}. Their approximation is slightly different but  the technical part can be adapted to Problem \eqref{ProblemEpsMuDeltaR}. In fact one can prove that the solutions are classical.
%
% Add here Correction Remark 1. Existence of mild solutions.
%
% Add here Correction Remark 2 : Regularity of the Fixed Point and prove that mild solution= the weak solution.

\medskip

\noindent\textbf{Notations.}
 We will often use $\int_0^t f(t) dt$ to avoid introducing new variables. Also, we will use $\int_{\RN}$ instead of $\int_{B_R}$ when integrating some expressions of $U_1,U_2$, which are supported in $B_R$, by identifying these functions with $0$ outside the domain $B_R$. The homogeneous Dirichlet boundary conditions ensures that the integrals coincide.

We will use $\to$ for strong convergence and $\rightharpoonup$ for weak convergence. We will write $\sim$ and $\lesssim$ when multiplying by constants depending on $N,\delta,R,\epsilon$ and the norms $p,q$ that we will use. We will keep explicit the constants relevant in the proof. We will also avoid to write the variable $x$ and write just $v(t)$ when considering the norms in $x$.

\subsection{Existence of solutions of \texorpdfstring{\eqref{ProblemEpsMuDeltaR}}{}}
 We will use a standard technique: first we will prove that there exists a unique weak solution  by the method of fixed point of a contraction mapping. Then we show the regularity of the fixed point and prove that it is in fact a strong solution to the problem.  We give now the definitions of weak and strong solution for \eqref{ProblemEpsMuDeltaR}.

\begin{defn}
We say that $U_1 $ is a weak solution of Problem \eqref{ProblemEpsMuDeltaR} if: (i) $U_1 \in L^1(B_R \times (0,T))$ , (ii) $\nabla (-\Delta)^{-1} \mathcal{L}^{1-s}_\epsilon [U_1] \in L^1(B_R \times (0,T))$, (iii) $(U_1+\mu)^{m-1} \nabla (-\Delta)^{-1} \mathcal{L}^{1-s}_\epsilon [U_1] \in L^1(B_R \times (0,T))$ and
\begin{equation}\label{weaksolAprox}
\int_0^T\int_{B_R} U_1(\phi_t+\delta \Delta \phi)dxdt-\int_0^T\int_{B_R}  (U_1+\mu)^{m-1} \nabla (-\Delta)^{-1} \mathcal{L}^{1-s}_\epsilon [U_1]\cdot\nabla \phi dxdt+ \int_{B_R} \widehat{u}_0(x) \phi(x,0)dx=0
\end{equation}
 for smooth test functions $\phi$ that vanish on the spatial boundary $\partial B_R$ and $t=T$. We will say that $U_1$ is a strong solution if additionally $(U_1)_t,\Delta U_1, \nabla\cdot ((U_1+\mu)^{m-1} \nabla (-\Delta)^{-1} \mathcal{L}^{1-s}_\epsilon [U_1]) \in L^p(B_R\times(0,T))$ for some $p\geq1$ and \eqref{ProblemEpsMuDeltaR} is satisfied pointwise almost everywhere.
\end{defn}

 \subsubsection{Solution of a heat equation with forcing term}
We consider an arbitrary value of the unknown $U_1$ in the last term of \eqref{ProblemEpsMuDeltaR} and solve the following heat equation with a forcing term
\begin{equation}\label{eq:HEright}
u_t=\delta \Delta u +\nabla \cdot G(v) \quad \textup{with} \quad G(v)=(v+\mu)^{m-1}\nabla (-\Delta)^{-1} \mathcal{L}_\epsilon^{1-s}[v]
\end{equation}
with initial data $u(x,0)=\widehat{u}_0(x)$ for $x\in B_R$ and lateral data $u(x,t)=0$ for $(x,t)\in B_R^c\times(0,T)$. We recall that $\widehat{u}_0(x)$ is a smooth approximation of $u_0$ but we will only use the  $L^p$ norms of $\widehat{u}_0$ and $\nabla \widehat{u}_0$. In order to apply of the fixed point theorem we will choose $v$ in a convenient  functional space and solve \eqref{eq:HEright} to find $u$. We want to define a mapping $\mathcal{T}:v\mapsto u$ and we will prove that $\mathcal{T}$ has a fixed point.

%(i) Let us take as a functional space $X=C([0,T]:L^1(B_R) \cap L^\infty(B_R))$ for a certain time $T>0$ to be chosen below.
\begin{proposition}\label{prop:welldefT}
Let $X=L^1(B_R) \cap L^\infty(B_R)$. Then $\mathcal{T}$ is well defined from $X_T:=C([0,T]:X)$ into $X_T$ for all $T>0$. Moreover, for every $v\in X_T\cap L^2([0,T], H^1_0(B_R))$, we have that $u=\mathcal{T}(v)$ is a strong solution of \eqref{eq:HEright} with the given initial and lateral data. We have also precise estimates for $\mathcal{T}$.
\end{proposition}
Before proving the result above, we need the following lemma:

\begin{lemma}\label{lem:GX}
For every $v\in X_T$ we have that $G(v)\in X_T$ with $\|G(v)\|_{X_T}\leq C \|v\|_{X_T}$ where $C=C(\|v\|_{L^\infty(Q_T)})$.
\end{lemma}
\begin{proof} Here $T$ is arbitrary and we denote $Q_T=B_R\times[0,T]$.  It is enough to prove the result for fixed time, and the continuity in time follows easily.
By Lemma \ref{Lemma:nonlocal} we have that
\[
\|\nabla (-\Delta)^{-1}\mathcal{L}_\epsilon^{1-s}v(\cdot,t)\|_{X} \lesssim \left\| v(\cdot,t)\right\|_{X}.\]
Taking supremums in $t\in[0,T]$ in the above equation we get $
\|\nabla (-\Delta)^{-1}\mathcal{L}_\epsilon^{1-s}v\|_{X_T}\lesssim \|v\|_{X_T}$. From here we conclude that
\[
\|G(v)\|_{X_T}\leq \|v+\mu\|^{m-1}_{L^\infty(Q_T)}\|\nabla (-\Delta)^{-1}\mathcal{L}_\epsilon^{1-s}v\|_{X_T}\lesssim\
 C \|v\|_{X_T}.
\]
\end{proof}
\begin{proof}[Proof of Proposition \ref{prop:welldefT}]

(i) The standard theory for the heat equation  (see for instance \cite{Pazy})   says that given such forcing term $F:=\nabla \cdot G(v)$, there exists a unique weak solution $u\in X_T$ of the above initial and boundary value problem. Moreover, by the regularity theory, we also know that $\nabla u\in L^p(Q_T)$ for every $p\in[1,\infty)$ since $G(v)\in L^p(Q_T)$. We can express the weak solution by means of the Duhamel formula:
\begin{equation*}
u(x,t)=\underbrace{e^{\delta t \Delta}\widehat{u}_0(x) + \int_0^t \nabla e^{\delta (t-\tau) \Delta} \cdot  G(v)(x,\tau)  d\tau}_{\mathcal{T}(v)}, \quad G(v)= (v+\mu)^{m-1}\nabla (-\Delta)^{-1} \mathcal{L}_\epsilon^{1-s}[v],
\end{equation*}
where $e^{ t \Delta}$ is the Heat Semigroup corresponding to the homogenous Dirichlet problem in the ball $B_R$.
This formula will be convenient to perform a priori estimates needed for the fixed point argument. When $v\in L^2([0,T], H^1_0(B_R))$ we can work out the expression for $F$
\[
F=\nabla(v+\mu)^{m-1}\cdot \nabla (-\Delta)^{-1} \mathcal{L}_\epsilon^{1-s}[v]-(v+\mu)^{m-1} \mathcal{L}_\epsilon^{1-s}[v].
\]
It follows that $F\in L^2(Q_T)$. The standard heat equation theory now implies that $u$ is a strong solution of the problem and $u_t,\Delta u \in L^2(Q_T)$.

(ii) We now prove that for $v\in X_T$ we have $\mathcal{T}(v(t)) \in X$ for all $t \in [0,T]$ with precise estimates.
We will need some decay properties of the Heat Semigroup in $B_R$.
% $$
%\| e^{\delta t \Delta}v\|_{L^p(B_R)} \le  e^{-c_p \lambda_1  \delta \, t}\|v\|_{L^p(B_R)},
%$$
%where $\lambda_1>0$ is the first eigenvalue of the $-\Delta$ operator on $B_R$ with homogenous Dirichlet boundary conditions, and $c_p=4(p-1)/p^2$. Here, $1\le p \le \infty$. This follows from energy estimates.
% Also using energy estimates
%$$\int_0^t \int_{B_R}|\nabla e^{\tau \Delta} v|^2 dx d\tau \le \frac{1}{2} \|v\|_{L^2(B_R)}.
%$$
%Due to the monotonicity of $\int_{B_R}|\nabla e^{\tau \Delta} v|^2 dx$ (non-increasing in time) we get that
%$$t \int_{B_R}|\nabla e^{t \Delta} v|^2 dx  \le    \|v\|^2_{L^2(B_R)}.
%$$
%and then
%$$\int_{B_R}|\nabla e^{t \Delta} v| dx \le  |B_R|^{1/2} \left( \int_{B_R}|\nabla e^{t \Delta} v|^2 dx\right)^{1/2}  \le   \, C t^{-\frac{1}{2}} \|v\|_{L^2(B_R)}.
%$$
Using classical estimates on the Green function for the heat operator in a bounded domain   \cite[p.413, Th. 16.3]{Ladyz}   we have that for $1 \le p\le \infty$
\begin{equation}\label{estim:heat1}
\|e^{ t \Delta}v\|_{L^p(B_R)} \le \|v\|_{L^p(B_R)}\quad \textup{and} \quad
\|\nabla e^{ t \Delta}v\|_{L^p(B_R)} \le  t^{-\frac{1}{2}} \|v\|_{L^p(B_R)}.
\end{equation}
Let now $v\in C([0,T]: X)$. Using the heat kernel estimates and Lemma \ref{Lemma:nonlocal}
\begin{align*}
\|\mathcal{T}(v(t))\|_{X} &\le \|\widehat{u}_0\|_X + \int_0^t (t-\tau)^{-\frac{1}{2}} \|G(v(\tau))\|_{X} d\tau\le \|\widehat{u}_0\|_X+ t^{1/2}\sup_{0\le \tau\le t} \|G(v(\tau))\|_{X} d\tau\\
&\le \|\widehat{u}_0\|_X + C t^{1/2}\sup_{0\le \tau\le t} (\|v(\tau)\|_{L^\infty(B_R)}+\mu)^{m-1}  \|v(\tau)\|_ {X}  <\infty.
\end{align*}
(iii) Moreover  $\mathcal{T}(v)$ is continuous with respect to $t$. Indeed, we have that
\begin{align*}
\mathcal{T}(v)(x,t+h) - \mathcal{T}(v)(x,t)  &= e^{\delta ( t+h) \Delta}u_0(x) - e^{\delta t \Delta}u_0(x)  + \int_t^{t+h} \nabla e^{\delta (t+h-\tau) \Delta} \cdot  G(v)(x,\tau)  d\tau \\
&+\int_0^{t} \nabla e^{\delta (t-\tau) \Delta} \cdot (e^{\delta h \Delta} G(v)(x,\tau)  - G(v)(x,\tau) ) d\tau =I+II+III.
\end{align*}

We want to prove that
$\|\mathcal{T}(v)(\cdot,t+h) - \mathcal{T}(v)(\cdot,t)\|_{L^1(B_R)\cap L^\infty(B_R)}  \to 0$ as $h\to 0$. For $p=\{1,\infty\}$, the $L^p$ norms of $I$ and $III$ go to $0$ as $h\to 0$ since the Heat Semigroup is well defined in this space: $\|e^{\delta h \Delta} (f) -f\| \to 0$.
For the second term we should use the decay of the Heat kernel \eqref{estim:heat1}
\begin{align*}
\|II\|_{L^p} &\lesssim \int_t^{t+h}   (t+h-\tau)^{-1/2}  \| G(v(\tau))\|_{L^p(B_R)}  d\tau \\ \nonumber
&\lesssim   h^{1/2}  \sup_{(0,T)}\|G(v)(\cdot, t)\|_{L^p(B_R)}   \to 0 \quad \text{as }h \to 0.
\end{align*}
\end{proof}
\subsubsection{Local in time contraction and existence of a fixed point}

\begin{proposition}\label{Prop:Mild}
Let $K=2 \|\widehat{u}_0\|_{X}$ and denote by  $\overline{B_K}$ the closed ball of radius $K$ centered  at $0$ in the space $X_T=C([0,T]:X)$. There exists $T=T(\|\widehat{u}_0\|_{L^\infty(B_R)})$ small enough such that $\mathcal{T}$ is a contraction in $\overline{B_K}\subset X_T$. Therefore, $\mathcal{T}$ has a fixed point in $\overline{B_K}\subset X_T$. More precisely, we can take $T\leq C(K+\mu)^{2(1-m)}$.
\end{proposition}
 \begin{proof}

 First we prove that $\mathcal{T}$ maps $\overline{B_K}$ into $\overline{B_K}$. Indeed, for $v\in  B_K$ we have that
$$\|\mathcal{T}(v(t))\|_{L^1(B_R)} \le \|u_0\|_{L^1(B_R)} + T^{1/2}\sup_{0\le \tau\le T} (\|v(\tau)\|_{L^\infty(B_R)}+\mu)^{m-1}  \|v(\tau)\|_ {X}  \le K$$
Indeed, if  $6 \, T^{1/2}  (K+\mu)^{m-1}\le 1$ we have that    $\mathcal{T}$ is a strict contraction mapping in $\overline{B_K}$. The proof is as follows.
Let $u_1,u_2 \in \overline{B_K}$. Then
\begin{align*}
(\mathcal{T}(u_2)- \mathcal{T}(u_1))(x,t)&=\int_0^t \nabla e^{\delta (t-\tau) \Delta} \cdot (u_2(\tau)+\mu)^{m-1}\nabla (-\Delta)^{-1} \mathcal{L}_\epsilon^{1-s}[u_2-u_1](\tau) d\tau \\
&+  \int_0^t \nabla e^{\delta (t-\tau) \Delta} \cdot \left((u_2(\tau)+\mu)^{m-1}-(u_1(\tau)+\mu)^{m-1}\right)\nabla (-\Delta)^{-1} \mathcal{L}_\epsilon^{1-s}[u_1](\tau) d\tau.
\end{align*}
Then, for any $1\le p\le \infty$,
\begin{align*}
\|(\mathcal{T}(u_2)&-\mathcal{T}(u_1))(x,t)\|_{L^p(B_R)} \le \int_0^t \|\nabla e^{\delta (t-\tau) \Delta} \cdot (u_2(\tau)+\mu)^{m-1}\nabla (-\Delta)^{-1} \mathcal{L}_\epsilon^{1-s}[u_2-u_1](\tau)\|_{L^p(B_R)} d\tau \\ \nonumber
&+  \int_0^t \|\nabla e^{\delta (t-\tau) \Delta} \cdot \left((u_2(\tau)+\mu)^{m-1}-(u_1(\tau)+\mu)^{m-1}\right)\nabla (-\Delta)^{-1} \mathcal{L}_\epsilon^{1-s}[u_1](\tau) \|_{L^p(B_R)}d\tau.
\end{align*}
Using one again \eqref{estim:heat1} we get
\begin{multline}\label{LpContraction}
\|(\mathcal{T}(u_2)-\mathcal{T}(u))(\cdot,t)\|_{L^p(B_R)} \le \int_0^t
 (t-\tau)^{-\frac{1}{2}} \| (u_2+\mu)^{m-1}\nabla (-\Delta)^{-1} \mathcal{L}_\epsilon^{1-s}[u_2-u_1]\|_{L^p(B_R)}  (\tau) d\tau \\
+  \int_0^t(t-\tau)^{-\frac{1}{2}} \| \left((u_2+\mu)^{m-1}-(u_1+\mu)^{m-1}\right)\nabla (-\Delta)^{-1} \mathcal{L}_\epsilon^{1-s}[u_1] \|_{L^p(B_R)}(\tau) d\tau.
\end{multline}
For the first term we use the estimates of Lemma \ref{Lemma:nonlocal}, taking into account that $u_1,u_2$ are in fact supported in the ball,  to show that for $p\in \{1,\infty\}$ we have
\begin{align*}
 \| (u_2+\mu)^{m-1}\nabla (-\Delta)^{-1} \mathcal{L}_\epsilon^{1-s}[u_2-u_1]
\|_{L^p( B_R)}   &\le (\| u_2\|_{L^\infty(B_R)} +\mu)^{m-1}
\| \nabla (-\Delta)^{-1} \mathcal{L}_\epsilon^{1-s}[u_2-u_1]\|_{L^p( B_R)} \\
&\lesssim   (\| u_2\|_{L^\infty(B_R)} +\mu)^{m-1} \|u_2-u_1\|_ {X} .
\end{align*}
Similarly, for the second term in \eqref{LpContraction}, we use Lemma \ref{Lemma:nonlocal}, to get
\begin{align*}
& \| \left((u_2+\mu)^{m-1}-(u_1+\mu)^{m-1}\right)\nabla (-\Delta)^{-1} \mathcal{L}_\epsilon^{1-s}[u_1] \|_{L^1(\RN)} \\
&\qquad \qquad \qquad \qquad\le
\| (u_2+\mu)^{m-1}-(u_1+\mu)^{m-1}\|_{L^\infty(\RN)}
\| \nabla (-\Delta)^{-1} \mathcal{L}_\epsilon^{1-s}[u_1]\|_{L^1( \RN)}  \\
&\qquad \qquad\qquad \qquad \lesssim \| u_2-u_1\|_{L^\infty(B_R)} \cdot \max(\mu^{m-2}, (\|u_1\|_\infty +\mu)^{m-2}, (\|u_2\|_\infty +\mu)^{m-2} )
 \|u_1\|_ {X}.
\end{align*}
Summing up, if $6  T^{1/2}  (K+\mu)^{m-1}  \le 1$ we have that
\begin{equation*}
\begin{split}
\|(\mathcal{T}(u_2)-\mathcal{T}(u_1))(x,t)\|_{L^1(B_R)}
&\lesssim  t^{1/2}  \sup_{0<\tau <t } \| (u_2+\mu)^{m-1}\nabla (-\Delta)^{-1} \mathcal{L}_\epsilon^{1-s}[u_2-u_1] (\tau)\|_{L^1(\RN)} \\
 + &t^{1/2}  \sup_{0<\tau <t }  \| \left((u_2+\mu)^{m-1}-(u_1+\mu)^{m-1}\right)\nabla (-\Delta)^{-1} \mathcal{L}_\epsilon^{1-s}[u_1] (\tau)\|_{L^1(\RN)} \\
&\lesssim 6  T^{1/2}  (K+\mu)^{m-1} \| u_2-u_1\|_{X}  \le  \| u_2-u_1\|_{X}.
\end{split}
\end{equation*}
The estimate of $\|(\mathcal{T}(u_2)-\mathcal{T}(u_1))(x,t)\|_{L^\infty(B_R)}$ follows similarly by taking $p=\infty$ in \eqref{LpContraction} and using Lemma \ref{Lemma:nonlocal}. Thus,  the mapping $\mathcal{T}$ is a strict contraction on $\overline{B_K}$ if  $6  T^{1/2}  (K+\mu)^{m-1}  \le 1$:
\begin{align*}
\|(\mathcal{T}(u_2)-\mathcal{T}(u_1))\|_{C([0,T]:X)} < \frac{1}{2} \|u_2-u_1\|_{C([0,T]:X)} .
\end{align*}

 \end{proof}

 \subsubsection{Local in time improved regularity of the fixed point and strong solution}
 Using the formulation of $u=\mathcal{T}(v)$ as a strong solution of the initial and lateral data problem for \eqref{eq:HEright},  multiplying by $u$,  and integrating, we get the identity
 \[
 \frac{1}{2}\int_{B_R} |u(T)|^2d x+ \delta \int_0^T\int_{B_R} |\nabla u(t)|^2 dx dt= \int_0^T \int_{B_R} G(v(t))\cdot \nabla u(t) dx dt+ \frac{1}{2}\int_{B_R}|\widehat{u_0}|^2 dx
 \]
 We now use Lemma \ref{lem:GX} so that $\|G(v)\|_{X_T}\leq C(\|v\|_{L^\infty(Q_T)}) \|v\|_{X_T}$ and since we take $\|v\|_{X_T}\leq K$ then $\|G(v)\|_{X_T}\leq C(K)$. Also the last term is bounded by $C(K)$. Using now Young's inequality on the first term of the right-hand side to absorb one term into the term with $|\nabla u(t)|^2$,  we get
 \[
  \int_0^T\int_{B_R} |\nabla u(t)|^2 dx dt\leq C(K,\delta),
 \]
 which means that in all the steps of this iteration $\nabla u\in L^2(Q_T)$ with a  uniform bound depending on $K$ and $\delta$ since $G(v)$ is uniformly bounded in $X_T$. In the limit of the iteration process that leads to the fixed point, we conclude that such a fixed point $u\in L^2([0,T]: H^1_0(B_R))$ with a uniform bound estimated by $K$. It is now easy to see that $u$ is indeed a strong solution of \eqref{ProblemEpsMuDeltaR}. This is what we take as $U_1$. Note that, for the moment, $U_1$ is only defined locally in time. In order to prove existence for all times, we need some properties that will be derived next.

\subsubsection{Nonnegativity and \texorpdfstring{$L^p$}{} decay of the local in time solution} Standard arguments shows that if $\widehat{u}_0$ is nonnegative, then $U_1$ is also nonnegative. Similarly, we get that the $L^\infty$ norm of the solution is nonincreasing. Moreover, given $T$ prescribed by Proposition \ref{Prop:Mild}, we have for all $0<t<T$ the following estimates for the $L^p$ of the strong solution $U_1$:
\begin{align*}
& \displaystyle\frac{d}{dt}\int_{B_R} U_1^p(x,t)dx=p\int_{B_R} U_1^{p-1} (U_1)_t dx = \\
&=- p\delta\int_{B_R}\nabla (U_1^{p-1})\cdot\nabla U_1 dx  - p \int_{B_R} \nabla U_1^{p-1} (U_1+\mu)^{m-1}\cdot \nabla (-\Delta)^{-1} \mathcal{L}^{1-s}_\epsilon [U_1]dx\\
&=- \frac{4(p-1)\delta}{p}\int_{B_R}\left|\nabla (U_1^{p/2})\right|^2 dx  - p (p-1)\int_{B_R}  U_1^{p-2} (U_1+\mu)^{m-1} \nabla U_1  \cdot \nabla (-\Delta)^{-1} \mathcal{L}^{1-s}_\epsilon [U_1]dx.
\end{align*}
The boundary terms are $0$ since $U_1=0$ on $\RN \setminus  B_R$. We analyze the second term:
\begin{align*}&\int_{B_R}  U_1^{p-2} (U_1+\mu)^{m-1} \nabla U_1 \cdot \nabla (-\Delta)^{-1} \mathcal{L}^{1-s}_\epsilon [U_1]dx =
\int_{B_R}  \nabla \psi(U_1) \cdot \nabla (-\Delta)^{-1} \mathcal{L}^{1-s}_\epsilon [U_1]dx \\
&=\int_{B_R}  \psi(U_1) (-\Delta) (-\Delta)^{-1} \mathcal{L}^{1-s}_\epsilon [U_1]dx = \int_{B_R}  \psi(U_1) \mathcal{L}^{1-s}_\epsilon [U_1]dx \\
&  = \int_{\RN}  \psi(U_1) \mathcal{L}^{1-s}_\epsilon [U_1]dx  \ge  \int_{\RN} |(\mathcal{L}^{1-s}_\epsilon)^{\frac{1}{2}} [ \Psi (U_1)]|^2dx
\end{align*}
We have used the generalized Stroock-Varopoulos Inequality \eqref{GenStroockVarAprox} in the following context:  the functions $\psi$ and $\Psi$ are such that $\psi'=(\Psi')^2$ and $\nabla \psi(U_1)=U_1^{p-2} (U_1+\mu)^{m-1}\nabla U_1$. The precise definition of these functions is given by
\[
\displaystyle \psi(z)=\int_0^z \zeta^{p-2}(\zeta+\mu)^{m-1}d\zeta, \quad \displaystyle \Psi(z)=\int_0^z \zeta^{\frac{p-2}{2}}(\zeta+\mu)^{\frac{m-1}{2}}d\zeta.
\]
We obtain the following $L^p$-energy estimate:
\begin{align}\label{energy6}
\int_{B_R} u_0^p(x)dx - \int_{B_R} U_1^p(x,t)dx &= \frac{4(p-1)\delta}{p}\int_0^t \int_{B_R} \left|\nabla (U_1^{p/2})\right|^2 dxdt \\
\nonumber &\qquad + p (p-1) \int_0^t \int_{B_R}  \psi(U_1) \mathcal{L}^{1-s}_\epsilon [U_1]dxdt,
\end{align}
and then
\begin{equation*}\begin{split}
\int_{B_R} u_0^p(x)dx &\ge  \int_{B_R} U_1^p(x,t)dx + \\
&  + \frac{4(p-1)\delta}{p}\int_0^t \int_{B_R} \left|\nabla (U_1^{p/2})\right|^2 dxdt + p (p-1) \int_0^t \int_{B_R}  |(\mathcal{L}^{1-s}_\epsilon)^{\frac{1}{2}} [\Psi (U_1)]|^2 dx dt.
\end{split}
\end{equation*}
As a consequence, we get that $(\mathcal{L}^{1-s}_\epsilon)^{\frac{1}{2}} [\Psi (U_1)] \in L^2(Q_T)$ for $u_0 \in L^p(\RN).$

\medskip

We also get the so-called \textbf{second energy estimate}:
\begin{equation*}
\begin{split}
&\frac{1}{2} \frac{d}{dt}\int_{B_R}|\SqLop [U_1]|^2 dx =\int_{B_R}\SqLop[U_1] \cdot \SqLop [(U_1)_t] dx\\
&= \int_{B_R}(-\Delta)^{-1} \mathcal{L}^{1-s}_\epsilon [U_1] \, (U_1)_t dx \\
&= \delta \int_{B_R}(-\Delta)^{-1} \mathcal{L}^{1-s}_\epsilon [U_1]  \Delta U_1  dx
+ \int_{B_R}(-\Delta)^{-1} \mathcal{L}^{1-s}_\epsilon [U_1] \nabla \cdot((U_1+\mu)^{m-1} \nabla (-\Delta)^{-1} \mathcal{L}^{1-s}_\epsilon [U_1]) dx \\
&= - \delta \int_{B_R}\left|(\mathcal{L}^{1-s}_\epsilon)^{\frac{1}{2}} [U_1]\right|^2  dx -
\int_{B_R} (U_1+\mu)^{m-1} | \nabla (-\Delta)^{-1} \mathcal{L}^{1-s}_\epsilon [U_1]|^2 dx .
\end{split}
\end{equation*}
Therefore, the quantity $\displaystyle \int_{B_R}|(-\Delta)^{-\frac{1}{2}}(\mathcal{L}^{1-s}_\epsilon)^{\frac{1}{2}} [U_1](x,t)|^2 dx $ is non-increasing in  $t$ and we have that
\begin{equation}\label{SecondEnergyU1}
\begin{split}
&\frac{1}{2} \int_{B_R}\left|\SqLop [u_0]\right|^2 dx = \frac{1}{2} \int_{B_R}\left|\SqLop [U_1(t)]\right|^2 dx  \\ &+
 \delta \int_0^t\int_{B_R}\left|(\mathcal{L}^{1-s}_\epsilon)^{\frac{1}{2}}[U_1]\right|^2  dx dt+
\int_0^t \int_{B_R}(U_1+\mu)^{m-1} \left| \nabla (-\Delta)^{-1} \mathcal{L}^{1-s}_\epsilon [U_1]\right|^2 dx  dt . \end{split}
\end{equation}

\subsubsection{Global-in-time solution}
The preceding analysis shows that the $L^p$ norm of the solution constructed in a finite time interval $[0,T]$ does not increase with time for any $p\in[1,\infty]$ by \eqref{energy6}. Therefore, we can continue the solution in a new time interval of the same length  with initial data $U_1(x,T)$, thus obtaining a solution in $[0,2T]$. We iterate this process to get a global in time solution.

We conclude the results obtained so far in the following theorem.

\begin{theorem}\label{ThmU1}
Let $s\in (0,1)$, $1<m<\infty$ and $N\geq1$. There exists a weak solution $U_1$ of Problem \eqref{ProblemEpsMuDeltaR} with initial data $\widehat{u}_0$. Moreover,   $U_1$ is a strong solution,   satisfies the $L^p$-energy estimate \eqref{energy6}, the second energy estimate \eqref{SecondEnergyU1}, and also
\begin{enumerate}
\item \textbf{(Decay of total mass)} For all $0<t<T$ we have
$\displaystyle{
\int_{B_R}U_1(x,t)dx \le\int_{B_R}u_0(x)dx.
}$
\item \textbf{($L^{\infty}$-estimate) } For all $0<t<T$ we have  $||U_1(\cdot,t)||_\infty\leq ||u_0||_\infty$.
\end{enumerate}
\end{theorem}

\begin{remark} {\rm
In Sections \ref{Subsec:Eps}, \ref{subsec:R}, \ref{Subsec:mu} and \ref{Subsec:delta} we will only consider $s\in (0,\frac{1}{2})$ when $N=1$ since the operator $(-\Delta)^{-s}$ is not well defined out of this range. We will devote Section \ref{SectN1} to comment   on   how  to deal with the  case $N=1$, $s\in [\frac{1}{2},1)$.}
\end{remark}

\subsection{Limit as \texorpdfstring{$\epsilon \to 0$}{}}\label{Subsec:Eps}
Let $U_1$ be a weak solution of problem \eqref{ProblemEpsMuDeltaR} with parameters $\delta,\mu, R>0$ fixed from the beginning. We will prove that $\lim_{\epsilon \to 0}U_1=U_2$, where $U_2$ is a weak solution of the problem
\[
\left\{
\begin{array}{ll}
(U_2)_t= \delta \Delta U_2 +\nabla \cdot((U_2+\mu)^{m-1}\nabla (-\Delta)^{-s} U_2)&\text{for } (x,t)\in B_R\times (0,T),\\
U_2(x,0)=\widehat{u}_0(x) &\text{for } x \in B_R,\\
U_2(x,t)=0 &\text{for } x\in \partial B_R, \ t\geq 0.
\end{array}
\right.
\tag{$P_{\delta\mu R}$}\label{ProblemMuDeltaR} \]
Moreover, we will also prove that $U_2$  inherits most of the properties of $U_1$. In particular, we will prove that $U_2$ can be extended by $0$ to $\RN \setminus B_R$, this allowing the definition of $(-\Delta)^{-s} U_2$.

\subsubsection{Existence of a limit. Compactness estimate I}\label{Subsec:CompactEps}
\noindent $\textbf{I.}$ Using the energy estimate \eqref{energy6} with $p=2$ we obtain that $U_1 \in L^2(0,T:H_0^1(B_R))$.

%
%\noindent $\textbf{II.}$  The second energy estimate \eqref{SecondEnergyU1} gives us that $(\mathcal{L}^{1-s}_\epsilon)^{\frac{1}{2}} [U_1] \in L^2(0,T:L^2(B_R))$. Therefore $U_1 \in L^2(0,T:H^{1-s}_\epsilon (B_R))$ where $H^{1-s}_\epsilon (B_R)$ is defined as in \eqref{spaceHs}. Consequently, for any $0<\epsilon<\epsilon_0\leq1$, we have
%$$U_1 \in L^2(0,T:H^{1-s}_{\epsilon_0} (B_R))$$
%uniformly bounded on $\epsilon$.
%
%
%
%

\noindent $\textbf{II.}$ Estimates on the derivative $(U_1)_t$. We use the equation $$(U_1)_t=\delta \Delta U_1 +\nabla \cdot((U_1+\mu)^{m-1} \nabla (-\Delta)^{-1} \mathcal{L}^{1-s}_\epsilon [U_1]).$$ The $H_0^1$ estimate of \eqref{energy6} ensures that $\delta \Delta U_1  \in L^2(0,T:H^{-1}(B_R))$. The second energy estimate \eqref{SecondEnergyU1} implies that
$$
(U_1+\mu)^{\frac{m-1}{2}} \nabla (-\Delta)^{-1} \mathcal{L}^{1-s}_\epsilon [U_1] \in L^2(0,T:L^2(B_R)).
$$ Since also  $U_1 \in L^\infty((0,T)\times B_R)$  then this  implies that $\nabla \cdot((U_1+\mu)^{m-1} \nabla (-\Delta)^{-1} \mathcal{L}^{1-s}_\epsilon [U_1]) \in L^2(0,T:H^{-1} (B_R))$. We conclude that
$$(U_1)_t \in L^2(0,T:H^{-1} (B_R)).$$

\noindent $\textbf{III.}$ We apply the compactness criteria of Simon (see Lemma \ref{ConvSimon1} in Section \ref{sec:app}) in the context of
\[
H^{1}_{0}(B_R)\subset L^2(B_R) \subset H^{-1}(B_R),
\]
where the left hand side inclusion is compact. We conclude that the family of approximate solutions $\{U_1\}_{\epsilon>0}$ is relatively compact in $L^2(0,T:L^2(B_R))$. Therefore, there exists a limit $(U_1)_{\epsilon,\delta, \mu, R} \to (U_2)_{\delta, \mu, R}$ as $\epsilon \to 0$ in $L^2(0,T:L^2(B_R)),$ up to subsequences. Note that, since $(U_1)_\epsilon$ is a family of positive functions defined on $B_R$ and extended to $0$ in $\RN\setminus B_R$, then the limit $U_2=0$ a.e. on $\RN\setminus B_R$.  We obtain that
\begin{equation}\label{convU1U2}
U_1\stackrel{\epsilon \to 0}{\longrightarrow}U_2   \quad \text{in }L^2(0,T:L^2(B_R))=L^2(B_R\times (0,T)).
\end{equation}

\subsubsection{The limit \texorpdfstring{$U_2$}{} is a solution of the new problem \texorpdfstring{\eqref{ProblemMuDeltaR}}{}}

We pass to the limit as $\epsilon \to 0$ in the definition \eqref{weaksolAprox} of a weak solution of Problem \eqref{ProblemEpsMuDeltaR} and we prove that the limit $U_2$ found in \eqref{convU1U2} is a weak solution of Problem \eqref{ProblemMuDeltaR}. The convergence of the first integral in \eqref{weaksolAprox} is justified by \eqref{convU1U2} since
\begin{equation}\label{weaklimiteps}
\left|\int_0^T\int_{B_R} (U_1-U_2)(\phi_t+\delta \Delta \phi)dxdt\right|\leq ||U_1-U_2||_{L^2(B_R\times (0,T))}||\phi_t+\delta \Delta \phi||_{L^2(B_R\times (0,T))}.
\end{equation}
To prove convergence of the second integral in \eqref{weaksolAprox} we argue as follows. Using \eqref{convU1U2} and the $L^{\infty}$-decay estimate from Theorem \ref{ThmU1} we get that
\begin{equation}\label{L2convPositPower}
 (U_1 + \mu)^{m-1} \to (U_2+\mu)^{m-1}  \quad \text{in }L^2(B_R\times (0,T)).
\end{equation}
The convergence of the nonlocal gradient term in \eqref{weaksolAprox} is proved in the following lemma.

\begin{lemma}\label{LemmaWeakConv}
We have that
 $$\nabla (-\Delta)^{-1} \mathcal{L}^{1-s}_\epsilon [U_1]\stackrel{\epsilon \to 0}{\rightharpoonup}\nabla (-\Delta)^{-s} U_2 \quad \mbox{ in } \quad L^2(B_R\times(0,T)) .$$
\end{lemma}
\begin{proof}
\noindent \textbf{I. There exists a weak limit.} From the second energy estimate \eqref{SecondEnergyU1} we note that
\begin{equation*}
\begin{split}
\left\|\nabla (-\Delta)^{-1} \mathcal{L}^{1-s}_\epsilon [U_1]\right\|_{L^2(B_R\times(0,T))}&=\left\|\frac{ (U_1+\mu)^{\frac{m-1}{2}}}{ (U_1+\mu)^{\frac{m-1}{2}}}\nabla (-\Delta)^{-1} \mathcal{L}^{1-s}_\epsilon [U_1]\right\|_{L^2(B_R\times(0,T))}\\
&\leq  \mu^{-\frac{m-1}{2}} \left\|(U_1+\mu)^{\frac{m-1}{2}}\nabla (-\Delta)^{-1} \mathcal{L}^{1-s}_\epsilon [U_1]\right\|_{L^2(B_R\times(0,T))}\leq C.
\end{split}
\end{equation*}
Then, Banach-Alaoglu Theorem ensures that there exists a subsequence such that
\[
\nabla (-\Delta)^{-1} \mathcal{L}^{1-s}_\epsilon [U_1] \, \stackrel{\epsilon \to 0}{\rightharpoonup}\, v \mbox{ in } L^2(B_R\times(0,T)) .
\]

\noindent  \textbf{II. Identifying the limit in the sense of distributions.} Now, we will prove that $$\nabla (-\Delta)^{-1} \mathcal{L}^{1-s}_\epsilon [U_1]\stackrel{\epsilon \to 0}{\longrightarrow}\nabla (-\Delta)^{-s} U_2$$ in distributions. More exactly, we will prove that
$$
\int_0^T\int_{B_R}  (-\Delta)^{-1} \mathcal{L}^{1-s}_\epsilon [U_1]\nabla \phi dxdt  \stackrel{\epsilon \to 0}{\longrightarrow} \int_0^T\int_{B_R}  (-\Delta)^{-s} U_2\nabla \phi dxdt
$$
for all $\phi\in C_c^\infty(B_R\times (0,T))$. We estimate the difference of the two integrals above as follows,
\begin{equation*}
\begin{split}
 I_\epsilon &= \int_0^T\int_{B_R} \left( (-\Delta)^{-1} \mathcal{L}^{1-s}_\epsilon [U_1] - (-\Delta)^{-s}U_1\right) \nabla \phi dxdt + \int_0^T\int_{B_R} \left( (-\Delta)^{-s} U_1 - (-\Delta)^{-s} U_2 \right) \nabla \phi dxdt \\
 &=I_{1,\epsilon} + I_{2,\epsilon}.
\end{split}
\end{equation*}
The first integral converges to $0$ as a consequence of the approximation of $(-\Delta)^{-s}$ in the sense derived in Lemma \ref{LemmaAproxInverseFracLap} a). Note that $U_1$ is changing with $\epsilon$, but we have the uniform bound $\|U_1\|_{1}\leq \|u_0\|_{1}$ which ensures that Lemma \ref{LemmaAproxInverseFracLap} can still being applied.
For the second integral we write
\begin{equation*}
\begin{split}
 I_{2,\epsilon} &= \int_0^T\int_{\RN} \left(  U_1 -  U_2 \right) \nabla (-\Delta)^{-s} \phi \, dxdt  \\
 &=\int_0^T\int_{B_{\rho}} \left(  U_1 -  U_2 \right) \nabla (-\Delta)^{-s} \phi \, dxdt + \int_0^T\int_{\RN \setminus B_{\rho}} \left(  U_1 -  U_2 \right) \nabla (-\Delta)^{-s} \phi \, dxdt \\
\end{split}
\end{equation*} for a $\rho$ to be chosen later. Now fix $\eta>0$. Then
\begin{equation}\label{bleble}
\begin{split}
 \int_0^T\int_{\RN \setminus B_{\rho}}  &  \left|  U_1 -  U_2 \right| | \nabla (-\Delta)^{-s} \phi | \, dxdt \le \|U_1 -U_2\|_{L^2(\RN \times (0,T))} \| \nabla (-\Delta)^{-s} \phi \|_{L^2 ( (\RN \setminus B_{\rho})\times (0,T))} \\
 &\le 2 T \|u_0\|_{L^2(\RN)}\| \nabla (-\Delta)^{-s} \phi \|_{L^2 (  (\RN \setminus B_{\rho}) \times  (0,T)) }.
\end{split}
\end{equation}
Since $\nabla (-\Delta)^{-s} \phi  \in L^2 (\RN \times  (0,T)) $ then we can choose  $\rho$ large enough such that\\ $\| \nabla (-\Delta)^{-s} \phi \|_{L^2 ( (\RN \setminus B_{\rho})\times  (0,T))} \le \eta/2$. On the other hand
\begin{equation*}
\begin{split}\int_0^T\int_{B_{\rho}} &\left(  U_1 -  U_2 \right) \nabla (-\Delta)^{-s} \phi \, dxdt \le \|U_1 -U_2\|_{L^2( B_{\rho} \times  (0,T) )} \| \nabla (-\Delta)^{-s} \phi \|_{L^2  ( B_{\rho}\times  (0,T) )}.
\end{split}
\end{equation*}
We choose $\epsilon$ small  enough such that $\|U_1 -U_2\|_{L^2( B_{\rho} \times  (0,T))} \le \eta/2$. Therefore $ I_{2,\epsilon}\to 0 $ as $ \epsilon \to 0. $

Note that we could have fixed $\rho=R$ and then the first integral in \eqref{bleble} is identically zero since $U_1$ and $U_2$ are supported in $B_R$. We keep the splitting here since it will be needed to estimate $I_{2,\epsilon}$ in the limit as $R\to \infty$ (see Section \ref{sec:blibli}).
\end{proof}

To conclude this part, we use the following: given two sequences $f_\epsilon \rightharpoonup f$ in $L^2$ and $g_\epsilon \to g$ strongly in $L^2$, then the scalar product converges $\displaystyle \int f_\epsilon g_\epsilon \,dx \to \int fg \,dx $. Then \eqref{L2convPositPower} together with Lemma \ref{LemmaWeakConv} implies that
$$
\int_0^T\int_{B_R}  (U_1+\mu)^{m-1} \nabla (-\Delta)^{-1} \mathcal{L}^{1-s}_\epsilon [U_1]\nabla \phi dxdt \stackrel{\epsilon \to 0}{\longrightarrow} \int_0^T\int_{B_R}  (U_2+\mu)^{m-1} \nabla (-\Delta)^{-s} U_2\nabla \phi dxdt.
$$

\subsubsection{Passing to the limit  in the \texorpdfstring{$L^p$}{} energy estimate \texorpdfstring{\eqref{energy6}}{}}\label{sec:bloblo}
We have that
\begin{align*}\int_{B_R}   \psi(U_1) \mathcal{L}^{1-s}_\epsilon [U_1]dx &= \int_{B_R}  \int_{B_R}  \psi (U_1(x)) \frac{U_1(x)-U_1(y)}{(|x-y|^2+\epsilon^2)^{\frac{N+2(1-s)}{2}}  } dx dy \\
&=\frac{1}{2}\int_{B_R}  \int_{B_R}  \left(\psi (U_1(x)) - \psi(U_1(y))\right) \frac{U_1(x)-U_1(y)}{(|x-y|^2+\epsilon^2)^{\frac{N+2(1-s)}{2}}  } dx dy.
\end{align*}
Let
$$
G_{\epsilon} (x,y):= \frac{1}{2} \left(\psi (U_1(x)) - \psi(U_1(y))\right) \frac{U_1(x)-U_1(y)}{(|x-y|^2+\epsilon^2)^{\frac{N+2(1-s)}{2}}  },
$$
and
$$
G (x,y):= \frac{1}{2} \left(\psi (U_2(x)) - \psi(U_2(y))\right) \frac{U_2(x)-U_2(y)}{|x-y|^{N+2(1-s)}  }.
$$
Note that  $G_{\epsilon}(x,y)\ge 0$ since $\psi$ is a non-decreasing function.
Also, $\int_{\RN}\int_{\RN}  G_{\epsilon}(x,y)\le C $ uniformly in $\epsilon>0$.
Since $U_1 \to U_2$  as $\epsilon \to 0$ pointwise a.e. in $x\in B_R$ then $G_{\epsilon}(x,y) \to G (x,y)$ a.e. $x,y\in \RN$.   We can pass to the limit $\epsilon \to 0$ in the last term of the energy estimate \eqref{energy6} according to the Fatou's Lemma
\[
\lim_{\epsilon \to 0}\int_0^t\int_{B_R}\int_{B_R}  G_{\epsilon}(x,y)dxdydt \ge \int_0^t\int_{B_R}\int_{B_R}  G(x,y) dx dy =  \int_0^t \int_{B_R}  \psi(U_2) (-\Delta)^{1-s} U_2dxdt.
\]

Now we pass to the limit in the $H^1$ term. The $L^p$ energy estimate \eqref{energy6} shows that $U_1^{p/2}$ is uniformly bounded in $L^2(0,T:H^1_0(B_R))$, therefore there exists a weak limit $w$ in $L^2(0,T:H^1_0(B_R))$. Since $H^1_0(B_R) \subset L^2(B_R)$ with continuous inclusion, then $U_1^{p/2} \to w$  in $L^2(B_R\times(0,T))$. By \eqref{convU1U2} we know that $U_1 \to U_2$ in $L^2(B_R\times(0,T))$. For $p>2$ we deduce that $U_1^{p/2} \to U_2^{p/2}$  in $L^2(B_R\times (0,T))$ and then we identify the limit  $w=U_2^{p/2}$. The weak lower semi-continuity of the $\| \cdot \|_{H_0^1(B_R)}$  norm implies that
\begin{equation*}
 \liminf_{\epsilon \to 0 }\int_0^t \int_{B_R} \left|\nabla (U_1^{p/2})\right|^2 dxdt \ge  \int_0^t \int_{B_R} \left|\nabla (U_2^{p/2})\right|^2 dxdt.
 \end{equation*}
 We used the fact that the norm of a Hilbert space is weakly semi-continuous. A similar idea will be employed to pass to the limit also in the integrals in the second energy estimate \eqref{SecondEnergyU1}.

\subsubsection{Passing to the limit  in the second energy estimate \texorpdfstring{\eqref{SecondEnergyU1}}{}}

The first two  terms involve integral operators, so the continuous inclusion $ L^2(B_R)\subset H^{-s/2}(B_R)$ together with  \eqref{convU1U2} allow to pass to the limit. For the third one we use the argument given in Section \ref{sec:bloblo} in the particular case $\psi(U_1)=U_1$. For the last term we have to prove the following inequality
\begin{equation*}
 \liminf_{\epsilon \to 0}\int_0^t \int_{B_R}  (U_1+\mu)^{m-1} \left| \nabla  (-\Delta)^{-1} \mathcal{L}_{\epsilon}^{1-s}[ U_1]\right|^2 dx  dt    \ge \int_0^t \int_{B_R}  (U_2+\mu)^{m-1} \left| \nabla  (-\Delta)^{-s} U_2\right|^2 dx  dt .
\end{equation*}
This is a consequence of the fact that the $L^2$ norm is weakly lower semi-continuous and $(U_1+\mu)^{\frac{m-1}{2}} \nabla  (-\Delta)^{-1} \mathcal{L}_{\epsilon}^{1-s}[ U_1] \rightharpoonup (U_2+\mu)^{\frac{m-1}{2}} \nabla (-\Delta)^{-s} U_2  $  in $L^2 (B_R \times (0,t))$.
Indeed, we have that
$$\int_0^t\int_{B_R}  (U_1+\mu)^{\frac{m-1}{2}} \nabla (-\Delta)^{-1} \mathcal{L}^{1-s}_\epsilon [U_1] \phi dxdt \stackrel{\epsilon \to 0}{\longrightarrow} \int_0^t\int_{B_R}  (U_2+\mu)^{\frac{m-1}{2}} \nabla (-\Delta)^{-s} U_2  \phi dxdt$$
for every $\phi \in L^2 (B_R \times (0,t))$. This is because $(U_1+\mu)^{\frac{m-1}{2}}  \phi \rightarrow  (U_2+\mu)^{\frac{m-1}{2}}  \phi$ in $L^2(B_R \times (0,t))$ (using the Dominated Convergence Theorem) and $\nabla (-\Delta)^{-1} \mathcal{L}^{1-s}_\epsilon [U_1] \rightharpoonup \nabla (-\Delta)^{-s} U_2$ in $L^2(B_R\times(0,t))$ by Lemma \ref{LemmaWeakConv}.

From now on, we do not need to consider a smooth initial data $\widehat{u}_0 \sim u_0$. We sum up the results of this section in the following theorem.

\begin{theorem}
Let $s\in(0,1)$, $1<m<\infty$, $N\geq1$. There exists a weak solution $U_2$ of Problem \eqref{ProblemMuDeltaR} with initial data $u_0\in L^1(\R^N)\cap L^\infty(\R^N)$. Moreover $U_2$ has the following properties:
\begin{enumerate}
\item \textbf{(Decay of total mass)} For all $0<t<T$ we have
$\displaystyle{
\int_{B_R}U_2(x,t)dx \le\int_{B_R}u_0(x)dx.
}$

\item \textbf{($L^{\infty}$ estimate) } For all $0<t<T$ we have  $||U_2(\cdot,t)||_\infty\leq ||u_0||_\infty$.

\item \textbf{($L^p$ energy estimate)} For all $1<p<\infty$ and $0<t<T$ we have
\begin{equation}\label{energyU2}
\begin{split}
 \int_{B_R} U_2^p(x,t)dx  & +\frac{4(p-1)\delta}{p}\int_0^t \int_{B_R} \left|\nabla (U_2^{p/2})\right|^2 dxdt \\\
 &+ p (p-1) \int_0^t \int_{B_R}  \psi(U_2) (-\Delta)^{1-s} U_2dxdt \le \int_{B_R} u_0^p(x)dx.\end{split}
\end{equation}

\item \textbf{(Second energy estimate)} For all $0<t<T$  we have
\begin{equation}\label{SecondEnergyU2}
\begin{split}
 \frac{1}{2} \int_{B_R}& \left|(-\Delta)^{-\frac{s}{2}} U_2(t)\right|^2  dx   +\delta \int_0^t\int_{B_R}\left|(-\Delta)^{\frac{1-s}{2}}[U_2]\right|^2  dx dt  \\
&
+\int_0^t \int_{B_R}  (U_2+\mu)^{m-1} \left| \nabla  (-\Delta)^{-s} U_2(t)\right|^2 dx  dt    \le \frac{1}{2} \int_{B_R} \left|(-\Delta)^{-\frac{s}{2}} u_0\right|^2 dx. \end{split}
\end{equation}
\end{enumerate}
\end{theorem}

\subsection{Limit as \texorpdfstring{$R\to \infty$}{}}\label{subsec:R}

In this section we argue for weak solutions $U_2=(U_2)_R$ of Problem \eqref{ProblemMuDeltaR}. The energy estimates \eqref{energyU2} and \eqref{SecondEnergyU2}  will give us sufficient information to accomplish the limits.

\subsubsection{Existence of a limit}\label{SectExistR}

We remark that the integrals in $B_R$ can be interpreted like integrals on whole $\RN$ since we have chosen $U_2$ to be zero outside $B_R$. Moreover, we can get, from the energy estimates  \eqref{energyU2} and \eqref{SecondEnergyU2}, upper bounds which are independent on $R$. Note that the compactness technique used (see Lemma \ref{ConvSimon1}) requires compact embeddings, which motivates us to work on bounded domains.

\noindent \textbf{I. Local existence of a limit. }Let $\rho>0$  and consider the ball $B_\rho\subset\RN$. From \eqref{energyU2} with $p=2$ we get that $U_2 \in L^2(0,T:H^1(B_{\rho}))$ uniformly in $R>0$ and then $ \delta \Delta U_2 \in L^2(0,T: H^{-1}(\RN))$. Also, \eqref{SecondEnergyU2} gives $U_2 \in L^2(0,T: H^{1-s}(B_{\rho}))$.  From \eqref{SecondEnergyU2} we get that $\nabla \cdot((U_2+\mu)^{m-1}\nabla (-\Delta)^{-s} U_2) \in L^2(0,T:H^{-1}(\RN))$. Applying Lemma \ref{ConvSimon1} in the context
\[
H^{1-s}(B_\rho)\subset L^2(B_\rho)\subset H^{-1}(B_\rho),
\]
and noting that the left hand side inclusion is compact, we obtain that  there exists a limit function $V_\rho\in L^2(B_\rho\times(0,T))$ such that, up to sub-sequences,
\begin{equation}\label{convU2U3loc}
U_2  \to V_\rho\quad \text{as} \quad R \to \infty \quad \text{in} \quad L^2(B_{\rho} \times (0,T)).
\end{equation}

\noindent \textbf{II. Finding a global limit. }In order to define a global limit in  $L^2(\RN \times (0,T))$ we adapt the classical covering plus diagonal argument.
Let $\bigcup_{k=1}^\infty B_{\rho_k}  $, with $(\rho_k)_{k=1}^{\infty} \subset \mathbb{R}_{\geq 0}$, be a countable covering of $\RN$.
By \eqref{convU2U3loc} we obtain there exists a subsequence $(R_{j})_{j=1}^\infty$ such that $U_2|_{B_{\rho_1}}  \to V_{\rho_1}$ as $ R_{j} \to \infty $ in $ L^2(B_{\rho_1} \times (0,T))$ and $V_{\rho_1}: B_{\rho_1} \to \mathbb{R}$.   Next, we perform a similar argument starting from the subsequence $(R_{j})_{j=1}^\infty$ and $U_2|_{B_{\rho_2}}$ to get that there exists a sub-subsequence $(R_{j_k})_{k=1}^\infty \subset (R_j)_{j=1}^\infty$ such that   $U_2|_{\rho_2}  \to V_{{\rho_2}}$ as $ R_{j_k} \to \infty $ in $ L^2(B_{\rho_2} \times (0,T))$ and $V_{\rho_2}: B_{\rho_2} \to \mathbb{R}$.  It is clear that $V_{\rho_1}=V_{\rho_2}$ in  $B_{\rho_1}\cap B_{\rho_2}$. The argument continues for the remaining balls $B_{\rho_3}, B_{\rho_4},...$ . In the end we define the function $V: \RN \to \mathbb{R}$ such that $V|_{B_{\rho_k}}=V_{\rho_k}$ for $k\in \mathbb{N}_{>0}$. We denote this limit $U_3$ for better organization. Therefore, up to subsequences,
$$
U_2 \to U_3 \text{ as } R \to \infty \text{ in } L^2(0,T:L^2_{\textup{loc}}(\RN)).
$$
In particular, this implies $U_2 \to U_3$ as $R \to \infty$ a.e. in $\RN$.
 We recall that the functions $U_2$ are extended by $0$ in $\RN \setminus B_R$ and then, by the energy estimate \eqref{energyU2}, we have that  $\int_{\RN} U_2^2 dx$ is uniformly bounded in $R>0$. Then, by Fatou's Lemma we get that $U_3\in L^2(\R^N\times(0,T))$ since
$$\liminf_{R\to \infty} \int_0^T\int_{\RN }(U_2)^2 dxdt \ge  \int_0^T\int_{\RN }(U_3)^2 dxdt .$$

\subsubsection{The limit \texorpdfstring{$U_3$}{} is a solution of the new problem \texorpdfstring{\eqref{ProblemMuDelta}}{}}\label{sec:blibli}

Similarly, one can prove that  $U_3$ is a weak solution of Problem \eqref{ProblemMuDelta}:
\[
\left\{
\begin{array}{ll}
(U_3)_t= \delta \Delta U_3 +\nabla \cdot((U_3+\mu)^{m-1}\nabla (-\Delta)^{-s}U_3)&\text{for } (x,t)\in \RN \times (0,T),\\
U_3(x,0)=\widehat{u}_0(x) &\text{for } x \in \RN.
\end{array}
\right.
\tag{$P_{\delta\mu}$}\label{ProblemMuDelta} \]
The test functions used in Subsection \ref{Subsec:CompactEps} are compactly supported so the arguments perfectly work here. Let $\phi$ be a suitable test function supported in a ball $B_\rho$ for some $\rho>0$. For the convergence of the nonlinear term we use that
\begin{equation*}
(U_2 + \mu)^{m-1} \to (U_3+\mu)^{m-1}  \quad \text{in }L^2(B_\rho\times (0,T))\text{ as }R \to +\infty,
\end{equation*}
and
\begin{equation}\label{weakconvgradU2U3}
 \nabla (-\Delta)^{-s} U_2 \rightharpoonup \nabla (-\Delta)^{-s} U_3 \quad  \mbox{ in } L^2(B_\rho\times(0,T))\text{ as }R \to +\infty,
\end{equation}
where \eqref{weakconvgradU2U3} is proved as in  Lemma \ref{LemmaWeakConv}.

\subsubsection{Energy estimates}

All the energy estimates of $U_2$ can be written with integrals in $\RN$ and they provide upper bounds which independent on $R$. As before, the existence of a pointwise limit plus Fatou's Lemma allow us to pass to the limit as $R\to +\infty$. We refer to \cite{StanTesoVazJDE} for the proof of mass conservation. However, in Theorem \ref{ThmExistL1} we prove this result in the general setting of measure data.  We conclude with the following theorem.
\begin{theorem}
Let $s\in(0,1)$, $1<m<\infty$ and $N \ge 1$. There exists a weak solution $U_3$ of Problem \eqref{ProblemMuDelta} with initial data $u_0\in L^1(\R^N)\cap L^\infty(\R^N)$. Moreover, $U_3$ has the following properties:
\begin{enumerate}
\item \textbf{(Conservation of total mass)} For all $0<t<T$ we have
$\displaystyle{
\int_{\RN}U_3(x,t)dx =\int_{\RN}u_0(x)dx.
}$
\item \textbf{($L^{\infty}$ estimate) } For all $0<t<T$ we have  $||U_3(\cdot,t)||_\infty\leq ||u_0||_\infty$.
\item \textbf{($L^p$ energy estimate)} For all $1<p<\infty$ and $0<t<T$ we have
\begin{equation}\label{energyU3}
\begin{split}
 \int_{\RN} U_3^p(x,t)dx  & +\frac{4(p-1)\delta}{p}\int_0^t \int_{\RN} \left|\nabla (U_3^{p/2})\right|^2 dxdt \\\
 &+ p (p-1) \int_0^t \int_{\RN}  \psi(U_3) (-\Delta)^{1-s} U_3dxdt \le \int_{\RN} u_0^p(x)dx.\end{split}
\end{equation}
\item \textbf{(Second energy estimate)} For all $0<t<T$  we have
\begin{equation}\label{SecondEnergyU3}
\begin{split}
 \frac{1}{2} \int_{\RN}& \left|(-\Delta)^{-\frac{s}{2}} U_3(t)\right|^2  dx   +\delta \int_0^t\int_{\RN}\left|(-\Delta)^{\frac{1-s}{2}}[U_3]\right|^2  dx dt  \\
&
+\int_0^t \int_{\RN}  (U_3+\mu)^{m-1} \left| \nabla  (-\Delta)^{-s} U_3(t)\right|^2 dx  dt    \le \frac{1}{2} \int_{\RN} \left|(-\Delta)^{-\frac{s}{2}} u_0\right|^2 dx. \end{split}
\end{equation}
\end{enumerate}
\end{theorem}

\subsection{Limit as \texorpdfstring{$\mu \to 0$}{}}\label{Subsec:mu}
We remark that some of previous arguments can not be applied here since $(U_3+\mu)^{-(m-1)}$ may degenerate as $\mu \to 0$ close to the free boundary. Therefore we adapt the proof to overcome this issue.

\subsubsection{Existence of a limit}
The energy estimates \eqref{energyU3} and \eqref{SecondEnergyU3} gives us uniform upper bounds in $\mu$ which allows us to prove the existence of a limit
\begin{equation}\label{blublu}
U_3 \to  U_4 \quad \textup{as}\quad  \mu \to 0
 \quad \textup{in} \quad L^2_{\textup{loc}}(\RN \times (0,T)),
 \end{equation}
 using the same covering plus diagonal argument of Section \ref{subsec:R}.

\subsubsection{The limit \texorpdfstring{$U_4$}{} is a solution of the new problem \texorpdfstring{\eqref{ProblemDelta}}{}}
As before the compact support of the test functions allows us to prove that $U_4$ is in fact a weak solution of the problem:
\[
\left\{
\begin{array}{ll}
(U_4)_t= \delta \Delta U_4 +\nabla \cdot(U_4^{m-1}\nabla  (-\Delta)^{-s} U_4)&\text{for } (x,t)\in \RN \times (0,T),\\
U_4(x,0)=u_0(x) &\text{for } x \in \RN.
\end{array}
\right.
\tag{$P_{\delta}$}\label{ProblemDelta} \]
The first integral of the weak formulation passes to the limit like in \eqref{weaklimiteps} as consequence of \eqref{blublu}. It remains to prove that
\begin{equation}\label{blabla1}
\int_0^T \int_{\R^N} (U_3+\mu)^{m-1}\nabla (-\Delta)^{-s}U_3 \cdot \nabla \phi dx dt\stackrel{\mu \to 0}{\longrightarrow} \int_0^T \int_{\R^N} U_4^{m-1}\nabla (-\Delta)^{-s}U_4 \cdot \nabla \phi dx dt.
\end{equation}
Let $\phi$ be supported in $B_\rho$ for some $\rho>0$. It is clear that
\begin{equation}\label{blabla2}
(U_3+\mu)^{m-1}\to U_4^{m-1} \quad \textup{as} \quad \mu \to 0 \quad \textup{in} \quad L^2(B_\rho\times (0,T)).
\end{equation}
Moreover, from the second energy estimate, we get that there exists a weak limit of $U_3$ in $L^2(0,T:H^{1-s}(B_\rho))$. Furthermore, the limit can be identified in $L^2(B_\rho\times(0,T))$ from \eqref{blublu}, and then
\[
U_3 \rightharpoonup U_4 \quad \text{ as } \quad \mu\to0 \quad \text{in} \quad L^2(0,T:H^{1-s}(B_\rho)).
\]
Since the term $\nabla(-\Delta)^{-s}$ is of order $1-2s$, which is smaller than $1-s$, then
\begin{equation}\label{blabla3}
\nabla (-\Delta)^{-s}U_3\rightharpoonup \nabla (-\Delta)^{-s}U_4 \quad   \text{in} \quad L^2(B_\rho\times(0,T)) .
\end{equation}
Combining \eqref{blabla2} and \eqref{blabla3} the convergence \eqref{blabla1} follows.

\subsubsection{Energy estimates}
We state the main properties of the solution of Problem \eqref{ProblemDelta}.
\begin{theorem}
Let $s\in(0,1)$, $1<m<\infty$ and $N \ge 1$. There exists a weak solution $U_4$ of Problem \eqref{ProblemDelta} with initial data $u_0\in L^1(\R^N)\cap L^\infty(\R^N)$. Moreover, $U_4$ has the following properties:
\begin{enumerate}
\item \textbf{(Conservation of total mass)} For all $0<t<T$ we have
$\displaystyle{
\int_{\RN}U_4(x,t)dx =\int_{\RN}u_0(x)dx.
}$
\item \textbf{($L^{\infty}$-estimate) } For all $0<t<T$ we have  $||U_4(\cdot,t)||_\infty\leq ||u_0||_\infty$.
\item \textbf{($L^p$-decay energy estimate)} For all $1<p<\infty$ and $0<t<T$
\begin{equation}\label{energyU4}
\begin{split}
 \int_{\RN} U_4^p(x,t)dx  & +\frac{4(p-1)\delta}{p}\int_0^t \int_{\RN} \left|\nabla (U_4^{p/2})\right|^2 dxdt \\\
 &+ \frac{p (p-1)}{m+p-2} \int_0^t \int_{\RN}  U_4^{m+p-2} (-\Delta)^{1-s} U_4dxdt \le \int_{\RN} u_0^p(x)dx.\end{split}
\end{equation}
\item \textbf{(Second energy estimate)} For all $0<t<T$  we have
\begin{equation}\label{SecondEnergyU4}
\begin{split}
 \frac{1}{2} \int_{\RN} \left|(-\Delta)^{-\frac{s}{2}} U_4(t)\right|^2 & dx   +\delta \int_0^t\int_{\RN}\left|(-\Delta)^{\frac{1-s}{2}}[U_4]\right|^2  dx dt  \\
&
+\int_0^t \int_{\RN}  U_4^{m-1} \left| \nabla  (-\Delta)^{-s} U_4(t)\right|^2 dx  dt    \le \frac{1}{2} \int_{\RN} \left|(-\Delta)^{-\frac{s}{2}} u_0\right|^2 dx. \end{split}
\end{equation}
\end{enumerate}
\end{theorem}
\noindent The proof is as in the previous part. The term $\displaystyle \int_0^t \int_{\RN}  (U_3+\mu)^{m-1} \left| \nabla  (-\Delta)^{-s} U_4(t)\right|^2 dx  dt   $ passes to the limit by Fatou's Lemma since $(U_3+\mu)^{m-1} \to U_4 ^{m-1}$ as $\mu \to 0$ pointwise.

\subsection{Limit as \texorpdfstring{$\delta \to 0$}{}}\label{Subsec:delta}

This part is quite interesting and brings some novelty in the techniques we have employed so far.
Here we  use  a different compactness criteria in order to derive the convergence  as $\delta \to 0$. This is a consequence of the lack of regularity that was given by the $\delta$-term in the previous approximating problems.

Estimates \eqref{energyU4} and \eqref{SecondEnergyU4} provide  an upper bound independent of $\delta$. The terms with $\delta$ coefficient are positive and bounded and therefore $U_4$ satisfies:
\begin{equation}\label{energyU4b}
 \int_{\RN} U_4^p(x,t)dx + \frac{p (p-1)}{m+p-2} \int_0^t \int_{\RN}  U_4^{m+p-2} (-\Delta)^{1-s} U_4dxdt \le \int_{\RN} u_0^p(x)dx,
\end{equation}
and
\begin{equation}\label{SecondEnergyU4b}
 \frac{1}{2} \int_{\RN} \left|(-\Delta)^{-\frac{s}{2}} U_4(t)\right|^2  dx
+\int_0^t \int_{\RN}  U_4^{m-1} \left| \nabla  (-\Delta)^{-s} U_4(t)\right|^2 dx  dt    \le \frac{1}{2} \int_{\RN} \left|(-\Delta)^{-\frac{s}{2}} u_0\right|^2 dx.
\end{equation}

\subsubsection{Existence of a limit. Compactness estimate II} \label{subsec:compacttemam} We will prove compactness for the following sequence:
\[
W_\delta:=
\left\{
\begin{array}{ll}
U_4 \quad   &  \text{if} \quad  m \le 2 \\
U_4^m \quad & \text{if}\quad   m >2.
\end{array}
\right.
\]
The idea is to apply Theorem \ref{RakotosonTemam3} for $W_\delta$ and in order to use this compactness criteria we need to work on a bounded domain $B_\rho$ for $\rho>0$. From \eqref{energyU4b}, applying Stroock-Varopoulos we obtain
\begin{equation}\label{energyU4SV}
\begin{split}
\int_{\RN} u_0^p(x)dx &\ge \int_{\RN} U_4^p(x,t)dx +\frac{4p(p-1)}{(m+p-1)^2}  \int_0^t \int_{\RN}  \left| (-\Delta)^{\frac{1-s}{2}} U_4^{\frac{m+p-1}{2}} \right|^2 dxdt.
\end{split}
\end{equation}
In this way we get a uniform bound for $W_\delta$ in  $L^2(0,T:H^{1-s}(B_\rho))$ by using \eqref{energyU4SV} with
  $p=3-m$ if $m\le 2$ and $p=m+1$ if $m>2$.
  Note that the exponent $3-m$ is again critical in the proof of existence, as happened in the article \cite{StanTesoVazJDE}. In both cases we get that there exists a weak limit
$$
W_\delta \rightharpoonup W \text{ in }L^2(0,T:H^{1-s}(B_\rho)).
$$
Then, hypothesis a) in Theorem \ref{RakotosonTemam3} is satisfied in the context $V=H^{1-s}(B_\rho)$ and $H=L^2(B_\rho)$. However, b) also holds due to the energy estimate  \eqref{energyU4SV}  for $p=2q$ where $q=1$ if $m\leq2$ and $q=m$ if $m>2$. Indeed we have the following estimate
$$
\sup_{\delta>0}\|W_\delta(t)\|_{L^2(B_\rho)} = \sup_{\delta>0}\|U_4^{q}(t)\|_{L^2(B_\rho)}= \sup_{\delta>0}\|U_4(t)\|_{L^{2q}(B_\rho)}^{q} \le \|u_0\|_{L^{2q}(B_\rho)}^{q} < +\infty
$$
for  every $ t \in (0,T)$. It remains to prove assumption c) of  Theorem  \ref{RakotosonTemam3}. Since $L^2(B_\rho)$ is a separable Hilbert space, we can find a countable set $D$ dense in $L^2(B_\rho)$. Moreover, we can assume that the elements $\psi\in D$ are smooth and nonnegative.

We want to prove that the family of  functions $g^\delta_\psi(t):=<U_4(\cdot,t),\psi>_{L^2(B_\rho)}$ is relatively compact in $L^1((0,T))$. First, $\{g^\delta_\psi\}_{\delta>0}$ is equibounded in $L^1((0,T))$ since
\begin{equation*}
\begin{split}
\|g^\delta_\psi\|_{L^1((0,T))}:=&\int_0^T \int_{B_\rho} U_4(x,t)\psi(x)dxdt\leq\left( \int_0^T \int_{B_\rho} (U_4)^2dxdt\right)^{1/2}\left( \int_0^T \int_{B_\rho} \psi^2dxdt\right)^{1/2}\\
 \le& T\|u_0\|_{L^2(B_\rho)}\|\psi\|_{L^2(B_\rho)}.
\end{split}
\end{equation*}
Moreover, we also have that $g_\psi^\delta(t)$ is equicontinuous in $L^1((0,T))$: using \eqref{ProblemDelta} we have
\begin{equation*}
\begin{split}
&\int_0^T(g_\psi^\delta)'(t)dt=\delta \int_0^T <U_4,\Delta\psi>dt+\int_0^T<(-\Delta)^{-\frac{s}{2}}U_4, \nabla (-\Delta)^{-\frac{s}{2}} (U_4^{m-1}\nabla \psi)>dt\\
&\leq \delta \|U_4\|_{L^2(B_\rho\times(0,T))}T^{1/2}\|\Delta\psi\|_{L^2(B_\rho)}+ \|(-\Delta)^{-\frac{s}{2}}U_4\|_{L^2(B_\rho\times (0,T))} \|\nabla (-\Delta)^{-\frac{s}{2}} (U_4^{m-1}\nabla \psi)\|_{L^2(B_\rho\times(0,T))},
\end{split}
\end{equation*}
where all the terms in the last inequality are absolutely bounded in $\delta$ due to the energy estimates \eqref{energyU4b} and \eqref{SecondEnergyU4b}. We use the fact that for any smooth  function $\psi\in D$ we have that $\psi U_4^{m-1} \in  L^2(0,T:H^{1-s}(\RN))$ and then
$\nabla   (-\Delta)^{-\frac{s}{2}} ( U_4^{m-1} \nabla \psi) \in L^2(\RN\times (0,T))$ uniformly on $\delta$.

In this way, if $m\leq2$, since $U_4 = W_\delta$, we have that hypothesis c) of Theorem  \ref{RakotosonTemam3} is satisfied by $W_\delta$. If $m\geq2$, then $<U_4^m, \psi>_{L^2(B_\rho)}$ is clearly equibounded in $L^1((0,T))$. Moreover, by the equicontinuity of  $g_\psi^\delta(t)$ and the following estimate
\[
\int_{t_1}^{t_2}<U_4^m, \psi>_{L^2(B_\rho)}dt \leq \|u_0\|^{m-1}_{L^\infty(\R^N)}\int_{t_1}^{t_2}<U_4, \psi>_{L^2(\Omega)}dt,
\]
we have that $<U_4^m, \psi>_{L^2(B_\rho)}$ is also equicontinuous in  $L^1((0,T))$. We apply  Theorem \ref{RakotosonTemam1}  to obtain
$$
W_\delta \rightarrow W \text{ in }L^2(B_\rho \times (0,T)).
$$
For $m \le2$ this means $U_4 \rightarrow W \text{ in }L^2(B_\rho \times (0,T))$ and we are done.
Now, let $m > 2$. We have $W_\delta=U_4^{m} \to W$ in $L^2(B_\rho\times (0,T))$.
Since $(U_4)_\delta \in L^\infty(\RN \times (0,T))$ uniformly in $\delta$ then also the limit $W (x,t) \in L^\infty (\RN \times (0,T))$. In both cases, by the covering plus diagonal argument and Fatou's Lemma as in Section \ref{SectExistR}, we obtain, up to a subsequence,  that
\begin{equation}\label{convU4u}
U_4 \rightarrow u \text{ in }L^2_{\textup{loc}}(\RN \times (0,T)).
\end{equation}

\subsubsection{The limit \texorpdfstring{$u$}{} is a weak solution of Problem \texorpdfstring{\eqref{model1}}{}} \label{subsubsec:limitissol}

We pass to the limit as $\delta \to 0$ in the weak formulation corresponding to Problem \eqref{ProblemDelta}.
Let $\phi$ a compactly supported test function  with support in $B_\rho$. Then by \eqref{convU4u} we get
\[
\int_0^T\int_{\R^N} U_4\phi_tdxdt\to\int_0^T\int_{\R^N} u\phi_tdxdt \quad \text{as} \quad \delta \to 0.
\]
Moreover,
\[
\delta \int_0^T\int_{\R^N} U_4 \Delta \phi dxdt\to0 \quad \text{as} \quad \delta \to 0.
\]
It remains to prove that
\begin{equation}\label{ConvGradU4}
\int_0^T\int_{\RN} U_4^{m-1} \nabla (-\Delta)^{-s} U_4  \nabla \phi dx dt \to \int_0^T\int_{\RN} u^{m-1} \nabla (-\Delta)^{-s} u  \nabla \phi dxdt.
\end{equation}
\noindent\textbf{I.} Case $m\le 2$. From $L^p$ estimate \eqref{energyU4SV} with $p=3-m$ we have that $U_4 \in H^{1-s}(B_\rho)$ and then $U_4 \rightharpoonup u$ in $H^{1-s}(\Omega)$. As a consequence
\begin{equation}\label{weakconvgradU4u}
\nabla (-\Delta)^{-s}U_4\rightharpoonup \nabla (-\Delta)^{-s}u \quad   \text{in} \quad L^2(B_\rho\times(0,T)) .
\end{equation}
Moreover, we have that $U_4^{m-1}\to u^{m-1}$ in $L^2(B_\rho\times(0,T))$, which together with  \eqref{weakconvgradU4u} implies \eqref{ConvGradU4}.

\noindent\textbf{II.}  Case $m>2$.   We will use the fact that $\nabla \cdot  (-\Delta)^{-s} ( U_4^{m-1} \nabla \phi) \in L^p(\RN\times (0,T))$ uniformly on $\delta$, for a certain  $p>1$. For the sake of a clean presentation, we present the proof of this fact in Appendix \ref{app:techresult}.   On the other hand, $U_4\in L^q(\RN)$ for any $L^q(\RN)$ uniformly on $\delta>0$ and thus we integrate by parts the first integral of \eqref{ConvGradU4} to get
 $$I(U_4):= \int_0^T\int_{\RN} U_4^{m-1} \nabla (-\Delta)^{-s} U_4  \nabla \phi dx dt = \int_0^T\int_{\RN}      U_4\nabla  \cdot   (-\Delta)^{-s} ( U_4^{m-1} \nabla \phi) dx dt. $$
Moreover, for every $\phi$  there exists a weak limit
$$
\nabla \cdot (-\Delta)^{-s} ( U_4^{m-1} \nabla \phi) \rightharpoonup v \quad \text{as}\quad \delta \to 0 \quad \text{in}\quad L^p(\RN\times (0,T)).
$$
We identify the limit in the sense of distributions and show that $v= \nabla \cdot(-\Delta)^{-s} ( u^{m-1} \nabla \phi) $: indeed  we have that
$$\int_0^T\int_{\RN} U_4^{m-1}  \nabla \phi \nabla   (-\Delta)^{-s}\psi    dx dt \to  \int_0^T\int_{\RN} u^{m-1}  \nabla \phi \nabla   (-\Delta)^{-s}\psi  dx dt \quad \text{for all} \quad \psi \in C_c^\infty(\R^N\times (0,T))$$
since $U_4^{m-1}  \to u^{m-1}$ in $L^1_{\text{loc}}(\RN \times (0,T)).$   Therefore
\begin{equation} \label{limU4PhiWeak}
\nabla \cdot(-\Delta)^{-s} ( U_4^{m-1} \nabla \phi) \rightharpoonup \nabla \cdot (-\Delta)^{-s} ( u^{m-1} \nabla \phi) \quad \text{as}\quad \delta \to 0 \quad \text{in}\quad L^p(\R^N\times(0,T)),
\end{equation}
for every test function $\phi$.

Let $R>0$. Then
\begin{align*}
 I(U_4)&=  \int_0^T\int_{B_R}      U_4\nabla  \cdot   (-\Delta)^{-s} ( U_4^{m-1} \nabla \phi) dx dt + \int_0^T\int_{\RN \setminus B_R}      U_4\nabla  \cdot   (-\Delta)^{-s} ( U_4^{m-1} \nabla \phi) dx dt\\
 &=I_1(U_4) + I_2(U_4).
\end{align*}
Since the sequence $U_4^{m-1} \nabla \phi$ has the same compact support for all $\delta$ then $\nabla  \cdot   (-\Delta)^{-s} ( U_4^{m-1} \nabla \phi) $ uniformly decays for large $|x|$ (see \eqref{grad1-2s}). Then we can choose $R$ big enough  such that $I_2(U_4)<\epsilon/3$. In the same way $I_2(u)<\epsilon/3$.
Now, with this given $R$ we use that $U_4 \to u$ in  $L^q_{\text{loc}}(\RN \times(0,T))$ together with \eqref{limU4PhiWeak} and we have $I_1(U_4) \to I_1(u)$ as $\delta \to 0.$
Thus, we choose $\delta>0$ such that
\begin{align*}
 | I(U_4)-I(u)|&\le |I_1(U_4)-I_1(u)| + |I_2(U_4)| +|I_2(u)| <\frac{\epsilon}{3} +\frac{\epsilon}{3}+\frac{\epsilon}{3}=\epsilon.
\end{align*}

We integrate by parts to obtain the desired convergence  \eqref{ConvGradU4}.

\subsubsection{Energy estimates}
We pass to the limit in the energy estimates. From \eqref{energyU4b}-\eqref{energyU4SV} we get that
$$
 \int_{\RN} u^p(x,t)dx + \frac{p (p-1)}{m+p-2} \int_0^t \int_{\RN}  u^{m+p-2} (-\Delta)^{1-s} u \, dxdt \le \int_{\RN} u_0^p(x)dx.
$$
From \eqref{SecondEnergyU4b} we get
$$
 \frac{1}{2} \int_{\RN} \left|(-\Delta)^{-\frac{s}{2}} u(t)\right|^2  dx
+\int_0^t \int_{\RN}  u^{m-1} \left| \nabla  (-\Delta)^{-s} u(t)\right|^2 dx  dt    \le \frac{1}{2} \int_{\RN} \left|(-\Delta)^{-\frac{s}{2}} u_0\right|^2 dx.
$$

We have obtained so far the existence of a weak solution of Problem \eqref{model1} enjoying regularity properties and the corresponding energy estimates. This concludes the proof of Theorem \ref{Thm1PMFP}.

\subsection{Dealing with the case \texorpdfstring{$N=1$, $s\in [\frac{1}{2},1)$}{}}\label{SectN1}

The operator $(-\Delta)^{-s}$ is not well defined when $N=1$ and $\frac{1}{2} <s<1$ since the convolution kernel $\displaystyle K_s=\frac{1}{|x|^{1-2s}}$ does not decay at infinity.
Therefore it does not make sense to think of equation \eqref{model1} in terms of a pressure. This may not be very convenient, but the issue can be avoided by writing the equation as
$$
u_t=\nabla \cdot( u^{m-1} \nabla^{1-2s} u),
$$
where $\nabla^{1-2s}$ denotes formally the composition operator $\nabla(-\Delta)^{-s}$. According to \cite{BilerImbertKarch}, $\nabla^{1-2s}$ can be written in the whole range $0<s<1$ in terms of the singular integral formula for smooth and bounded functions
\begin{equation}\label{grad1-2s}
\nabla^{1-2s}\psi(x)= C_{N,s} \int (\psi(x)-\psi(x+z)) \frac{ \text{sign}(z)}{|z|^{N+1-2s}} dz.
\end{equation}
Note that for $\frac{1}{2} <s<1$, $|z|^{-N-1+2s}\in L^1_{\textup{loc}}(\RN)$ and decays at infinity. Note also that $\nabla^{1-2s}$ has the Fourier symbol given by $i \, \text{sign}(\xi) |\xi|^{1-2s}.$ Moreover, the operator $(-\Delta)^{-\frac{s}{2}}$ is well defined in the whole range $0<s<1$ even in dimension $N=1$. In this way, we have the following property:
\[
\nabla^{1-2s}=(-\Delta)^{-\frac{s}{2}} \nabla (-\Delta)^{-\frac{s}{2}} =(-\Delta)^{-\frac{s}{2}}\nabla^{1-s}.
\]

The $L^p$ energy estimate \eqref{Lpdecay} still has the same form, while the second energy estimate \eqref{SecondEnergy} needs has to be reformulated as
\[
  \frac{1}{2} \intr \left|(-\Delta)^{-\frac{s}{2}} u(t)\right|^2 dx  +
\int_0^t \intr  u^{m-1} \left|\nabla^{1-2s}  u(t)\right|^2 dx  dt \le \frac{1}{2} \intr \left|(-\Delta)^{-\frac{s}{2}} u_0\right|^2 dx.
\]
The proofs of Section  \ref{SectExist} follow similarly. For the $\epsilon \to 0$ limit, we shall use part b) of Lemma \ref{LemmaAproxInverseFracLap}.

\section{Existence of solutions with measure data}
\label{SectionExistL1}
In this section we give the proof of the existence of weak solutions  taking as initial data any $\mu \in \mathcal{M}^{+} (\RN)$, the space of nonnegative Radon measures on $\RN$ with \em{finite mass}\em. In particular, this includes the case of only integrable data $u_0 \in L^1(\RN)$. Therefore, we improve the results from \cite{CaffVaz,StanTesoVazJDE} to less restrictive initial data. As precedent we mention \cite{CaffSorVaz} where the authors  extend the existence theory for $m=2$ to every $u_0 \in L^1(\RN)$. The case of measures has been considered for the case $m=2,\ s \to 1$ in \cite{SerfVaz}, and for model \eqref{FPME} in \cite{VazquezBarenblattFPME}.

\begin{defn}\label{def2}
Let  $\mu \in \mathcal{M}^{+} (\RN)$.   We say that $u\geq0 $ is a weak solution of Problem \eqref{model1} with initial data $\mu$ if: \\ (i) $u \in L^1_{\textup{loc}}(\RN \times (0,T))$ , (ii) $\nabla (-\Delta)^{-s}u \in L^1_{\textup{loc}}(\RN \times (0,T))$, (iii) $u^{m-1} \nabla (-\Delta)^{-s} u\in L^1_{\textup{loc}}(\RN \times (0,T))$,
\begin{equation*}
\int_0^T\int_{\RN} u \phi_t\,dxdt-\int_0^T\int_{\RN}  u^{m-1} \nabla (-\Delta)^{-s} u \cdot\nabla \phi \,dxdt+  \int_{\RN} \phi(x,0) d \mu(x)=0.
\end{equation*}
for all test functions  $\phi \in C^1_c(\RN \times [0,T))$.

\end{defn}

\begin{theorem}\label{ThmExistL1}
Let  $1<m<\infty$, $N\ge 1$ and     $\mu \in \mathcal{M}^{+} (\RN)$.   Then there exists a weak solution $u\geq0$ (in the sense of Definition \ref{def2}) of Problem \eqref{model1} such that the smoothing effect \eqref{smoothform} holds for $p=1$ in the following sense:
\[
\| u(\cdot,t)\|_{L^{\infty}(\RN)} \le C_{N,s,m} \, t^{-\gamma} \mu(\R^N)^{\delta}   \quad \textup{for all} \quad t>0,
\]
where
$\gamma=\frac{N}{(m-1)N+2(1-s)}$, $\delta=\frac{2(1-s)}{(m-1)N+2(1-s)}$. Moreover,
$$u \in   L^\infty((0,\infty):L^1(\RN)) \cap L^\infty (\RN \times (\tau, \infty))    \quad \textup{for all}  \quad \tau >0$$
and it has the following properties
\begin{enumerate}
\item \textbf{(Conservation of mass)} For all $0 < t < T$ we have
$\displaystyle{
\int_{\RN}u(x,t)dx= \int_{\RN}d\mu(x).
}$

\item \textbf{($L^p$ energy estimate)} For all $1<p<\infty$ and $0 <\tau< t < T$ we have
\begin{equation*} \intr u^p(x,t)dx + \frac{4p(p-1)}{(m+p-1)^2}\int_\tau^t \int_{\RN}\left|(-\Delta)^{\frac{1-s}{2}} u^{\frac{m+p-1}{2}}\right|^2 dxdt \le \intr u^p(x,\tau)dx .
\end{equation*}

\item \textbf{(Second energy estimate)} For all $0 < \tau<t < T$  we have
\begin{equation*}
  \frac{1}{2} \intr \left|(-\Delta)^{-\frac{s}{2}} u(t)\right|^2 dx  +
\int_\tau^t \intr  u^{m-1} \left| \nabla  (-\Delta)^{-s} u(t)\right|^2 dx  dt \le \frac{1}{2} \intr \left|(-\Delta)^{-\frac{s}{2}} u(\tau)\right|^2 dx.
\end{equation*}

\end{enumerate}

\end{theorem}

\begin{remark}
{\rm If $\mu$ is an absolutely continuous with respect to the Lebesgue measure, it has a density $u_0\in L^1(\R^N)$ such that $d\mu(x)=u_0(x)dx$. In this case $u_0$ is an initial condition in the sense given in Definition \ref{def1}.}
\end{remark}

\begin{proof}
\noindent \textbf{I. Approximation with bounded solutions.}  Let $\{\rho_n\}_{n>0}$ be a sequence  of standard mollifiers. We define the approximate initial data by convolution, i.e.,  for any $n >0$ we consider the function $(u_0)_n \in L^1(\R^N)\cap L^\infty(\R^N)$ defined by
$$
(u_0)_n(x):=\int_{\R^{N}}\rho_n(x-z)d \mu(z).
$$
Note that, by Fubini's Theorem, we have that
\[
\|(u_0)_n\|_{L^1(\R^N)}=\int_{\R^{N}}d \mu(z)=\mu(\R^N).
\]
It is clear that $(u_0)_n\to \mu$ as $n\to \infty$ in the sense required by Definition \ref{def2}, that is,
\begin{equation}\label{weakconvmu}
\int_{\R^N}(u_0)_n(x)\psi(x)d x\to\int_{\R^N}\psi(x)d \mu(x) \quad \textup{as} \quad n \to \infty,
\end{equation}
for all $\psi \in C_c^1(\R^N)$. Now let $u_n\in  L^1(\R^N)\cap L^\infty(\R^N)$ be the solution of Problem \eqref{model1} with initial data $(u_0)_n$ provided by Theorem \ref{Thm1PMFP}. Moreover, thanks to the $L^1$-$L^\infty$ smoothing effect given by Theorem \ref{ThmSmoothing} we have the following estimates that are independent of $n$:

i) For all $0<t<T$ we have $\|u_n(\cdot,t)\|_{L^1(\R^N)}= \|(u_0)_n\|_{L^1(\R^N)} =\mu(\R^N)$.

ii) For all $0<\tau<t\leq T$ we have \[
\|u_n(\cdot,t)\|_{L^\infty(\R^N)}\leq \|u_n(\cdot,\tau)\|_{L^\infty(\R^N)} \leq C_{N,s,m} \, \tau^{-\gamma}\|(u_0)_n\|_{L^1(\RN)}^{\delta}=C_{N,s,m} \, \tau^{-\gamma}\mu(\R^N)^{\delta},\]
where
$\gamma=\frac{N}{(m-1)N+2(1-s)}$, $\delta=\frac{2(1-s)}{(m-1)N+2(1-s)}$.

Furthermore, since i) and ii) show that $u_n\in L^\infty(\R^N \times (\tau, T)) \cap L^1(\R^N \times (0, T)) $ uniformly in $n$, we have the following energy estimates for which the right hand side are absolutely bounded in $n$ (the precise bounds will be given later):

iii) For all $1<p<\infty$ and $0<\tau<t\leq T$,
$$ \intr u_n^p(x,t)dx + \frac{4p(p-1)}{(m+p-1)^2}\int_\tau^t \int_{\RN}\left|(-\Delta)^{\frac{1-s}{2}} u_n^{\frac{m+p-1}{2}}\right|^2 dxdt \le \intr u_n^p(x,\tau)dx .
$$

iv) For all  $0<\tau<t\leq T$,
\begin{equation*}
\begin{split}
 \frac{1}{2} \int_{\RN} \left|(-\Delta)^{-\frac{s}{2}} u_n(t)\right|^2 & dx
+\int_\tau^t \int_{\RN}  u_n^{m-1} \left| \nabla  (-\Delta)^{-s} u_n(t)\right|^2 dx  dt    \le \frac{1}{2} \int_{\RN} \left|(-\Delta)^{-\frac{s}{2}} u_n(\tau)\right|^2 dx. \end{split}
\end{equation*}

\noindent \textbf{II. Convergence away from $t=0$.}  Given any $\tau>0$ we can use the compactness criteria given by Theorem \ref{RakotosonTemam3} as in Section \ref{subsec:compacttemam} to show that
\begin{equation}\label{convntau}
u_n\longrightarrow u^\tau \quad \text{as}\quad  n \to \infty \quad  \text{in} \quad L^2_{\text{loc}}(\R^N\times(\tau,T)).
\end{equation}
In the weak formulation, for any $\phi \in C_c^\infty(\R^N\times[0,T))$, $u_n$ satisfies:
$$
\int_\tau^T\int_{\R^N} u_n\phi_tdxdt - \int_\tau^T\int_{\RN} u_n^{m-1} \nabla (-\Delta)^{-s} u_n  \nabla \phi dx dt  + \int_{\R^N} u_n(\tau) \phi(x,\tau)dx=0.
$$
Moreover, we can proceed as in Section \ref{subsubsec:limitissol} to prove that for any test function $\phi$ we have
\[
\int_\tau^T\int_{\R^N} u_n\phi_tdxdt\to\int_\tau^T\int_{\R^N} u^\tau\phi_tdxdt \quad \text{as} \quad \delta \to 0.
\]
and
\[
\int_\tau^T\int_{\RN} u_n^{m-1} \nabla (-\Delta)^{-s} u_n   \nabla \phi \,dx dt \to \int_\tau^T\int_{\RN} (u^\tau)^{m-1} \nabla (-\Delta)^{-s}  u^\tau  \nabla \,\phi dxdt.
\]

\noindent \textbf{III. Uniform estimates at $t=0$.}  In order to show that we can pass to the limit as $\tau \to 0 $ to obtain a weak solution of Problem \eqref{model1} we need to prove that the remaining terms converge to zero as $\tau \to0$. First of all,
\[
\left|\int_0^\tau\int_{\R^N} u_n\phi_t\,dxdt\right| \leq C \int_0^\tau \|u_n(\cdot,t)\|_{L^1(\R^N)}dt =C\tau \mu(\R^N).
\]
Now we use the classical Riesz embedding (c.f \cite{Stein70}) and that $u_n(\cdot,t) \in L^1(\R^N)\cap L^\infty(\R^N)$ for any $t>0$ to get
\[
\int_{\R^N} \left|(-\Delta)^{-\frac{s}{2}}u_n(t)\right|^2dx \leq C\|u_n(t)\|_{p}^2 \quad \text{with} \quad \frac{1}{2}=\frac{1}{p}- \frac{s}{N}.
\]
Also, from the smoothing effect, we have
\[
\|u_n(t)\|_p^p\leq \|u_n(t)\|_1 \|u_n(t)\|_\infty^{p-1} \leq C \mu(\R^N)^{1+(p-1)\delta} t^{-\gamma(p-1)}.
\]
In this way, we get
\[
\int_{\R^N} \left|(-\Delta)^{-\frac{s}{2}}u_n(t)\right|^2dx \leq C \mu(\R^N)^\sigma t^{-\lambda},
\]
for some $\sigma>0$ and
\[
\lambda=\frac{2\gamma(p-1)}{p}=\frac{2N}{(m-1)N+2-2s}\frac{N-2s}{2N}=\frac{N-2s}{(m-1)N+2-2s}.
\]
Consider the strip $Q_k=\R^N\times(t_k,t_{k-1})$ with $t_k=2^{-k}$. Then
\begin{equation*}
\begin{split}
\int\int_{Q_k} u_n^{m-1} |\nabla (-\Delta)^{-s} u_n|dx dt &\leq \left(\int\int_{Q_k} u_n^{m-1}dxdt\right)^{1/2} \left(\int\int_{Q_k}u_n^{m-1}  |\nabla (-\Delta)^{-s} u_n|^2 dxdt\right)^{1/2}\\
&\leq \|u_n(t)\|_{L^{\infty}(Q_k)}^{\frac{m-2}{2}}\left( \mu(\R^N) t_{k}\right)^{1/2} \left( \frac{1}{2} \int_{\RN} \left|(-\Delta)^{-\frac{s}{2}} u_n(t_k)\right|^2 dx\right)^{1/2}\\
&\leq  C \mu(\R^N)^{\hat{\sigma}} t_k^{-\gamma \frac{m-2}{2}} t_k^{\frac{1}{2}} t_k^{-\frac{\lambda}{2}}\\
&= C \mu(\R^N)^{\hat{\sigma}} t_k^\alpha,
\end{split}
\end{equation*}
for some $\hat{\sigma}>0$ and
\begin{equation*}
\begin{split}
\alpha&=\frac{1}{2}\left(1-\gamma (m-2)- \frac{N-2s}{(m-1)N+2-2s}\right)=\frac{1}{(m-1)N+2-2s}>0.
\end{split}
\end{equation*}
In this way,
\begin{equation}\label{lastestim}
\begin{split}
\left|\int_0^\tau\int_{\RN} u_n^{m-1} \nabla (-\Delta)^{-s} u_n  \nabla \phi dx dt \right|\leq\|\nabla \phi\|_{\infty}\int_0^\tau\int_{\RN} u_n^{m-1} |\nabla (-\Delta)^{-s} u_n|dx dt\leq \Lambda (\tau)
\end{split}
\end{equation}
for some modulus of continuity $\Lambda$.

\noindent \textbf{IV. Initial data.}  The only thing left is to prove that the initial data is taken. Let $\phi$ be a $C_c^1(\R^N)$ test function. Then, using the estimate given by \eqref{lastestim}, we get
\begin{equation}\label{lastestim2}
\begin{split}
 \left|\int_{\R^N}(u_n(\tau)-(u_0)_n)\phi dx \right|=  \left|\int_0^\tau\int_{\R^N}\partial_t u_n\phi \,dx dt\right|=\left|\int_0^\tau\int_{\RN} u_n^{m-1} \nabla (-\Delta)^{-s} u_n  \nabla \phi \,dx dt \right|\leq \Lambda(\tau).
\end{split}
\end{equation}
A standard diagonal procedure in $n$ and $\tau$ concludes the proof.

\noindent \textbf{V. Conservation of mass.} We can also conclude conservation of mass by taking a sequence of test functions of the cutoff type, $\phi_R(x)=\phi(x/R)$ with $0\le \phi\le 1$ and $\phi_1(x)=1$ for $|x|\le 1$ and such that $\|\nabla \phi_R\|_{L^\infty(\R^N)}= O(R^{-1})$ (see appendix A.2 in \cite{StanTesoVazJDE} for more details). Then, using \eqref{lastestim} and \eqref{lastestim2}, we get that for any $\tau>0$ we have
\begin{equation*}
 \left|\int_{\R^N}u_n(\tau)\phi_R dx -\int_{\R^N}(u_0)_n\phi_R dx \right|  \leq C\frac{\Lambda(\tau)}{R}.
\end{equation*}
In particular, the previous estimate implies that
\begin{equation*}
\begin{split}
\int_{\R^N}u_n(\tau)\phi_R dx&\geq \int_{\R^N}(u_0)_n\phi_R dx-C\Lambda(\tau)/R\\
&= \int_{\R^N}(u_0)_n\phi_R dx- \int_{\R^N}\phi_R(x)d\mu(x)+\int_{\R^N}\phi_R(x) d\mu(x)-C\Lambda(\tau)/R
\end{split}
\end{equation*}
In view of \eqref{convntau} and \eqref{weakconvmu} we can let $n\to\infty$ in the previous estimate to get
\begin{equation*}
\begin{split}
\int_{\R^N}u(\tau)\phi_R dx\geq\int_{\R^N}\phi_R(x) d\mu(x)-C\Lambda(\tau)/R.
\end{split}
\end{equation*}
Note that, since $\mu$ is measure with finite mass in $\R^N$, then
\[\int_{\R^N}\phi_R(x) d\mu(x)\geq \mu(\R^N)-\epsilon(R)\]
 with $\epsilon(R)\to 0$ as $R\to \infty$. Therefore,
\[
\int_{\R^N}u(\tau)\phi_R dx\geq\mu(\R^N)-\epsilon(R)-C\Lambda(\tau)/R.
\]
Letting now $R\to\infty$ we get
$$
\int_{\R^N}u(\tau)\, dx\ge   \mu(\R^N).
$$
In this way we show that no mass is lost at infinity during the evolution. The other inequality comes from the construction of solutions.

\end{proof}

\begin{remark} {\rm The proof of mass conservation  given in Theorem \ref{ThmExistL1} is strongly based  on the estimates available from the  $L^1-L^{\infty}$ smoothing effect. This is a more powerful tool than the one presented in \cite{StanTesoVazJDE} where the assumption of the boundedness on solution was unavoidable.}
\end{remark}

%%%%%%%%%%%%%%%%%%%%%%%%%%%%%%%%%%%%%%%%%%%%%%%%%%%%%%
\section{Comments and open problems}\label{sec:comments}

\noindent  $\bullet$ \textbf{First energy estimate.} Let $u$ be the solution of Problem \eqref{model1}.  The following formal estimates can be derived for any $t>0$:
\begin{equation}\label{FirstEnergy}
   \begin{array}{ll} \displaystyle
|(2-m)(3-m)|\int_0^t  \int_{\RN} |\nabla (-\Delta)^{-\frac{s}{2}}u|^2dxdt  +\int_{\RN}u(t)^{3-m}dx \le \int_{\RN}u_0^{3-m} dx \quad &\textup{if} \quad m\not=2,3.\\[3mm]
 \displaystyle
\int_0^t  \int_{\RN} |\nabla (-\Delta)^{-\frac{s}{2}}u|^2dxdt  +\int_{\RN} \left(u(t)-\log(u(t))\right)dx\leq \int_{\RN} \left(u_0-\log(u_0)\right)dx\quad &\textup{if} \quad m=3.\\[3mm]
 \displaystyle
\int_0^t  \int_{\RN} |\nabla (-\Delta)^{-\frac{s}{2}}u|^2dxdt  +\int_{\RN} u(t)\log(u(t))dx\leq \int_{\RN} u_0\log(u_0)dx\quad &\textup{if} \quad m=2.
    \end{array}
\end{equation}

This kind of energy estimates were a key tool to prove existence in the previous paper \cite{StanTesoVazJDE}. When $m\in (1,2)$, they only require $u_0 \in L^1(\RN) \cap L^{\infty}(\RN)$ in order to have uniform bounds on the $L^2(\R^N\times (0,T))$ norm of $\nabla (-\Delta)^{-\frac{s}{2}}u$. When $m\in [2,3)$ they are still being   useful energy estimates, but an additional decay has to be imposed to $u_0$.  In \cite{StanTesoVazJDE} we proved that if $u_0$ decays exponentially for large $|x|$, then $u(t)$ has a similar decay and  \eqref{FirstEnergy} gives  us meaningful  information.
For $m\ge 3$, \eqref{FirstEnergy} is not valid anymore with a decay property. This has motivated us to use a different approximation technique in the present paper which satisfies a different energy estimate \eqref{Lpdecay} without any additional conditions to be imposed on the initial data.

\noindent $\bullet$ The $L^p$-energy estimate \eqref{Lpdecay} can be proved for a general nonlinearity $\varphi(u)$:
$$
\intr \varphi(u)(x,t)dx + \int_0^t \int_{\RN}\left|(-\Delta)^{\frac{1-s}{2}} \psi(u) \right|^2 dxdt \le \intr \varphi(u_0)(x)dx ,
$$
where $(\psi')^2(a)=\varphi''(a) a^{m-1}$. This kind of energy estimate is used in \cite{BilerImbertKarch} and in \cite{dTEnJa16b}.

\noindent $\bullet$ \textbf{ More general equations and estimates.} The techniques employed in this paper can be used to prove  existence results for more general equations of the form
\begin{equation}\label{GenEq}
u_{t}(x,t) = \nabla \cdot (G'(u) \nabla (-\Delta)^{-s}u),
\end{equation}
where $G:[0,+\infty)\to[0,+\infty) $ has at most linear growth at the origin or $G'>0$.
The general Stroock-Varopoulos Inequality \eqref{StroockVar2} allows us to obtain an energy inequality also in this case:
$$
\intr \varphi(u)(x,t)dx +  \int_0^t \int_{\RN}\left|(-\Delta)^{\frac{1-s}{2}} \psi(u) \right|^2 dxdt \le \intr \varphi(u_0)(x)dx ,
$$
where $(\psi')^2(a)=\varphi''(a) G'(a)$.  We give a few examples below.

a) For instance we consider $G(u)=\frac{1}{m}(u+1)^{m}$, then $G'(u)=(u+1)^{m-1}$ and the model is
\begin{equation}\label{NonDegDiff}
u_{t}(x,t) = \nabla \cdot \left((u+1)^{m-1} \nabla (-\Delta)^{-s}u \right).
\end{equation}
 This corresponds to the approximating problem \eqref{ProblemMuDelta} without viscosity $\mu=1$, $\delta=0$.
There is positive velocity and the solutions seem to have infinite speed of propagation. See Figure \ref{FigNonDegDiff} for the particular case $m=2$.

b) Let $G(u)=\log(1+u)$, then $G'(u)=\frac{1}{1+u}$ and the model is
\begin{equation}\label{LogDiff}
u_{t}(x,t) = \nabla \cdot \left(\frac{1}{1+u} \nabla (-\Delta)^{-s}u \right).
\end{equation}
We provide a numerical simulation in Figure \ref{FigLogDiff}.
This may correspond to $m\to 0$, $m>0$.
This nonlinearity has been considered for the Fractional Porous Medium Equation $u_t + (-\Delta)^s \log(1+u)=0$ in \cite{PQRV3}.

\begin{figure}[ht!]
\centering
\begin{subfigure}{.49\textwidth}
  \centering
\includegraphics[width=\textwidth]{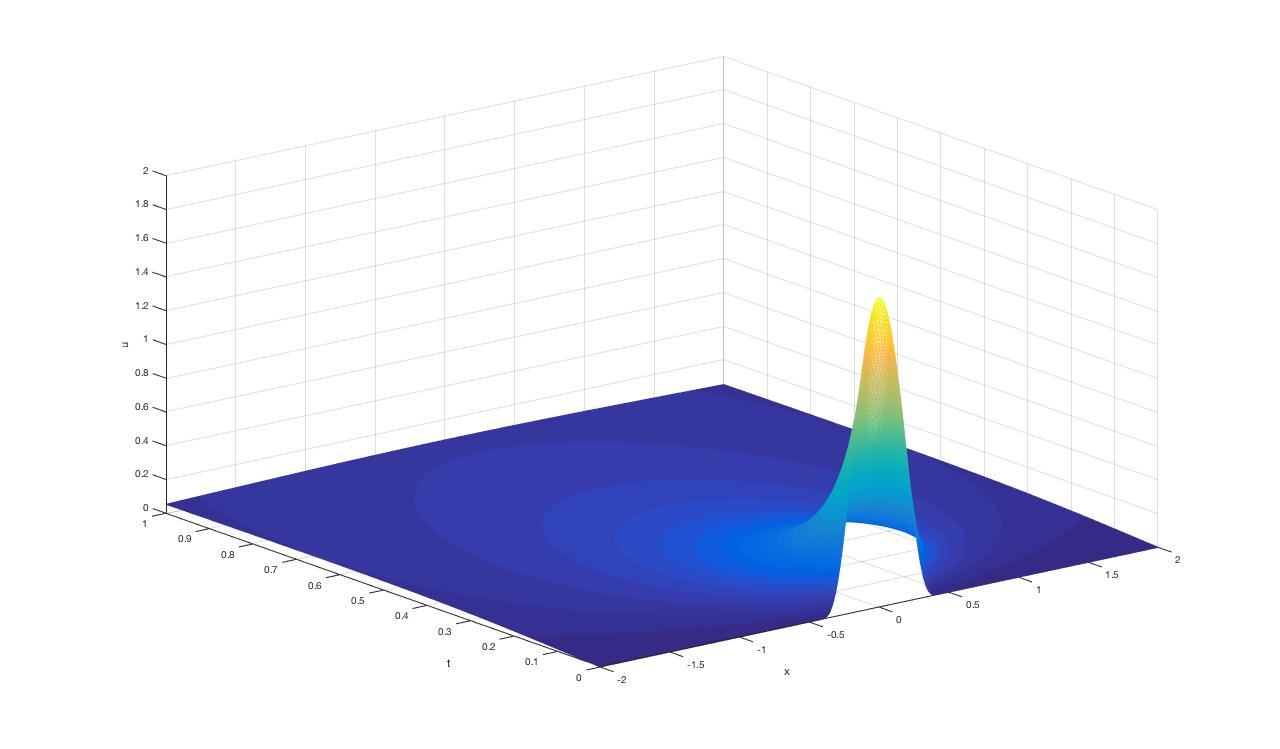}
\caption{Solution for model \eqref{NonDegDiff} with $s=0.5$}\label{FigNonDegDiff}
\end{subfigure}
\begin{subfigure}{.49\textwidth}
  \centering
  \includegraphics[width=\textwidth]{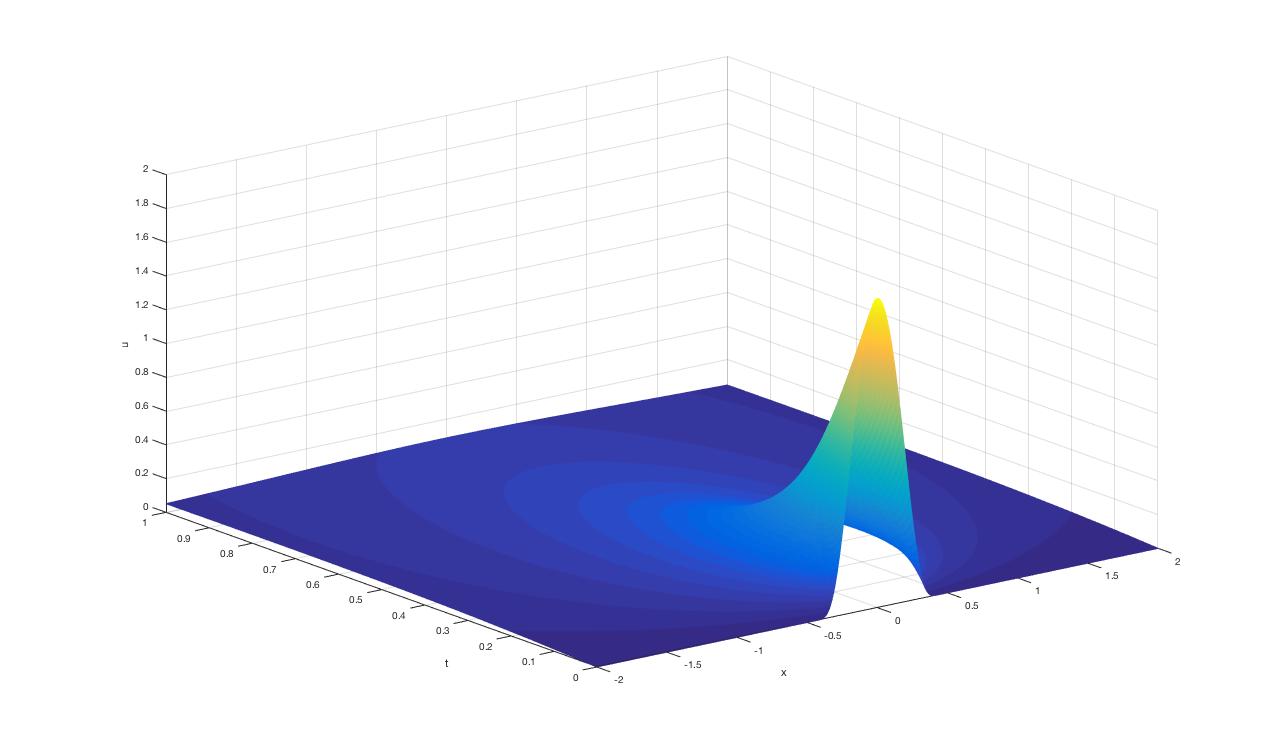}
\caption{Solution for model \eqref{LogDiff} with $s=0.5$}\label{FigLogDiff}
\end{subfigure} \caption{More general equations of type \eqref{GenEq}}   \label{figura2}
\end{figure}

\noindent $\bullet$ \textbf{Finite/infinite speed of propagation depending on the nonlinearity.} In \cite{StanTesoVazJDE}   some preliminary results have  been obtained concerning  the positivity properties of the solution of Problem \eqref{model1}. Jointly with the existence theory developed in the present work for all $1<m<\infty$ we have the following results so far:

a) Let $N\ge 1$, $m \in [2,+\infty)$, $s\in (0,1)$  and let $u$ be a constructed weak solution to Problem \eqref{model1} with compactly supported initial data $u_0 \in L^1(\RN)\cap L^\infty(\RN)$.  Then,  $u(\cdot,t)$ is also compactly supported for any $t>0$, i.e. the solution has finite speed of propagation.
This causes the appearance of free boundaries.

b) Let $N=1$, $m\in (1,2)$, $s\in (0,1)$. Then for any $t>0$ and any $R>0$, the set $\mathcal{M}_{R,t}=\{x: |x|\ge R,\  u(x,t)>0\}$ has positive measure even if $u_0$ is compactly supported. This is a  weak form of  infinite speed of propagation. If moreover $u_0$ is radially symmetric and monotone non-increasing in $|x|$, then we get a clearer result:  $u(x,t)>0$ for all $x\in \R$ and $t>0$.

%%%%%%%%%%%%%%%%%%%%%%%%%%%%%%%%%%%%%%%%%%%%%%%%%%%%%%%%%%%%%%%%

\noindent $\bullet$ \textbf{The effect of the nonlocal operator on the diffusion.} The parameter $s \in (0,1)$ plays a crucial role in the the diffusion effects.

a) In the limit $s\to 1$, we get $u_t=\nabla\cdot (u^{m-1}\nabla (-\Delta)^{-1}u)$, which is no more a diffusion equation. This is an interesting problem to be further investigated. When $m=2$, it has been  proved in  \cite{SerfVaz} that  the model gives in the limit $s\to 1$ a ''mean field'' equation arising in superconductivity and superfluidity. For this equation, the authors obtain uniqueness in the class of bounded solutions, universal bounds and regularity results. To note that H\"older regularity is no more true for the standard  class of bounded integrable solutions.

b) When $s\to 0$ we get  $u_t=\nabla\cdot (u^{m-1} \nabla u)$ which is the classical Porous Medium Equation $u_t=\frac{1}{m}\Delta u^m$ with $m>1$. It is known that solutions propagate with finite speed and have $C^\alpha$ regularity.

 Such limit processes have not been justified with analytical rigor  for $m\ne 2$. We provide some numerical simulations which confirm the behaviour of solutions for different values of $m$ and $s$ (see \cite{delTesoCalcolo,delTesoVaz}).
Figures \ref{m15s025}, \ref{m15s05}, \ref{m15s075} indicate the effect of diffusion in the infinite speed of propagation case.
Figures \ref{m2s025}, \ref{m2s05}, \ref{m2s075} indicate the effect of diffusion in the finite speed of propagation case. Note that the larger the $s$, the slower is the  diffusion velocity.

%%%%%%%%%%%%%%%%%%%%%%%%%%%%%%%%%%%%%%%%%%%%%%%%%%%%%%%%%%%%%%%%%%%%%%%%%%%%%%%%%%%%%%%%%%%%%%%%%

\begin{figure}[ht!]
\centering
\begin{subfigure}{.5\textwidth}
  \centering
\includegraphics[width=\textwidth]{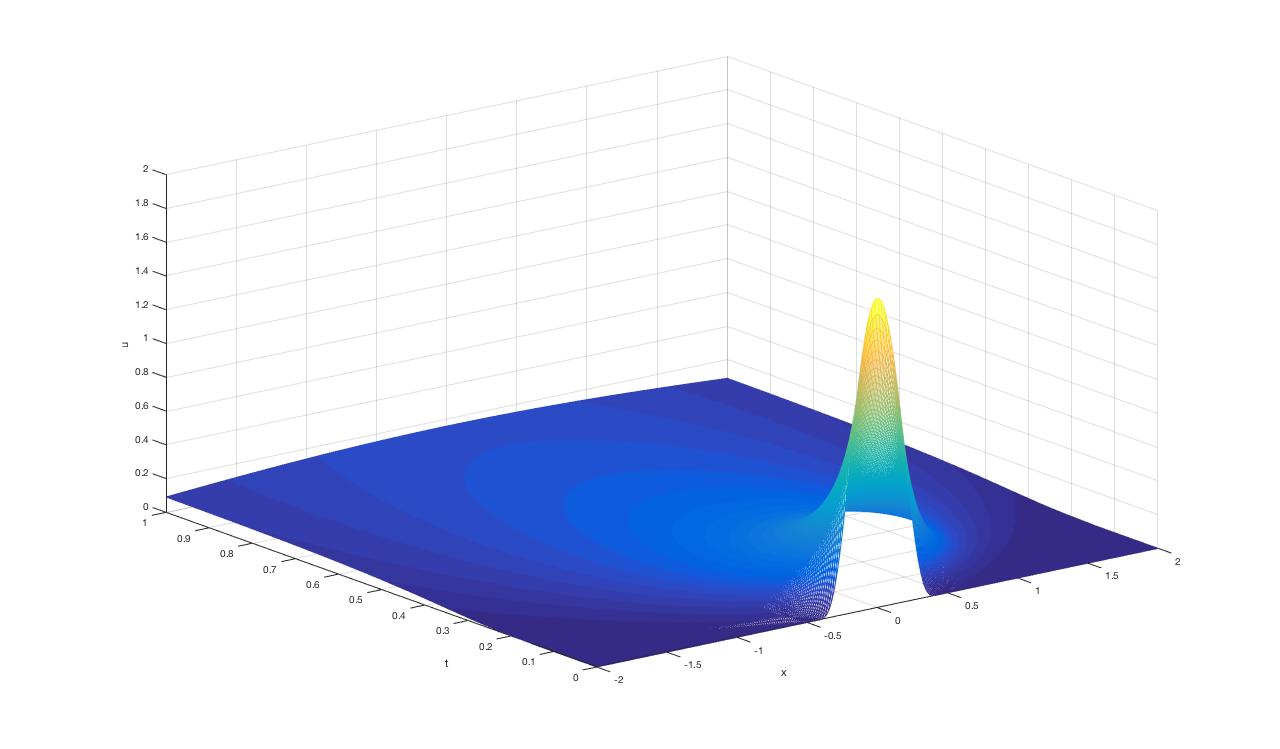}\caption{$m=1.5$, $s=0.25$}\label{m15s025}
\end{subfigure}%
\begin{subfigure}{.5\textwidth}
  \centering
  \includegraphics[width=\textwidth]{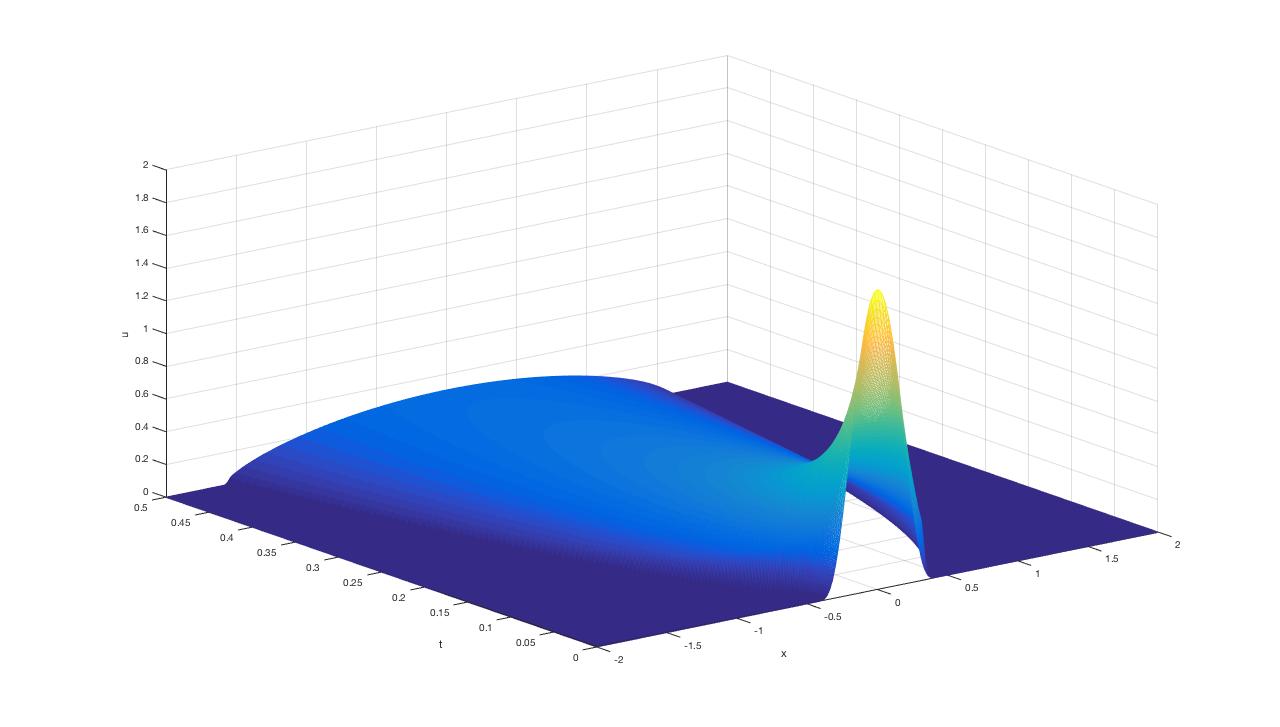}\caption{$m=2$, $s=0.25$}\label{m2s025}
\end{subfigure}

\begin{subfigure}{.5\textwidth}
  \centering
\includegraphics[width=\textwidth]{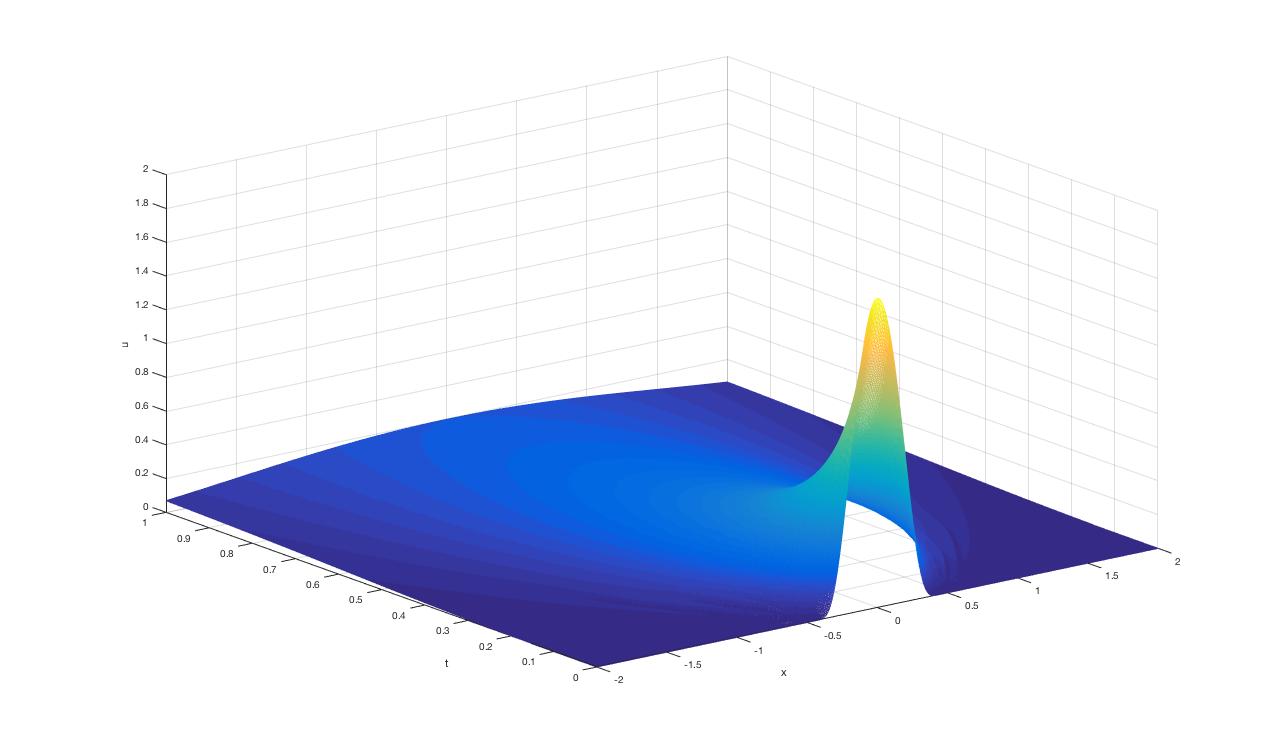}\caption{$m=1.5$, $s=0.5$}\label{m15s05}

\end{subfigure}%
\begin{subfigure}{.5\textwidth}
  \centering
  \includegraphics[width=\textwidth]{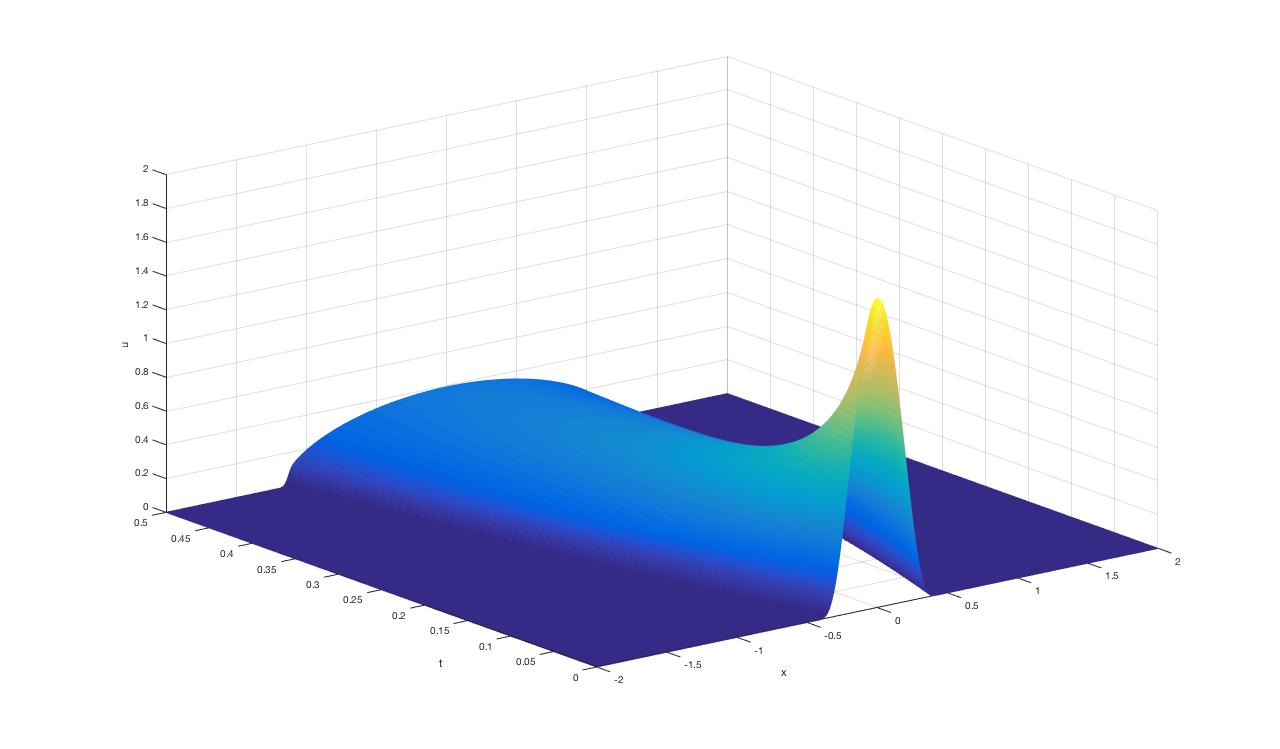}\caption{$m=2$, $s=0.5$}\label{m2s05}
\end{subfigure}

\begin{subfigure}{.5\textwidth}
  \centering
\includegraphics[width=\textwidth]{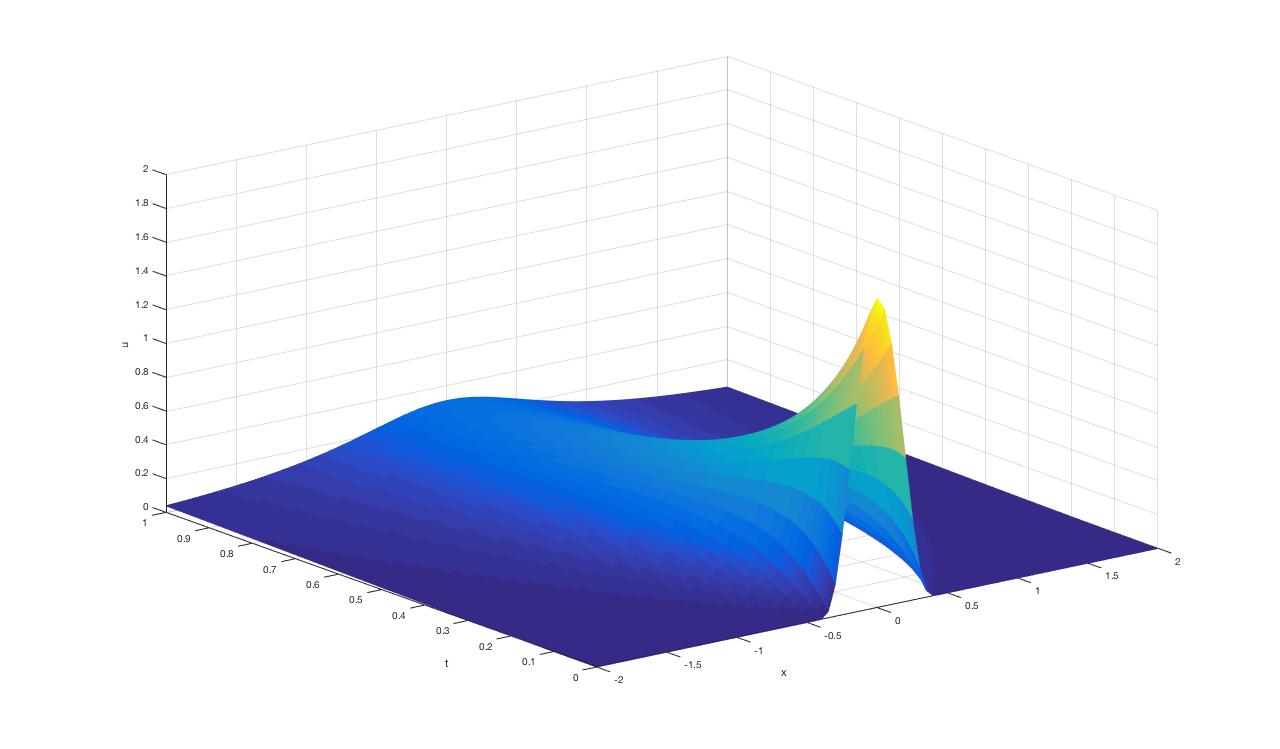}\caption{$m=1.5$, $s=0.75$}\label{m15s075}
\end{subfigure}%
\begin{subfigure}{.5\textwidth}
  \centering
  \includegraphics[width=\textwidth]{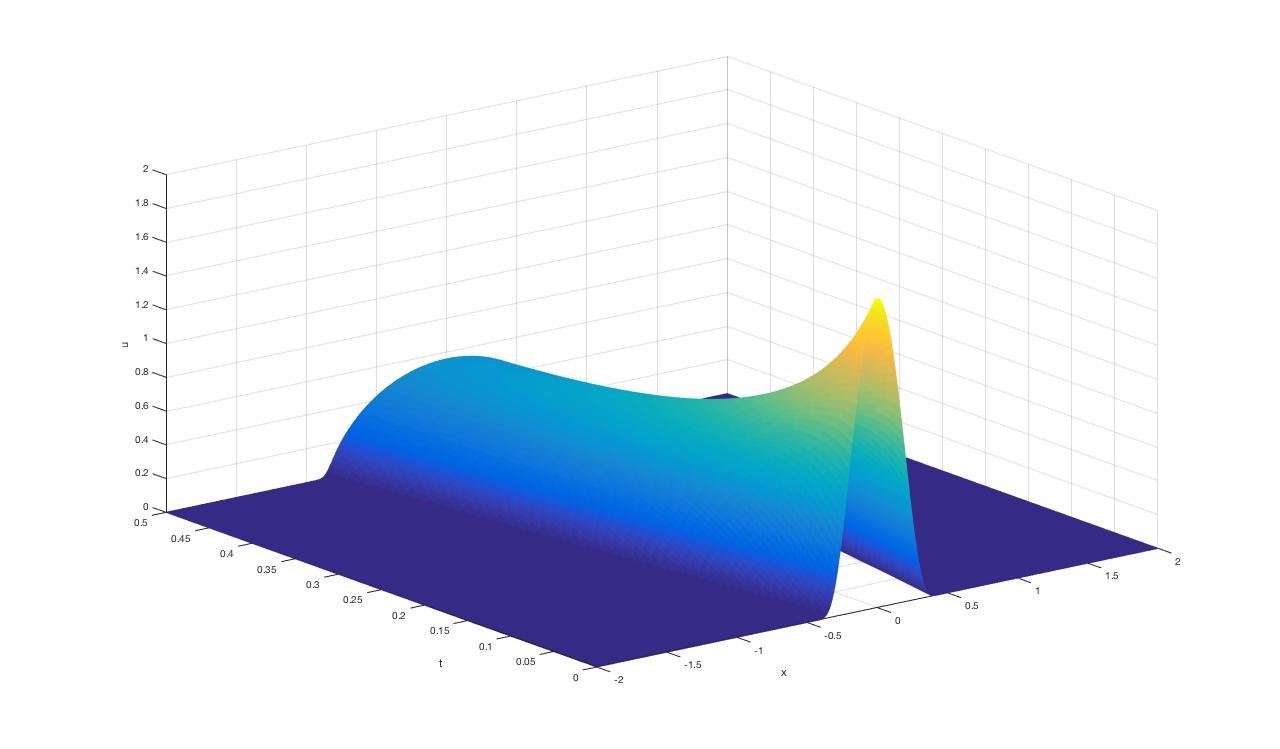}\caption{$m=2$, $s=0.75$}\label{m2s075}
\end{subfigure}
\caption{Infinite vs. finite speed of propagation for different pressures}
\label{figura1}
\end{figure}

\medskip
%%%%%%%%%%%%%%%%%%%%%%%%%%%%%%%%%%%%%%%%%%%%%%%%%%%%%%%%%%%%%%%%

\noindent $\bullet$ \textbf{The question of uniqueness.}

 As mentioned in the introduction there is an open problem about uniqueness in several space dimensions. There are recent uniqueness results if the initial data are smooth, see Zhou et al. \cite{ZXC} that obtain unique local-in-time strong solutions in Besov spaces; thus, for initial data in $B^\alpha_{1,\infty} $  if $1/2 \le s < 1$ and $\alpha > N + 1$ with $N\ge  2.$ See also \cite{XZ2018}. On the other hand, Duerincks \cite{Duer2018} proves uniqueness and stability of solutions having a given regularity, based on previous work by Serfaty in the Coulomb case \cite{Serfaty}. These results need to be extended to our model.

\medskip

\noindent $\bullet$ \textbf{Other open problems.}

\noindent $-$   The problem in a bounded domain with Dirichlet or Neumann data has
not scarcely studied. See Nguyen and V\'azquez   \cite{NgVaz} for  Dirichlet data.

\noindent  $-$ We have considered only nonnegative solutions on physical grounds. But we could have also considered signed solutions after writing the equation as $u_t=\nabla\cdot (|u|^{m-1}\nabla (-\Delta)^{-s}u)$.

\noindent $-$ Good numerical studies are needed.   A rigorous study of convergent numerical schemes is developed in \cite{dTJa18} in dimension $N=1$.

\section{Appendix}\label{sec:app}

\subsection{Functional inequalities related to the fractional Laplacian}

We recall some functional inequalities related to the fractional Laplacian operator that we used throughout the paper. We refer to \cite{PQRV2} for the proofs.

\begin{lemma}[\textbf{Stroock-Varopoulos Inequality}] Let $0<s<1$, $q>1$. Then
$$
\int_{\RN}|v|^{q-2}v (-\Delta)^{s}v dx \ge \frac{4(q-1)}{q^2}\int_{\RN}\left| (-\Delta)^{s/2}|v|^{q/2}\right|^2 dx
$$
for all $v\in L^q(\RN)$ such that $(-\Delta)^{s}v \in L^q(\RN)$.
\end{lemma}

\begin{lemma}[\textbf{Generalized Stroock-Varopoulos Inequality}] Let $0<s<1$. Then
\begin{equation}\label{StroockVar2}
\int_{\RN}\psi(v)(-\Delta)^{s}v dx \ge \int_{\RN}\left| (-\Delta)^{s/2}\Psi(v)\right|^2 dx
\end{equation}
whenever $\psi'=(\Psi')^2$.
\end{lemma}

\begin{theorem}[\textbf{Sobolev Inequality}] Let $0<s<1$ ($s<\frac{1}{2}$ if $N=1$). Then
$$
\|f\|_{\frac{2N}{N-2s}}\le \mathcal{S}_s \left\| (-\Delta)^{s/2}f\right\|_2,
$$
where the best constant is given in \cite{BV2014} page 31.
\end{theorem}

\begin{theorem}[\textbf{Nash-Gagliardo-Nirenberg type inequality}] Let $0<s<1$ ($s<\frac{1}{2}$ if $N=1$), $p\ge 1$, $r>1$, $0<s < \min \{N/2,1\}$. Then there exists a constant $C=C(p,r,s, N)>0$  such that for any $f \in L^p(\RN)$ with $(-\Delta)^{s}f \in L^r(\RN)$ we have
\begin{equation}\label{NGNIneq}
\|f\|_{r_2}^{\alpha +1 }\le C  \left\| (-\Delta)^{s}f\right\|_r \|f\|_{p}^{\alpha},
\end{equation}
where $r_2=\frac{N(rp+r-p)}{r(N-2s)}$, $\alpha =\frac{p(r-1)}{r}$.
\end{theorem}

\subsection{Compactness criteria}

Necessary and sufficient conditions of convergence in the spaces $L^p(0,T:B)$ are given by Simon in \cite{simon}. We recall now their applications to evolution problems. We consider the spaces $X\subset B\subset Y$ with compact embedding $X\subset B$.

\begin{lemma}\label{ConvSimon1}Let $\mathcal{F}$ be a bounded family of functions in $L^p(0,T:X)$, where $1 \leq p <\infty$ and $\partial \mathcal{F}/\partial t=\{\partial f/\partial t: f \in \mathcal{F}\}$ be bounded in  $L^1(0,T:Y)$. Then the family $\mathcal{F}$ is relatively compact in $L^p(0,T:B)$.
\end{lemma}

\noindent We refer to Rakotoson and Temam \cite{RakotosonTemam} for the proof of the following Lemma \ref{RakotosonTemam1} and \ref{RakotosonTemam2}.

\begin{lemma}\label{RakotosonTemam1}
Let $(V,\|\cdot\|_V)$, $(H,\| \cdot  \|_H)$ two separable Hilbert spaces. Assume that $V\subset H$ with a compact and dense embedding. Consider a sequence $(u_\delta)_{\delta>0}$ converging weakly to a function $u$ in $L^2(0,T:V)$, $T<+\infty$. Then $u_\delta \to u$ strongly in $L^2(0,T:H)$ if and only if
\begin{enumerate}
\item[(i)] $u_\delta(t) \rightharpoonup u(t)$ in $H$ for a.e. $t$.
\item[(ii)] $\lim_{meas(E) \to 0 , E\subset [0,T]} sup_{\delta >0} \int_{E}\|u_\delta(t) \|_H^2 dt =0.$
\end{enumerate}
\end{lemma}

\begin{lemma}\label{RakotosonTemam2}
Let $H$ be a separable Hilbert space. Consider $u_\delta$ a sequence of functions satisfying
the following:
\begin{enumerate}
\item[1)] For almost every $ t \subset (0,T)$, $\sup_{\delta>0}\|u_\delta(t)\|_{H}$ is finite.
\item[2)] $u \rightharpoonup u$ in $L^2(0,T:H)$.
\item[3)] There exists a countable set $D$ dense in $H$ such that for all $\psi \in D$, the sequence $g^{\delta}_{\psi} (t) =<u_\delta(t), \psi>_H$ is relatively compact in $L^1(0, T)$.
\end{enumerate}
Then, there exists a subsequence $(\delta) =(\delta_D)$ such that $u^\delta(t)\rightharpoonup u(t)$ in $H$-weak for almost every $t$.
\end{lemma}

Combining both lemmas above  the following optimal compactness theorem holds.

\begin{theorem}\label{RakotosonTemam3}
Let $(V,\|\cdot\|_V)$, $(H,\| \cdot  \|_H)$ two separable Hilbert spaces. Assume that $V\subset H$ with a compact and dense embedding. Consider a sequence $(u_\delta)_{\delta>0}$ such that
\begin{enumerate}
\item[a)]
$u_\delta \rightharpoonup u$ in $L^2(0,T:V)$, $T<+\infty$.
\item[b)] For almost every $ t \in(0,T)$, $\sup_{\delta>0}\|u_\delta(t)\|_{H}$ is finite.
\item[c)] There exists a countable set $D$ dense in $H$ such that for all $\psi \in D$, the sequence \\
$g^{\delta}_{\psi}(t)=<u_\delta(t), \psi>_H$ is relatively compact in $L^1((0, T))$.\end{enumerate}
Then, up to a subsequence,  $u_\delta \to u$ strongly in $L^2(0,T:H)$.
\end{theorem}
\begin{proof}
Weak convergence in $L^2(0,T:V)$ implies weak convergence in $L^2(0,T:H)$, therefore a) implies  assumption 2) in Lemma \ref{RakotosonTemam2}.
By Lemma \ref{RakotosonTemam2}  we obtain that, up to a subsequence,  $u^\delta(t)\rightharpoonup u(t)$ in $H$-weak for almost every $t$.
Moreover, the upper bound given by 1) implies (ii) from
Lemma \ref{RakotosonTemam1}.
Then using Lemma \ref{RakotosonTemam1}  we obtain that  $u_\delta \to u$ strongly in $L^2(0,T:H)$.
\end{proof}

\subsection{A technical result related to the approximation arguments}\label{app:techresult}
Let $U_4$ be as given in Section \ref{Subsec:delta}. We want to show  that $\nabla \cdot  (-\Delta)^{-s} ( U_4^{m-1} \nabla \phi) \in L^p(\RN\times (0,T))$ for some $p>1$.
%Proof of the $L^2$ estimate for $m>2$:
%$$ \nabla \cdot (-\Delta)^{-s}(U_4^{m-1} \nabla \phi) \in L^2(\RN \times (0,T)).$$
We will express the operator $\nabla \cdot (-\Delta)^{-s}$  using  the Riesz transforms applied to the Riesz potential operator or to a fractional operator, depending on the range of $s$.

% Observe that
%\begin{equation}\label{decomp}
% \nabla \cdot (-\Delta)^{-s}(U_4^{m-1} \nabla \phi) = \nabla (-\Delta)^{-s}(U_4^{m-1} ) \cdot \nabla \phi +  (-\Delta)^{-s}(U_4^{m-1}) \Delta\phi.
%\end{equation}
First let $s\in (0, 1/2)$. We have that
$$\nabla \cdot(-\Delta)^{-s}(U_4^{m-1}   \nabla \phi) = \nabla \cdot (-\Delta)^{-1/2} (-\Delta)^{1/2-s}(U_4^{m-1} \nabla \phi)
= \sum_{j=1} ^N \partial_{x_j} (-\Delta)^{-1/2} (-\Delta)^{1/2-s} (U_4^{m-1} \, \partial_{x_j} \phi ),$$
where $\mathcal{R}_j :=\partial_{x_j} (-\Delta)^{-1/2}$ are the Riesz Transforms which are  bounded linear operators from $L^2$ to $L^2$.
Notice that $(-\Delta)^{1/2-s}(U_4^{m-1} \nabla \phi)  \in L^2$ since
$$  (-\Delta)^{1/2-s} (U_4^{m-1} \, \partial_{x_j} \phi ) =  (-\Delta)^{1/2-s} (U_4^{m-1}) \, \partial_{x_j} \phi  +   U_4^{m-1} \, (-\Delta)^{1/2-s} (\partial_{x_j} \phi ) - H^{1/2-s}(U_4^{m-1}, \partial_{x_j} \phi)$$
where $H^s$ is the remaining in the fractional Leibniz formula, also called Carr\'e du Champ operator  \cite[Ch. 1.4.2]{Bakry}   :
$$H^s(f,g)(x):= P.V. \int_{\RN} \frac{(f(x)-f(y))(g(x)-g(y))}{|x-y|^{N+2s}}dy.$$
Note that, by H\"older's Inequality
\begin{align*}
\|H^{1/2-s}(U_4^{m-1}, \partial_{x_j} \phi)\|_{L^2}^2  &\le \int |H^{1/2-s}(U_4^{m-1}, U_4^{m-1})| dx \cdot  \int | H^{1/2-s} (\partial_{x_j} \phi, \partial_{x_j} \phi)| dx \\
&= \| (-\Delta)^{\frac{1/2-s}{2}}(U_4^{m-1}) \|_{L^2}^2 \, \| (-\Delta)^{\frac{1/2-s}{2}}(\partial_{x_j} \phi) \|_{L^2}^2 <\infty.
\end{align*}
Thus, using the energy estimate \eqref{energyU4} with $p=m-1$, we get that
\begin{align*}
\| (-\Delta)^{1/2-s} (U_4^{m-1} \, \partial_{x_j} \phi ) \|_{L^2} \le& \| (-\Delta)^{1/2-s} (U_4^{m-1})\|_{L^2} \, \|\partial_{x_j} \phi  \| _{L^2} + \| U_4^{m-1} \|_{L^2}  \, \|(-\Delta)^{1/2-s} (\partial_{x_j} \phi ) \|_{L^2} \\
& + \|H^{1/2-s}(U_4^{m-1}, \partial_{x_j} \phi)\|_ {L^2}  <\infty.
\end{align*}
We have used the energy estimate (4.24) for $U_4$ with $p=m-1$. We then obtain that \[
\nabla \cdot (-\Delta)^{-s}(U_4^{m-1} \nabla \phi) \in L^2(\RN \times (0,T))\] since
\begin{align*}
\| \nabla\cdot (-\Delta)^{-s}(U_4^{m-1} \cdot  \nabla \phi)  \|_{L^2} &=\left\| \sum_{j=1} ^N \partial_{x_j} (-\Delta)^{-1/2} (-\Delta)^{1/2-s} (U_4^{m-1} \, \partial_{x_j} \phi ) \right\|_{L^2} \\
& \le  \sum_{j=1} ^N  \|\mathcal{R}_j (-\Delta)^{1/2-s} (U_4^{m-1} \, \partial_{x_j} \phi )  \|_{L^2} \\
& \le \sum_{j=1}^N\|(-\Delta)^{1/2-s}(U_4^{m-1} \, \partial_{x_j} \phi)\|_{L^2}<\infty.
\end{align*}

\medskip
Consider now $s\in [1/2,1)$. We interpret the term as follows:
$$\nabla \cdot (-\Delta)^{-s}(U_4^{m-1} \nabla \phi) = \nabla (-\Delta)^{-1/2} (-\Delta)^{-(s-1/2)}(U_4^{m-1} \nabla \phi) $$
where the Riesz vector transform $\mathcal{R}_j :=\partial_{x_j} (-\Delta)^{-1/2}$ is a bounded operator in $L^p$ for $1<p<\infty$  \cite[Cor. 4.2.8, pp. 274]{Grafakos}.
Then $U_4^{m-1} \nabla \phi \in L^p(\RN \times (0,T))$ for every $p>1$ since
\begin{align*}
\| U_4^{m-1} \nabla \phi \|_{L^p} = \left( \int_{\RN}U_4^{p(m-1)}|\nabla \phi|^p dx \right)^{1/p} \le \|U_4\|_{\infty}^{\frac{p(m-2)}{p}} \left( \int_{\RN}U_4^{p}|\nabla \phi|^p dx \right)^{1/p} <\infty.
\end{align*}
It follows that for $s\in [1/2,1)$, the operator $ (-\Delta)^{-(s- 1/2)}=I_ {2s-1}$  is the Riesz potential, and we have
$$ \| (-\Delta)^{-(s-1/2)}(U_4^{m-1} \nabla \phi)  \|_{L^q} \le \| U_4^{m-1} \nabla \phi \|_{L^p}<\infty,
,\quad \frac{1}{q}=\frac{1}{p}-\frac{2s-1}{N},$$
for all $p<N/(2s-1)$. Since $N/(2s-1)>1$ for all $s\in[1/2,1)$, this shows that $\nabla \cdot (-\Delta)^{-s}(U_4^{m-1} \nabla \phi) \in L^q( \R^N\times(0,T))$ for some $q>1$.

\bigskip

\noindent {\bf \large Acknowledgements}
%\color{red}Update your acknowledgements.

\noindent  Authors partially supported by the Spanish Project MTM2014-52240-P.    D.S. and F.d.T are partially supported by the MEC-Juan de la Cierva postdoctoral fellowships number FJCI-2015-25797 and FJCI-2016-30148 respectively, and by the BCAM Severo Ochoa accreditation SEV-2017-0718.
F.d.T. is partially supported by the Toppforsk (research excellence) project Waves and Nonlinear Phenomena (WaNP), grant 250070 from the Research Council of Norway.
  J.L.V. has been a Visiting Professor at Univ. Complutense de Madrid during the academic year 2017-2018. The authors want to thank the anonymous referee for accurate suggestions that allowed to improve the original text.

%%%%%%%%%%%%%%%%%%%%%%%%%%%%%%%%%%%%%%%%%%%%%%%%%%%%%%%%%%%%%%%%%%%%%%%%%%%%%%
%

\bibliographystyle{siam}

\begin{thebibliography}{10}

\bibitem{AGS2008}
{\sc  L. Ambrosio, N.  Gigli, G. Savar\`e}. {\em ``Gradient flows in metric spaces and in the space of probability measures''}, Second edition. Lectures in Mathematics ETH Z\"urich. Birkh\"auser Verlag, Basel, 2008.


\bibitem{AMRT}
{\sc F. Andreu-Vaillo, J. M. Mazon, J. D. Rossi,  J. J. Toledo-Melero}. {\em ``Nonlocal diffusion problems},
Mathematical Surveys and Monographs {\bf 65} (American Mathematical Society, Providence, RI, 2010).


\bibitem{Bakry}{\sc D. Bakry, I. Gentil, M. Ledoux.} {\em
``Analysis and geometry of Markov diffusion operators'',}
Grundlehren der Mathematischen Wissenschaften [Fundamental Principles of Mathematical Sciences], 348. Springer, Cham, 2014. xx+552 pp.


\bibitem{BilerImbertKarch}
{\sc P.~Biler, C.~Imbert,  G.~Karch}. {\sl The nonlocal porous
  medium equation: {B}arenblatt profiles and other weak solutions}, Arch.
  Ration. Mech. Anal., \textbf{215} (2015),  497--529.

\bibitem{BilerKarchMonneau}
{\sc P.~Biler, G.~Karch, R.~Monneau}. {\em Nonlinear diffusion of
  dislocation density and self-similar solutions}, Comm. Math. Phys., \textbf{294}
  (2010),  145--168.


\bibitem{BFV18} {\sc M. Bonforte,  A. Figalli, J.~L. V{\'a}zquez}. {\sl
Sharp global estimates for local and nonlocal porous medium-type equations in bounded domains}, Analysis of PDEs {\bf 11} (2018), no. 4, 945--982.


\bibitem{BSV16} {\sc M. Bonforte, Y. Sire, J.~L. V{\'a}zquez}. {\sl Optimal Existence and Uniqueness Theory for the Fractional Heat Equation}, Nonlinear Analysis,
  \textbf{153} (2017), 142--168.
  %  ArXiv:1606.00873v1.


\bibitem{BV2014}
{\sc M.~Bonforte, J.~V{\'a}zquez}. {\em Quantitative local and global a
  priori estimates for fractional nonlinear diffusion equations}, Adv. Math.,
  \textbf{250} (2014),  242--284.


\bibitem{CaffSilvExt}
{\sc L.~Caffarelli, L.~Silvestre}. {\em An extension problem related to the fractional {L}aplacian.}
 Comm. Partial Differential Equations, \textbf{32} (2007), no.7-9:1245--1260.

\bibitem{CaffSorVaz}
{\sc L.~Caffarelli, F.~Soria, J.~L. V{\'a}zquez}, {\em Regularity of
  solutions of the fractional porous medium flow}, J. Eur. Math. Soc. (JEMS),
  \textbf{15} (2013),  1701--1746.

\bibitem{CaffVazReg}
{\sc L.~Caffarelli,  J.~V\'azquez}. {\em Regularity of solutions of the
  fractional porous medium flow with exponent 1/2},  Algebra i Analiz  [St. Petersburg Mathematical Journal], {\bf 27} (2015),  no. 3, 125--156; translation in
St. Petersburg Math. J. {\bf  27} (2016), no. 3, 437--460.

\bibitem{CaffVaz}
{\sc L.~Caffarelli, J.~L. Vazquez}. {\em Nonlinear porous medium flow with
  fractional potential pressure}, Arch. Ration. Mech. Anal., \textbf{202} (2011),
   537--565.

\bibitem{CaffVaz2}
{\sc L.~A. Caffarelli, J.~L. V{\'a}zquez}. {\em Asymptotic behaviour of a
  porous medium equation with fractional diffusion}, Discrete Contin. Dyn.
  Syst., \textbf{29} (2011),  1393--1404.

\bibitem{CarrilloHuangVazquez}
{\sc J.~A. Carrillo, Y.~Huang, M.~C. Santos,  J.~L. V{\'a}zquez}. {\em
  Exponential convergence towards stationary states for the 1{D} porous medium
  equation with fractional pressure}, J. Differential Equations, \textbf{258} (2015),
   736--763.

\bibitem{PQRV1}
{\sc A.~de~Pablo, F.~Quir{\'o}s, A.~Rodr{\'{\i}}guez,  J.~V{\'a}zquez}. {\em
  A fractional porous medium equation}, Adv. Math., \textbf{226} (2011),  1378--1409.

\bibitem{PQRV2}
{\sc A.~de~Pablo, F.~Quir{\'o}s, A.~Rodr\'iguez,  J.~V{\'a}zquez}. {\em A
  general fractional porous medium equation}, Comm. Pure Appl. Math. \textbf{65}
  (2012),  1242--1284.

\bibitem{PQRV3}{\sc A.~de~Pablo, F.~Quir{\'o}s, A.~Rodr\'iguez,  J.~V{\'a}zquez}. {\em
Classical solutions for a logarithmic fractional diffusion equation}, J. Math. Pures Appl. (9) \textbf{101} (2014), no. 6, 901--924.


\bibitem{delTesoCalcolo}
{\sc F.~del Teso}. {\em Finite difference method for a fractional porous medium
  equation}, Calcolo, \textbf{51} (2014),  615--638.



\bibitem{EndalJakobsenTeso}{\sc F.~del Teso, J.~Endal, E.~R.~Jakobsen}. {\em Uniqueness and properties of distributional solutions of nonlocal equations of porous medium type.}
Adv. Math. {\bf 305} (2017), 78--143.

 \bibitem{dTEnJa16b}{\sc F.~del Teso, J.~Endal, E.~R.~Jakobsen}. {\em On the well-posedness of solutions with finite energy for nonlocal equations of porous medium type.}  EMS Series of Congress Reports: Non-Linear Partial Differential Equations, Mathematical Physics, and Stochastic Analysis  (2018)  129-167.


  \bibitem{dTJa18}
  {\sc F.~del Teso, E. R. Jakobsen}. {\em A convergent numerical method for the porous medium equation with fractional pressure}, In preparation.



  \bibitem{delTesoVaz}
{\sc F.~del Teso, J.~L. V\'azquez}. {\em Finite difference method for a
  general fractional porous medium equation}, {\tt arXiv:1307.2474}.  (2013).


   \bibitem{Hitch2012}
{\sc E.~Di~Nezza, G.~Palatucci,  E.~Valdinoci}. {\em Hitchhiker's guide to
  the fractional {S}obolev spaces}, Bull. Sci. Math., \textbf{136} (2012),  521--573.

   \bibitem{DoZh17}
{\sc J.~ Dolbeault, A.~ Zhang}. {\em Flows and functional inequalities for fractional operators}
Appl. Anal. \textbf{96} (2018) 1547--1560

   \bibitem{Duer2018} {\sc M. Duerinckx}. {\em Mean-field limits for some Riesz interaction gradient flows.}
SIAM J. Math. Anal. \textbf{48} (2016), no. 3, 2269--2300.


\bibitem{GiacLeb1} {\sc G. Giacomin, J. L. Lebowitz.} {\em Phase segregation dynamics in particle systems with long range interaction I. Macroscopic limits,} J. Statist. Phys. {\bf 87}, (1997), no. 1-2, 37--61.

\bibitem{GiacLeb2} {\sc G. Giacomin, J. L. Lebowitz.} {\em Phase segregation dynamics in particle systems with long range interactions. II. Interface motion}, SIAM J. Appl. Math. {\bf 58} (1998), no. 6, 1707--1729.


\bibitem{Grafakos}{\sc L. Grafakos}, {\em ``Classical Fourier Analysis''},
Second edition. Graduate Texts in Mathematics, 249. Springer, New York, 2008.

\bibitem{IgRo09}
{\sc L. I. Ignat, J. D. Rossi}. {\em Decay estimates for nonlocal problems via energy methods.} J. Math. Pures Appl. (9) \textbf{92} (2009), no. 2, 163--187.

\bibitem{Imb16}
{\sc C. Imbert}.
{\em Finite speed of propagation for a non-local porous medium equation.}
Colloq. Math. \textbf{143} (2016), no. 2, 149--157.


\bibitem{Ladyz}
{\sc O.A. Lady\v{z}enskaja, V.A. Solonnikov, N.N Ural'ceva. }
{\em ``Linear and quasilinear equations of parabolic type.} (Russian)
Translated from the Russian by S. Smith. Translations of Mathematical Monographs, Vol. 23 American Mathematical Society, Providence, R.I. 1968 xi+648 pp.

\bibitem{LiskevichSemPAMS} {\sc V. A. Liskevich,  Yu. A. Semenov.} {\em
Some inequalities for sub-Markovian generators and their applications to the perturbation theory''},
Proc. Amer. Math. Soc. \textbf{119} (1993), no. 4, 1171--1177.

\bibitem{Lisini} {\sc S. Lisini, E. Mainini, A. Segatti.}
{\em A gradient flow approach to the porous medium equation with fractional pressure},
Arch. Ration. Mech. Anal. \textbf{227} (2018), no. 2, 567--606.


\bibitem{NgVaz} {\sc  Q.-H. Nguyen, J.~V{\'a}zquez}. {\em Porous medium equation with nonlocal pressure in a bounded domain,} Comm. PDEs (available online at https://doi.org/10.1080/03605302.2018.1475492).


\bibitem{Pazy}
{\sc A. Pazy.} {\em ``Semigroups of linear operators and applications to partial differential equations''. }
Applied Mathematical Sciences, 44. Springer-Verlag, New York, 1983. viii+279 pp. ISBN: 0-387-90845-5


\bibitem{RakotosonTemam} {\sc J. M. Rakotoson, R. Temam}. {\em An optimal compactness theorem and application to elliptic-parabolic systems.} Appl. Math. Lett. \textbf{14} (2001), no. 3, 303--306.


\bibitem{Ros08}
     {\sc J. D. Rossi}.
      {\em Approximations of local evolution problems by nonlocal ones},
    {Bol. Soc. Esp. Mat. Apl. S$\vec{\rm e}$MA},
      \textbf{42},       {(2008)},       {49--65},


\bibitem{Serfaty}{\sc  S. Serfaty}.  {\em Mean-Field Limits of the Gross-Pitaevskii and Parabolic Ginzburg-Landau Equations}, J. Amer. Math. Soc. {\bf 30} (2017), no. 3, 713--768.


\bibitem{SerfVaz}{\sc S. Serfaty,  J.~L. V{\'a}zquez}. {\em A mean field equation as limit of nonlinear diffusions with fractional Laplacian operators.} Calc. Var. Partial Differential Equations 49 (2014), no. 3-4, 1091--1120.

\bibitem{simon}
{\sc J.~Simon}. {\em Compact sets in the space ${L}^p(0,{T};{B})$}. Ann. Mat.
  Pura Appl., \textbf{146} (1987),  65--96.

  \bibitem{StanTesoVazCRAS}
{\sc D.~Stan, F.~del Teso, J.~L. V{\'a}zquez}. {\em Finite and infinite
  speed of propagation for porous medium equations with fractional pressure},
  C. R. Math. Acad. Sci. Paris, \textbf{352} (2014),  123--128.

\bibitem{StanTesoVazTrans}
{\sc D.~Stan, F.~del. Teso,  J.~L. V{\'a}zquez}. {\em Transformations of
  self-similar solutions for porous medium equations of fractional type},
  Nonlinear Anal., \textbf{119} (2015),  62--73.

  \bibitem{StanTesoVazJDE}
{\sc D.~Stan, F.~del Teso,  J.~L. V{\'a}zquez}. {\em Finite and infinite
  speed of propagation for porous medium equations with nonlocal pressure},
  Journal of Differential Equations \textbf{260}, 2 (2016), 1154--1199.

 \bibitem{StanTesoVazSpringer}
{\sc D.~Stan, F.~del Teso,  J.~L. V{\'a}zquez}. {\em Porous medium equation with nonlocal pressure}, Current Research in Nonlinear Analysis, 277--308, Springer Optim. Appl., \textbf{135} (2018).
% ISBN 978-3-319-89800-1.


\bibitem{Stein70} {\sc E. Stein. } {\em  ``Singular Integrals and Differentiability Properties of Functions''}, Princeton University Press, Princeton, 1970.

\bibitem{StroockLargeDev} {\sc  D. W.  Stroock.}  {\em ``An introduction to the theory of large deviations''},
Universitext. Springer-Verlag, New York, 1984. vii+196 pp.


\bibitem{VazBook}{\sc  J.~L. V{\'a}zquez. } {\sl ``The Porous Medium Equation. Mathematical Theory''}, Oxford Mathematical Monographs, Oxford University Press, Oxford, 2007.


\bibitem{Vazquez2014} {\sc J.~L. V{\'a}zquez}. {\em Recent progress in the theory of nonlinear diffusion with fractional Laplacian operators}, Discrete Contin. Dyn. Syst. Ser. S {\bf 7} (2014), no. 4, 857-885.

\bibitem{VazquezBarenblattFPME}
{\sc J.~L. V{\'a}zquez}. {\em Barenblatt solutions and asymptotic behaviour for
  a nonlinear fractional heat equation of porous medium type}, J. Eur. Math.
  Soc. (JEMS), \textbf{16} (2014),  769--803.



\bibitem{VazCIME} {\sc J.~L. V{\'a}zquez}. {\em
 The mathematical theories of diffusion: nonlinear and fractional   diffusion},
In {\em ``Nonlocal and nonlinear diffusions and interactions: new
  methods and directions''}, volume 2186 of  Lecture Notes in Math., pages
  205--278. Springer, Cham, 2017.

\bibitem{XZ2018} {\sc W. Xiao, X. Zhou.} {\em Well-Posedness of a Porous Medium Flow with Fractional Pressure in Sobolev Spaces}, Electron. J. Differential Equations 2017, No. 238, 7 pp.


\bibitem{ZXC} {\sc X. Zhou, W. Xiao, J. Chen}. {\em Fractional porous medium and mean field equations in Besov spaces}, Electron. J. Differential Equations 2014, No. 199, 14 pp.



\end{thebibliography}

\end{document}